%% file: main.tex
\documentclass[11pt]{amsart}

\input{preamble.tex}

\begin{document}

\title[Singular Legendrian knots and the surgery formula in two dimensions]{Contact homology computations for singular Legendrian knots and the surgery formula in two dimensions}

\author{Martin Bäcke}
\address{Department of mathematics, Uppsala University, Box 480, 751 06
Uppsala, Sweden}
\email{martin.backe@math.uu.se}

\maketitle 

\begin{abstract}	
	The Chekanov--Eliashberg dg-algebra is an algebraic invariant of
	Legendrian submanifolds of contact manifolds, whose definition recently
	has been extended to singular Legendrians. We describe a way of
	constructing simpler models of this dg-algebra for singular Legendrian
	knots in $\R^{3}$, and give the first examples of singular knots for
	which the full cohomology can be computed. This includes the
	$A_{n}$-Legendrians. An
	important question in the study of the Chekanov--Eliashberg dg-algebra
	is to understand how its quasi-isomorphism type changes when the
	Legendrian undergoes Weinstein ribbon isotopy. By explicit computation
	we show that quite dramatic changes are possible
	and that among other things, Weinstein ribbon isotopic Legendrians can
	have Koszul dual dg-algebras.
	Along the way, we compute the cohomology of the
	Chekanov--Eliashberg dg-algebra in the boundary of a Weinstein surface.  
	This finishes the
	proof of the Bourgeois--Ekholm--Eliashberg surgery formula in dimension
	two, which was missing from the literature.
\end{abstract}

\setcounter{tocdepth}{2}
\tableofcontents

\input{introduction.tex}
\input{smoothcealgebras.tex}
\input{singcealgebras.tex}
\input{stoppedsubalg.tex}
\input{examples.tex}

\bibliographystyle{alpha} 
\bibliography{refs}
\end{document}

%% file: preamble.tex
\usepackage{import} 
\usepackage{pdfpages} 
\usepackage{transparent} 
\usepackage{lmodern} 
\usepackage[utf8]{inputenc}
\usepackage[english]{babel} 
\usepackage{amsmath}
\usepackage{tikz} 
\usepackage{tikz-cd} 
\usepackage{amssymb} 
\usepackage{amsthm}
\usepackage{mathrsfs}
\usepackage{enumitem}
\usepackage{thmtools}
\usepackage{quiver} 
\usepackage{extarrows}
\usepackage{float}
\usepackage{mathscinet}
\usetikzlibrary{babel} 
\usepackage[pdftex,pdfpagelabels,bookmarks,hyperindex,hyperfigures]{hyperref}
\usepackage[margin=2.5cm, marginparwidth=2cm]{geometry}
\usepackage{pdfpages}
\usepackage{bm}
\usepackage{import}
\usepackage{pdfpages}
\usepackage{transparent}
\usepackage{xcolor}
\usepackage{mathscinet}

\newcommand{\sbm}[1]{{\let\amp=&\left[\begin{smallmatrix}#1\end{smallmatrix}\right]}}
\newcommand\isomto{\stackrel{\sim}{\smash{\longrightarrow}\rule{0pt}{0.4ex}}}

\newcommand{\N}{\mathbb{N}}
\newcommand{\Z}{\mathbb{Z}}

\newcommand{\R}{\mathbb{R}}
\newcommand{\C}{\mathbb{C}}

\newcommand{\bbc}{\textbf{c}}

\newtheorem{theorem}{Theorem}[section] 
\newtheorem{proposition}[theorem]{Proposition}

\newtheorem{lemma}[theorem]{Lemma}
\newtheorem{corollary}[theorem]{Corollary}

\theoremstyle{definition}

\newtheorem{definition}[theorem]{Definition} 
\newtheorem{example}[theorem]{Example}

\theoremstyle{remark}

\newtheorem{remark}[theorem]{Remark}

\graphicspath{ {./figures/} }

\addto\extrasenglish{ 

 }

%% file: introduction.tex
\section{Introduction}
The Chekanov--Eliashberg dg-algebra is an invariant of Legendrian submanifolds
of contact manifolds. In recent years, its definition has been extended to
singular Legendrian skeleta of Weinstein hypersurfaces. This was first done
combinatorially by An--Bae in \cite{AB20} for singular Legendrians in $\R^{3}$,
and given a geometric-analytic generalization by Asplund--Ekholm in \cite{AE21}
to skeleta of Weinstein hypersurfaces in the boundary of a general Weinstein
manifold.  For a singular Legendrian $\Lambda$, which is the skeleton of some
Weinstein hypersurface $V$, living in the ideal contact boundary of some
Weinstein manifold $W$, we denote this algebra by $CE^{*}( \Lambda;V_{0};W)$,
where $V_{0} \subset W$ is a subcritical Weinstein subdomain of $V$ that
contains all subcritical handles for some choice of handle decomposition of
$W$. Through the surgery formula
\cite[Theorem 1.1]{AE21}, this algebra is quasi-equivalent to the endomorphism
algebra of the co-cores of the Weinstein sector obtained by stopping the manifold at
the hypersurface $V$. When $\Lambda$ is smooth, $CE^{*}( \Lambda;V_{0};W )$ is
quasi-isomorphic to the Chekanov--Eliashberg dg-algebra with based loop space
coefficients \cite[Theorem 1.2]{AE21}. For a connected smooth Legendrian knot
in $\R^{3}$, this dg-algebra is also
quasi-isomorphic to the more classical Chekanov--Eliashberg dg-algebra with
non-central coefficients in $\Z[t,t^{-1}]$. 

In this paper, we describe a way of constructing simpler, finitely generated
models of the Chekanov--Eliashberg dg-algebra for singular Legendrians in
$\R^{3}$, and apply these results to compute the cohomology of the dg-algebra in a
number of examples. These include the infinite families of the $\theta$-Legendrian and its
generalizations, and the standard $A_{n}$-Legendrians. There is also an
infinite class of singular Legendrians for which this method produces finite
dimensional models. These are the first
examples of singular Legendrians for which the full
cohomology can be computed. 

One important feature of the Chekanov--Eliashberg dg-algebra is that while it
is invariant up to quasi-isomorphism of under Legendrian isotopy, the
quasi-isomorphism type is not invariant under Weinstein isotopy, i.e. isotopy
of $V$ through Weinstein hypersurfaces, under which the skeleton can undergo
deformations, see \autoref{ssec:isotopy}. Through the
surgery formula \cite[Theorem 1.1]{AE21}, this corresponds to different
choices of generators for the partially wrapped Fukaya category. We give
several explicit examples of how invariance fails under Weinstein isotopy, and
show among other things that each $A_{n}$-Legendrian is Weinstein isotopic to
a Legendrian with Koszul dual cohomology, see \autoref{res:an-isotopy}.

The algebra $CE^{*}( \Lambda;V_{0};W )$ contains the Chekanov--Eliashberg
dg-algebra $CE^{*}( \partial \Lambda;V_{0} )$ of the attaching spheres $\partial
\Lambda \subset \partial V_{0}$ as a unital dg-subalgebra. We compute the cohomology
and minimal model of this subalgebra when $\dim V=2$, where it is also known as
the internal algebra of Ekholm--Ng. This finishes the proof of the
Bourgeois--Ekholm--Eliashberg surgery formula in dimension two, which
previously was not considered in the literature. In contrast to the situation
in higher dimensions, the surgery map for surfaces fails to be an
$A_{\infty}$-isomorphism, while it still is a quasi-isomorphism.

\subsection*{Outline}
In the remaining part of Section 1, we state our main results. Section 2 and
Section 3 are mainly a review of basic definitions and results about
Chekanov--Eliashberg dg-algebras and singular Legendrians. The exception is
Section 2.2, which deals with the surgery formula for Weinstein surfaces. In
Section 4, we prove our main results.  Finally, in Section 5 we give examples
of how one can apply these results to compute the cohomology of the
Chekanov--Eliashberg dg-algebra. 

\subsection{Chekanov--Eliashberg dg-algebras and the surgery formula
for Weinstein surfaces}

Let $V$ be a compact Weinstein surface and let $\partial \Lambda \subset
\partial V$ be an embedding which is a $0$-dimensional Legendrian, partitioned
into two-point spheres and one-point stops. 
(We write $\partial \Lambda$ since this subset naturally is the boundary set of
a choice of skeleton.)
The Reeb chords are
counter-clockwise arcs with endpoints on $\partial \Lambda$.  We fix a base
point in $\partial V \setminus \partial \Lambda$ for each component of
$\partial V$ and denote the chords of $\partial \Lambda$ by $c_{ij}^{p}$, where
$i$ is the starting
point, $j$ is the endpoint, and $p$ is the number of times the chord passes
through the base point. See \autoref{fig:internal-algebra}.  
\begin{figure}[!htb]
    \centering
    
    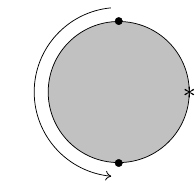

    \caption{The Chekanov--Eliashberg dg-algebra in the boundary of a disk.}
    \label{fig:internal-algebra}
\end{figure}
The Chekanov--Eliashberg dg-algebra $CE^{*}( \partial \Lambda;V )$ is then an
infinitely generated semi-free unital dg-algebra over the ground field
$\bf{k}$, whose underlying associate algebra is the path algebra of the quiver
with one vertex for each sphere and stop contained in $\partial \Lambda$, and the Reeb
chords as arrows. The results in this paper can be shown to hold over any field
$\bf{k}$, but in order to avoid worrying about signs, our computations will be
carried out in the case $\textbf{k}=\Z_2$.  However, the results remain true in
all characteristics. The differential is given by $\partial = \partial_{0} +
\partial_{-1}$, where
\[
	\partial_{0}( c_{ij}^{p} )  = \sum \pm c_{kj}^{l}c_{ik}^{p-l}
\]
with the sum taken over all $k$ and $l$ for which the chords on the right hand
side exist, and $\partial_{-1}( c_{ij}^{p} ) = e_{k}$ if $c_{ij}^{p}$ is a
chord with $i=j$, $p=1$, and which lives in the boundary of a disk, where
$e_{k}$ is the idempotent corresponding to the sphere or stop containing $i$,
and $\partial_{-1}(c_{ij}^{p} )=0$ otherwise. The differential comes from a
count of pseudo-holomorphic curves as usual, see \autoref{res:zero-dim-diff}.

We attach Weinstein handles at the spheres and half-handles at the stops to
obtain a Weinstein sector $V_{\partial \Lambda}$. Let $C \subset V_{\partial
\Lambda}$ be the co-cores of the handles and half-handles, and let $CW^{*}(
C;V_{\partial \Lambda} )$ be the partially wrapped Floer cohomology of $C$. The
following theorem combines the surgery formula for smooth Legendrians by
Bourgeois--Ekholm--Eliashberg and the surgery formula for singular Legendrians
by Asplund--Ekholm.  In this paper we prove the case when $\dim V_{\partial
\Lambda}=2$, which previously was not treated.  
\begin{theorem}[\autoref{res:surgery-map-surface}]
	There is an $A_{\infty}$-quasi-isomorphism
	\[
		CW^{*}( C;V_{\partial \Lambda} ) \isomto CE^{*}( \partial \Lambda;V ).
	\] 
\end{theorem}
A proof of the surgery formula in higher dimensions has appeared in
\cite[Appendix B]{EL17}\cite{Ekh19}, and there the map is even an
$A_{\infty}$-isomorphism.  In the two-dimensional case, all Reeb chords are
degenerate, which prevents the map from being a chain level isomorphism. We
prove that it is a quasi-isomorphism by explicitly computing the cohomology of
both sides.  As noted in \cite[Example 7.3]{AE21}, when $\partial \Lambda$
consists entirely of one-point stops, the two-dimensional case follows from
computations done in \cite[Corollary 11]{EL19}.

Using the surgery formula, one can then also describe the higher operations of
the minimal model of $CE^{*}( \partial \Lambda;V)$. We call a word in $CE^{*}(
\partial \Lambda;V)$ \emph{unconcatable} if there are no adjacent Reeb chords
in the word whose respective starting point and endpoint coincide, and
\emph{concatable} otherwise. We call a
chord \emph{short} if it is not the concatenation of any other two chords.
Finally, we call a word $c_k \cdots c_1$ of Reeb chords (of which none is an
idempotent) a \emph{disk sequence} if all adjacent chords in the word have
respective starting points and endpoints which coincide, and if when
concatenating the chords of the word into a single chord, one obtains a chord
of the form $c_{ii}^{1}$ which bounds a disk in $V$. For two short chords $c_2$
and  $c_1$ we define their concatenation $c_2 * c_1$ as the geometric
concatenation to a Reeb chord, if it makes sense, and as $0$ otherwise.

\begin{theorem}
\label{res:zero-dim-leg-min-mod}
	Let $\partial \Lambda$ be a $0$-dimensional Legendrian in the boundary
	of a Weinstein surface $V$ and suppose that none among the co-cores $C
	\subset V$ is null-homotopic relative the boundary $\partial V$. As a
	vector space, $H^{*}CE( \partial \Lambda;V )$ has a basis consisting of
	the unconcatable words of short chords. The $A_{\infty}$-operations
	$\mu_{k}$ of the minimal model structure on $H^{*}CE( \partial
	\Lambda;V )$ act as follows. The multiplication is associative and
	unital, and determined by the relations 
	\begin{align*}
		\mu_{2}( c_2 \otimes c_1 ) = 
		\begin{cases}
			\partial_{-1}( c_{2} * c_{1} ) & \text{if $c_{2}c_{1}$
			is concatable,}\\
			c_{2}c_{1} & \text{if $c_{2}c_{1}$
			is unconcatable,}\\
		\end{cases}
	\end{align*}
	for any short chords $c_2$ and  $c_{1}$, neither of which is an
	idempotent. The higher operations are for $k \ge 3$ given by
	\begin{align*}
		\mu_{k}( c_{k}\otimes\ldots\otimes c_{1}c ) = c,\\
		\mu_{k}( cc_{k}\otimes \ldots \otimes c_{1} ) = c, 
	\end{align*}
	on all tensor products of words of the form $c_{k}\otimes\ldots\otimes
	c_{1}c$ and $cc_{k}\otimes \ldots \otimes c_{1}$ such that $c_{k}\ldots
	c_{1}$ is a disk sequence, and  $c_{1}c$ and $cc_{k}$ are unconcatable,
	and vanish on all other words. Here the chord $c$ is also allowed to be an
	idempotent.

	On the other hand, if $\partial \Lambda$ contains a $0$-sphere
	$\partial \Pi$ for which $C_{\Pi} \subset C$ is null-homotopic relative
	the boundary in $V$, then there is a
	quasi-isomorphism $CE^{*}( \partial \Lambda \setminus \partial
	\Pi;V_{\partial \Pi} ) \isomto CE^{*}( \partial \Lambda;V )$.
\end{theorem}

\subsection{Finitely generated Chekanov--Eliashberg dg-algebras for singular
Legendrians}

We here state our main result. Let $V$ be a Weinstein hypersurface in $\R^{3}$
with singular Legendrian skeleton $\Lambda$.  Recall that a Weinstein
hypersurface is a hypersurface $V \subset (\R^3,\ker (dz-ydx))$ such that there
exists a contact form $e^f(dz-ydx)$ that restricts to a Liouville form on $V$
that is compatible with a Weinstein structure. In \cite[Section 2]{Eli18}, it
was shown that the skeleton of a Weinstein hypersurface is independent on the
choice of contact form (as long as the restriction is Liouville).  We assume
that the skeleton is smooth away from the subcritical part $V_{0} \subset V$,
where the latter is a union of balls. Further we assume that the skeleton
consists of the cone of a finite number of points in the boundary of $V_{0}$.
We let $\partial \Lambda \subset \partial V_{0}$ be the attaching spheres in
the handle decomposition of $V$.  The Chekanov--Eliashberg dg-algebra $CE^{*}(
\Lambda;V_{0};\R^{4})$ is, as an associative algebra, isomorphic to the path
algebra over $\bf{k}$ of the quiver with one idempotent for each top handle of
$V$, one arrow for each Reeb chord of $\Lambda$ in $\R$, and one arrow for each
Reeb chord of $\partial \Lambda$ in $\partial V_0$. Note that the algebra is
infinitely generated. The differential is defined by counting
pseudo-holomorphic curves in the symplectization. In $\R^{3}$, it can be
explicitly computed by counting admissible disks in the Lagrangian projection
\cite[Section 7.1]{AE21}, similarly to in the classical definition due to
Chekanov \cite{Che02}.

\begin{figure}[!htb]
    \centering
    
    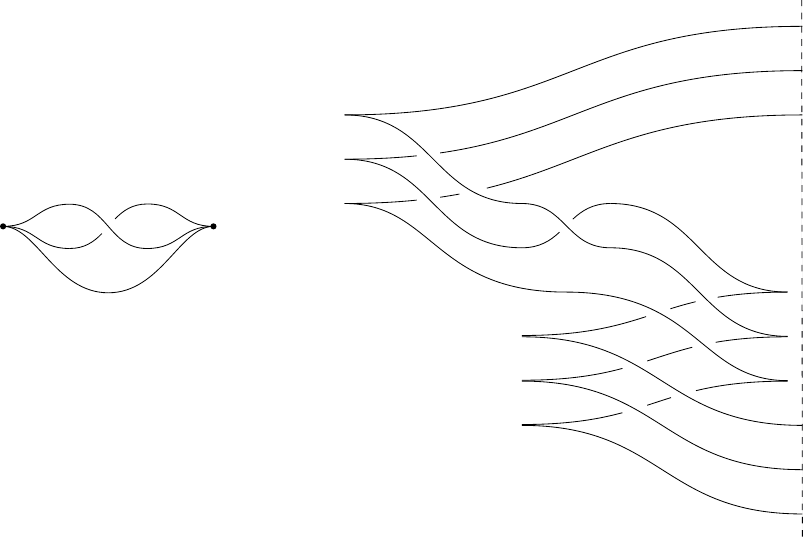

    \caption{To the left is the front projection of a singular Legendrian
	    $\Lambda$ and to the right is
    its bordered Legendrian resolution $\Lambda^{\bullet}$.}
    \label{fig:intro-resolution}
\end{figure}

Our main theorem states that for each singular Legendrian $\Lambda$, there is a
so-called \emph{bordered} or open Legendrian $\Lambda^{\bullet}$ in the sense
of \cite{Siv11, ABS22, ABS19, ABK22}, with Chekanov--Eliashberg dg-algebra
quasi-isomorphic to that of $\Lambda$. We call $\Lambda^{\bullet}$ the
\emph{resolution} of $\Lambda$; see
\autoref{def:resolution}. We need to warn the reader $\Lambda^{\bullet}$ is not
canonically defined up to compactly supported Legendrian isotopy. Despite this,
however, we can associate an invariant $CE^{*}( \Lambda^{\bullet};\R^{4} )$ to
$\Lambda^{\bullet}$ which is quasi-isomorphic to the Chekanov--Eliashberg
dg-algebra of $\Lambda$.  In the front projection, $\Lambda^{\bullet}$ is
constructed by first moving $\Lambda$ by an isotopy into a position such that
the singularities all have the same $x$-coordinate, and the rest of the
Legendrian lies to the right of the singularities. We then replace each
singularity with a negative half-twist as in \autoref{fig:intro-resolution}.
It is immaterial in which order the twists are performed, and if the ends of
the twists go above or below the Legendrian.  
\begin{theorem}[\autoref{res:resolution}]
\label{res:intro-res}
	Let $\Lambda \subset \R^{3}$ be a singular Legendrian. There is a
	quasi-isomorphism
	 \[
		 CE^{*}( \Lambda;V_{0};\R^{4} ) \cong 
		 CE^{*}( \Lambda^{\bullet};\R^{4} ).
	\] 
\end{theorem}
 
The Chekanov--Eliashberg dg-algebra of $\Lambda^{\bullet}$ has one generator for
each crossing and right cusp in the front, and the differential is given by
counting disks, just as for ordinary compact Legendrians. In particular,
$CE^{*}( \Lambda^{\bullet};\R^{4} )$ is finitely generated.

If $\Lambda$ only has one singularity, and the rest of $\Lambda$ lies to the
left of the singularity in the front projection, there is a related
construction which gives an even simpler bordered Legendrian $\Lambda^{\circ}$,
which we call the \emph{opening} of $\Lambda$; see \autoref{def:opening}.  In
the front, $\Lambda^{\circ}$ is obtained from $\Lambda$ by removing the
singularity and separating the strands, as in \autoref{fig:intro-opening}.
\begin{theorem}[\autoref{res:opening}]
\label{res:intro-opening}
	Let $\Lambda \subset \R^{3}$ be a singular Legendrian with only one
	singularity, such that the rest of $\Lambda$ lies to the left of the
	singularity in the front projection. Then there is a quasi-isomorphism
	 \[
		 CE^{*}( \Lambda;V_{0};\R^{4} ) \cong 
		 CE^{*}( \Lambda^{\circ};\R^{4} ).
	\] 
\end{theorem}
\begin{figure}[!htb]
    \centering
    
    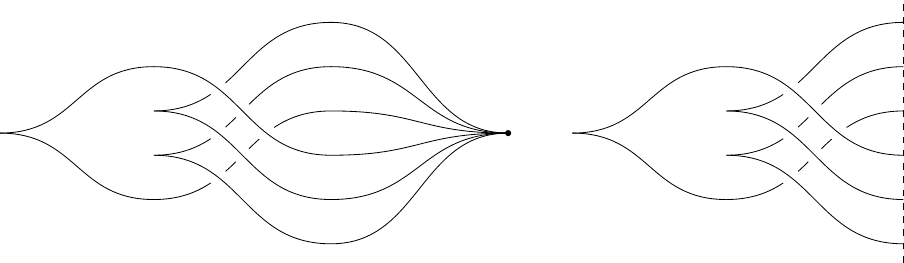

    \caption{To the left is the front projection of a singular Legendrian
	 $\Lambda$ and to the right is its bordered Legendrian opening
    	$\Lambda^{\circ}$.}
    \label{fig:intro-opening}
\end{figure}
In Section 5 we explore a number of examples where the Chekanov--Eliashberg
dg-algebra of $\Lambda^{\circ}$ is simple enough to allow one to compute the
minimal $A_{\infty}$-model. 

\subsection*{Acknowledgments}
This paper is based on the author's bachelor and master theses, which were
supervised by Georgios Dimitroglou Rizell. The author would like to express his
gratitude to him for introducing him to this research area, many helpful
discussions, and for his careful reading of earlier drafts of this paper. The
author is supported by the grant KAW 2021.0300 from the Knut and Alice
Wallenberg Foundation.

%% file: figures/internal-algebra.pdf_tex
\begingroup%
  \makeatletter%
  \providecommand\color[2][]{%
    \errmessage{(Inkscape) Color is used for the text in Inkscape, but the package 'color.sty' is not loaded}%
    \renewcommand\color[2][]{}%
  }%
  \providecommand\transparent[1]{%
    \errmessage{(Inkscape) Transparency is used (non-zero) for the text in Inkscape, but the package 'transparent.sty' is not loaded}%
    \renewcommand\transparent[1]{}%
  }%
  \providecommand\rotatebox[2]{#2}%
  \newcommand*\fsize{\dimexpr\f@size pt\relax}%
  \newcommand*\lineheight[1]{\fontsize{\fsize}{#1\fsize}\selectfont}%
  \ifx\svgwidth\undefined%
    \setlength{\unitlength}{93.63903737bp}%
    \ifx\svgscale\undefined%
      \relax%
    \else%
      \setlength{\unitlength}{\unitlength * \real{\svgscale}}%
    \fi%
  \else%
    \setlength{\unitlength}{\svgwidth}%
  \fi%
  \global\let\svgwidth\undefined%
  \global\let\svgscale\undefined%
  \makeatother%
  \begin{picture}(1,0.96520292)%
    \lineheight{1}%
    \setlength\tabcolsep{0pt}%
    \put(0,0){\includegraphics[width=\unitlength,page=1]{internal-algebra.pdf}}%
    \put(0.58122925,0.91049014){\color[rgb]{0,0,0}\makebox(0,0)[lt]{\lineheight{1.25}\smash{\begin{tabular}[t]{l}$1$\end{tabular}}}}%
    \put(0.58225909,0.0041809){\color[rgb]{0,0,0}\makebox(0,0)[lt]{\lineheight{1.25}\smash{\begin{tabular}[t]{l}$2$\end{tabular}}}}%
    \put(-0.00576681,0.45184797){\color[rgb]{0,0,0}\makebox(0,0)[lt]{\lineheight{1.25}\smash{\begin{tabular}[t]{l}$c_{12}^{0}$\end{tabular}}}}%
    \put(0,0){\includegraphics[width=\unitlength,page=2]{internal-algebra.pdf}}%
  \end{picture}%
\endgroup%

%% file: figures/intro-resolution.pdf_tex
\begingroup%
  \makeatletter%
  \providecommand\color[2][]{%
    \errmessage{(Inkscape) Color is used for the text in Inkscape, but the package 'color.sty' is not loaded}%
    \renewcommand\color[2][]{}%
  }%
  \providecommand\transparent[1]{%
    \errmessage{(Inkscape) Transparency is used (non-zero) for the text in Inkscape, but the package 'transparent.sty' is not loaded}%
    \renewcommand\transparent[1]{}%
  }%
  \providecommand\rotatebox[2]{#2}%
  \newcommand*\fsize{\dimexpr\f@size pt\relax}%
  \newcommand*\lineheight[1]{\fontsize{\fsize}{#1\fsize}\selectfont}%
  \ifx\svgwidth\undefined%
    \setlength{\unitlength}{385.33287229bp}%
    \ifx\svgscale\undefined%
      \relax%
    \else%
      \setlength{\unitlength}{\unitlength * \real{\svgscale}}%
    \fi%
  \else%
    \setlength{\unitlength}{\svgwidth}%
  \fi%
  \global\let\svgwidth\undefined%
  \global\let\svgscale\undefined%
  \makeatother%
  \begin{picture}(1,0.66857887)%
    \lineheight{1}%
    \setlength\tabcolsep{0pt}%
    \put(0,0){\includegraphics[width=\unitlength,page=1]{intro-resolution.pdf}}%
  \end{picture}%
\endgroup%

%% file: figures/intro-opening.pdf_tex
\begingroup%
  \makeatletter%
  \providecommand\color[2][]{%
    \errmessage{(Inkscape) Color is used for the text in Inkscape, but the package 'color.sty' is not loaded}%
    \renewcommand\color[2][]{}%
  }%
  \providecommand\transparent[1]{%
    \errmessage{(Inkscape) Transparency is used (non-zero) for the text in Inkscape, but the package 'transparent.sty' is not loaded}%
    \renewcommand\transparent[1]{}%
  }%
  \providecommand\rotatebox[2]{#2}%
  \newcommand*\fsize{\dimexpr\f@size pt\relax}%
  \newcommand*\lineheight[1]{\fontsize{\fsize}{#1\fsize}\selectfont}%
  \ifx\svgwidth\undefined%
    \setlength{\unitlength}{433.86549582bp}%
    \ifx\svgscale\undefined%
      \relax%
    \else%
      \setlength{\unitlength}{\unitlength * \real{\svgscale}}%
    \fi%
  \else%
    \setlength{\unitlength}{\svgwidth}%
  \fi%
  \global\let\svgwidth\undefined%
  \global\let\svgscale\undefined%
  \makeatother%
  \begin{picture}(1,0.29049556)%
    \lineheight{1}%
    \setlength\tabcolsep{0pt}%
    \put(0,0){\includegraphics[width=\unitlength,page=1]{intro-opening.pdf}}%
  \end{picture}%
\endgroup%

%% file: smoothcealgebras.tex
\section{Chekanov--Eliashberg dg-algebras for smooth Legendrians}
In this section, we first give a brief overview of the construction of the
Chekanov--Eliashberg dg-algebra for smooth Legendrians, the wrapped Floer
cohomology, and the surgery and cobordism maps.  This material follows
\cite{BEE12,EL17}, where details can be found.  We then compute the cohomology
of the Chekanov--Eliashberg dg-algebra in the boundary of a surface, which
finishes the proof of the surgery formula in dimension two.

\subsection{Smooth Chekanov--Eliashberg dg-algebras and wrapped Floer cohomology}
\label{ssec:smooth-ce}
Let $W$ be a Weinstein $2n$-manifold, with subcritical part $W_{0}$, and let
$\Lambda$ be a smooth Legendrian link of $(n-1)$-spheres living in the ideal
contact boundary $\partial W$ of $W$. Write $W_{\Lambda}$ for the Weinstein
manifold obtained by attaching critical Weinstein $n$-handles to $W$ at
$\Lambda$, and $C$ for the co-cores of the handles. We refer to \cite{CE12} for
details.  

\subsubsection{The Chekanov--Eliashberg dg-algebra} 
The Chekanov--Eliashberg dg-algebra $CE^{*}(\Lambda;W)$ of $\Lambda$ is the
semi-free dg-algebra generated by the Reeb chords of $\Lambda$, with one
idempotent for each component of $\Lambda$. We use cohomological grading and
define the degree of a chord $c$ of $\Lambda$ to be $|c|=-\text{CZ}( c )+1$,
where $\text{CZ}( c )$ is the Conley-Zehnder index. The differential $\partial
c$ is given by counting rigid (modulo translations along the $\R$-coordinate)
pseudo-holomorphic disks in the symplectization $\R\times \partial W$ of
the contact boundary $\partial W$, where the disk has boundary on the
cylindrical Lagrangian submanifold $\R \times \partial \Lambda$, taking $c$ as
input, which is the asymptotic Reeb chord at the positive end, and outputting a
word $c_{k}\ldots c_1$ of chords, which are the asymptotics at the negative
end. The disks are allowed to have anchors in $W$. An anchor in $W$ is an
interior puncture of the disk with a Reeb orbit
asymptotic at the negative end, along with an embedding of a rigid
pseudo-holomorphic punctured Riemann sphere $\C \to W$, such that the puncture
of the Riemann sphere converges the same Reeb orbit as the interior puncture of
the disk.  It follows from the dimension formula in \cite{CEL10} that
$|\partial|=1$. To see that $\partial^{2}=0$ one uses a standard
SFT-compactness argument by studying the $1$-dimensional moduli spaces of disks
taking $c$ as input and sees that they split at infinity into two-level
buildings of rigid disks corresponding to terms of $\partial^{2}$. See
\cite{EL17,Ekh19} for details. 

\subsubsection{Wrapped Floer cohomology}
The wrapped Floer cohomology of the co-cores $C$, denoted by $CW^{*}(
C;W_{\Lambda})$, is an $A_{\infty}$-algebra whose underlying module is freely
generated by the Reeb chords of the Legendrian $\partial C$ in the ideal
contact boundary $\partial W_{\Lambda}$, along with one self-intersection point
for each component of $C$. We here use the version of the wrapped Floer
cohomology without Hamiltonian from \cite[Appendix B]{EL17}.  The
self-intersections arise from choosing a system of parallel copies of $C$, see
\cite[Section 3.3]{EL17}.  The $A_{\infty}$-operations $\mu_{k}$ are defined as
follows. Write $\mu_{k} = \mu'_{k} + \mu_{k}''$ where $\mu'_{k}$ and
$\mu''_{k}$ are the components of $\mu_{k}$ that takes values in the
self-intersection point and Reeb chords, respectively. We define $\mu'$ on a
word $a_{k}\ldots a_1$ of generators of
$CW^{*}( C;W_{\Lambda})$ by counting rigid pseudo-holomorphic disks in
$W_{\Lambda}$ with boundary on $C$ taking $a_k \ldots a_1$ as input and
outputting one intersection point. We call these disks \emph{filling disks}.
Here the disk has positive asymptotics at the
Reeb chords $a_i$.  Strictly speaking, the count is performed for a
perturbation of $C$ which gives rise to a single self-intersection in the
interior of each component of $C$. If $C$ is embedded, then these are the only
self-intersection points. We define $\mu''$ by counting \emph{partial
holomorphic buildings}. These are buildings consisting of one \emph{primary
disk} in $\R\times \partial W_{\Lambda}$ with boundary on $\R \times \partial
C$ with some number of positive and negative punctures. The primary disk must
have at least two negative Reeb chord asymptotics (if it is a single strip,
then it is not of correct dimension). We call one of the negative punctures of
the primary disk the \emph{distinguished puncture} and consider it to be the
output of the building.  For each of the other negative punctures, the building
also has a \emph{secondary disk} in $W_\Lambda$ which has boundary on $C$, with
one puncture of the secondary disk converging to the corresponding output of
the primary disk. For a word $a_{k}\ldots a_1$ of generators of $CW^{*}(
C;W_{\Lambda})$ we then define $\mu_{k}''( a_{k}\ldots a_1 )$ by counting rigid
(modulo translation along the $\R$-coordinate for the primary disk) partial
holomorphic buildings, which have $a_{k}\ldots a_1$ as input, i.e. positive
asymptotics, when going counter-clockwise the around the entire building,
starting at the distinguished puncture.

\subsubsection{The surgery and cobordism maps}
The Chekanov--Eliashberg dg-algebra and the wrapped Floer cohomology are related
by the surgery formula of Bourgeois--Ekholm--Eliashberg \cite{BEE12} and
Ekholm--Lekili \cite{EL17}.

\begin{theorem}[\cite{BEE12,EL17}]
\label{res:surgery-map}
	There is an $A_{\infty}$-quasi-isomorphism	
	\[
		\Phi_{\Lambda}: CW^{*}( C;W_{\Lambda} ) \isomto CE^{*}( \Lambda;W ).
	\]
\end{theorem}
\begin{proof}
	For $n > 1$, this is \cite[Theorem 2]{EL17}. For $n=1$, it is a special
	case of \autoref{res:surgery-map-surface}
\end{proof}
The surgery map $\Phi_{\Lambda}$ is defined by considering the Weinstein
cobordism $W^{\circ}_{\Lambda} = (\R\times \partial W )_{\Lambda}$, obtained by
attaching Weinstein $n$-handles to the symplectization $\R\times \partial W$ at
$\Lambda$, where $\Lambda$ is viewed as living in the positive end of the
$\R\times \partial W$. The ends of the cobordism are then $\partial_{+}
W^{\circ}_{\Lambda} = \partial W_{\Lambda}$ and $\partial_{-}
W^{\circ}_{\Lambda} = \partial \Lambda$.  Each map
\[
	\Phi_{\Lambda}^{k}:CW^{*}( C;W_{\Lambda})^{\otimes k} \to CE^{*}( \Lambda;W )
	\quad k > 0
\]
of the $A_{\infty}$-homomorphism $\Phi_{\Lambda}$ is then defined by counting
rigid pseudo-holomorphic cobordism disks in $W_{\Lambda}^{\circ}$, as in
\autoref{fig:cobordism-disk}. For details, see \cite{EL17}.  
\begin{figure}[!htb]
    \centering
    
    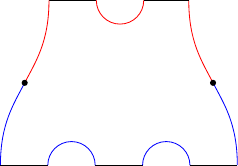

    \caption{Example of a disk contributing to the output of
    $\Phi_{\Lambda}^{2}$. The black segments are boundary punctures converging to
    Reeb chords of $\Lambda$ and $\partial C$. The red and blue segments are
    mapped to to $C$ and $L$ respectively. The black dots are mapped to $C \cap
    L$.}
    \label{fig:cobordism-disk}
\end{figure}

\begin{remark}
Though not explicitly stated, it is clear from the proof given in \cite{EL17}
that when $n > 1$, the map is in fact an $A_{\infty}$-\emph{isomorphism}, at
least this can be assumed to hold when the maps is restricted to the generators
below any fixed action level. (The
final map is defined as a limit of such maps.) This is not true when $n=1$, the
problem being that the low dimension makes all Reeb chords degenerate. However,
it is still true that the map is a quasi-isomorphism, which we prove in 
\autoref{ssec:surgery-surfaces}.  
\end{remark}
There is a closely related result which we will call the cobordism formula.
Consider a subset $\Pi$ of the attaching spheres in $\Lambda$, and the
Chekanov--Eliashberg dg-algebra of $CE( \Lambda \setminus \Pi;W_{\Pi} )$ of the
remaining part of $\Lambda$ after attaching handles at $\Pi \subset \Lambda$.
Note that $\Lambda \setminus \Pi$ is considered as a Legendrian inside the
boundary of the Weinstein manifold $W_{\Pi}$.

\begin{theorem}{\cite[Theorem 5.1 when $n > 1$]{BEE12}}
\label{res:cobordism-map}
	Let $\Pi$ be a subset of the spheres in $\Lambda$.  Then there is a
	dg-algebra morphism
	\[
		\Psi_{\Pi}: CE^{*}( \Lambda \setminus \Pi ;W_{\Pi} ) \to 
		CE^{*}( \Lambda;W )
	\]
	which is a quasi-isomorphism onto the subalgebra $CE^{*}(
	\Lambda;W)[\Lambda \setminus \Pi,\Lambda \setminus \Pi]$, consisting of
	all word of Reeb chords whose initial and terminal vertices are both
	contained in $\Lambda \setminus \Pi$. 
\end{theorem}
Note that this theorem is true for all $n > 0$.  To define $\Psi_{\Pi}$, we
perform the same handle attachment on $\R\times \partial W$ as in the
construction of the surgery map, but this time only for the subset $\Pi$ of
components of $\Lambda$. This results in a cobordism $\partial
W_{\Pi}^{\circ}$, with $\partial^{+} W^{\circ}_{\Pi} = \partial W_{\Pi}$ and
$\partial_{-} W^{\circ}_{\Pi} = \partial W$.  We define an exact Lagrangian
cobordism $L = \bigcup_{i \in I} L_{i}$ inside $W^{\circ}_{\Pi}$, by letting
$L_{i} = \Lambda_{i} \times [0,1]$ for $\Lambda_{i} \in \Lambda \setminus \Pi$
and $L_{i} = \Lambda_{i}\times [0,1] \cup L_{\Lambda_i}$ for $\Lambda_{i} \in
\Pi$, where $L_{\Lambda_{i}}$ is the core of the handle attached at
$\Lambda_{i}$. We have $\partial_{-} L \subset \partial^{-} W^{\circ}_{\Pi}$
and $\partial_{+} L = \Lambda \setminus \Pi \subset \partial^{+}
W^{\circ}_{\Pi}$. We then define   
\[
	\Psi_{\Pi}: CE^{*}( \Lambda \setminus \Pi;W_{\Pi} ) \to CE^{*}( \Lambda;W ).
\] 
to be the standard Symplectic Field Theory cobordism map as defined in e.g
\cite{EGH00, Ekh08}.  Its action on a chord $c \in CE^{*}( \Lambda \setminus
\Pi;W_{\Pi} )$ is given by counting pseudo-holomorphic disks in
$W_{\Pi}^{\circ}$ with boundary on $L$, taking $c$ as input and which outputs a
word in $CE^{*}( \Lambda;W )$.

A complete proof of \autoref{res:cobordism-map} has yet to appear in the
literature in the case when $n > 1$, but the proof sketched in \cite{BEE12} is
similar to the proof of
\autoref{res:surgery-map} which is worked out in detail in \cite[Appendix
B]{EL17}\cite{Ekh19}. \autoref{res:cobordism-map} can also be reduced to
\autoref{res:surgery-map} using the following proposition.

\begin{proposition}
\label{res:commutative-square}
Up to homotopy, the following diagram of chain complexes commutes in homotopy.
(When $n = 1$ the diagram commutes on the chain level)
\[\begin{tikzcd}
	{CW^{*}(C \setminus C_\Pi;W_\Lambda)} & {CW^{*}(C;W_\Lambda)} \\
	{CE^{*}(\Lambda \setminus 
	\Pi;W_\Pi)} & {CE^{*}(\Lambda;W),}
	\arrow["{\Psi_{\Pi}}"', from=2-1, to=2-2]
	\arrow[hook,"\iota ", from=1-1, to=1-2]
	\arrow["{\Phi_{\Lambda \setminus \Pi}}"', from=1-1, to=2-1]
	\arrow["{\Phi_{\Lambda}}", from=1-2, to=2-2]
\end{tikzcd}\]
where $\iota$ is the inclusion of generators.
\end{proposition}
Note again that this is true for all $n > 0$.
\begin{proof}
	Each term of $\Phi_{\Lambda} \circ \iota$ corresponds to a disk in
	$W_{\Lambda}^{\circ}$ with boundary on $L$ and $C\setminus C_{\Pi}$,
	and each term of $\Psi_{\Pi} \circ \Phi_{\Lambda \setminus \Pi}$
	corresponds to a building consisting of one disk in $W_{\Lambda
	\setminus \Pi}^{\circ}$ with boundary on $L$ and $C\setminus C_{\Pi}$
	and several disks in $W_{\Pi}^{\circ}$ with boundary on $L$, see
	\autoref{fig:commutative-square}. These are in one-to-one
	correspondence by a standard SFT-stretching and gluing argument. More
	specifically, we stretch the neck around the contact type hypersurface
	$\partial W_{\Pi}$ inside $W_{\Lambda}$ where the latter cobordism is
	obtained by first attaching the handles in $\Pi$ and then attaching the
	handles in $\Lambda \setminus \Pi$. For $n=1$, we obtain a direct bijection 
	of the moduli spaces of disks and buildings contributing to
	$\Phi_{\Lambda} \circ \iota$ and $\Psi_{\Pi} \circ \Phi_{\Lambda
	\setminus \Pi}$ respectively, by gluing relevant the disks by an
	explicit concatenation along the Reeb chords. When $n > 1$ the argument
	is more involved because of two reasons. First, one can only arrange so
	that the surgery formula becomes a quasi-isomorphism below some finite
	action level, depending on choices of contact forms and complex
	structures on the contact manifold obtained by surgery; one must take a
	limit of forms and complex structures to the full quasi-isomorphism.
	Note that the Chekanov-Eliashberg algebra in the bottom left corner of
	the diagram also depends on these choices of structures. Second, we
	need to use an analytic gluing of the pseudoholomorphic discs in order
	to relate disks before and after stretching. We give a sketch of the
	argument here. First, we choose a sequence of almost complex structures
	$J_{k}$, stretching the neck. Any sequence of pseudo-holomorphic disks,
	with respect to the respective $J_{k}$, contributing to $\Phi_{\Lambda}
	\circ \iota$, and of bounded energy, degenerates into a building as
	illustrated to the right in \autoref{fig:commutative-square}, when $k
	\to \infty$. By choosing an action cutoff for the generators of
	$CW^{*}(C \setminus C_\Pi;W_\Lambda)$ we obtain a bound on the energy
	of the disks, and can thus conclude that for a sufficiently large $k$,
	we get the desired bijection of the moduli spaces. Thus, the diagram
	commutes below any given action level. Taking the categorical limit of
	the diagrams with respect to the inclusion maps from lower to higher
	action levels, we then obtain the desired result.
\begin{figure}[ht]
    \centering
    
    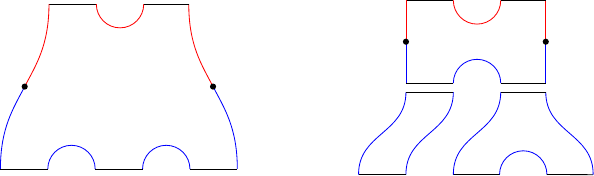

    \caption{Illustrated to the left is the type of disks contributing to
	    $\Phi \circ \iota$, and to the type of buildings contributing to
    	$\Psi_{\Pi} \circ \Phi_{\Lambda \setminus \Pi}$. The red segments map
	to $C \setminus C_{\Pi}$ and the blue segments map to $L$.}
    \label{fig:commutative-square}
\end{figure}

\end{proof}
\begin{proof}[Proof of \autoref{res:cobordism-map}]
	It is clear by construction of the wrapped Floer complex that the
	inclusion $\iota:CW^{*}( C \setminus C_{\Pi};W_{\Lambda} ) \to CW^{*}(
	C;W_{\Lambda} )$ is an isomorphism onto the subcomplex of all chords
	without endpoints on the co-cores $C_{\Pi}$ that correspond to the
	attaching spheres in $\Pi$.  The surgery maps in
	\autoref{res:commutative-square} are quasi-isomorphisms by
	\autoref{res:surgery-map} and will send any chord with endpoints on a
	given pair of co-cores to a sum of words of Reeb chords whose initial
	and terminal vertices both are located on the corresponding attaching
	spheres of $\Lambda$. It thus follows that $\Psi_{\Pi}$ is a
	quasi-isomorphism onto $CE^{*}( \Lambda;W )[\Lambda \setminus
	\Pi,\Lambda \setminus \Pi]$. 	
\end{proof}

\subsubsection{Removal of Legendrians with exact idempotents.}
The following is an algebraic lemma which we will make use of several times in
the subsequent sections. 
\begin{lemma}
\label{res:exact-removal}
	Let $A$ be a semi-free dg-algebra, whose underlying associative algebra
	is the path algebra of a quiver $Q$ with vertices $Q_0$ (the
	idempotents of $A$) and arrows $Q_1$ (the generators over the
	idempotent ring).  Suppose that there for a vertex $i \in Q_{0}$ exists
	an element $a \in Q_1$ that satisfies $\partial a = e_i$, where $e_i$
	is the idempotent of $i$, such that $a$ does not occur the differential
	of any other arrow. Let $I$ be the ideal generated by all arrows with
	source or target on $i$. Then the projection 
	\[
		A \to A \slash I
	\] 
	is a quasi-isomorphism.
\end{lemma}
\begin{proof}
	Since $A$ is an algebra over the idempotent ring, the differential will
	preserve the source and target of words. Therefore, $I$ is a chain
	complex, and we have a short exact sequence 
	\[ 
		0 \to I \to A \to A \slash I \to 0, 
	\] 
	of chain complexes.  Let $\partial$ be the differential on $I$. The
	dg-algebra has a double complex structure $\partial = \bar{\partial}_{a} +
	\partial_{a}$, where $\partial_{a}$ acts as $\partial$ on $a$ and
	vanishes on all other arrows, and $\bar{\partial}_{a} = \partial -
	\partial_{a}$. If one considers the spectral sequence arising from this
	double complex with the $\partial_{a}$-cohomology on the first page,
	one sees that it vanishes. The complex $I$ is thus exact and the result
	then follows by considering the long exact sequence in cohomology.
\end{proof}
\begin{remark}
	The special case in which we will use this is when $A = CE^{*}( \Lambda
	\cup \Pi;W )$ for some Legendrians $\Lambda$ and $\Pi$, and there is a
	Reeb chord $a$ such that $\partial a = e_{\Pi}$, which does not occur
	in the differential. Then $A \slash I$ is isomorphic to $CE^{*}(
	\Lambda ;W )$, and we thus have a quasi-isomorphism $CE^{*}( \Lambda
	\cup \Pi;W ) \isomto CE^{*}( \Lambda ;W )$.
\end{remark}

\subsection{The surgery formula for surfaces}
\label{ssec:surgery-surfaces}
As mentioned above, when $n=1$, i.e. when the contact manifold is
one-dimensional, one cannot perturb the Legendrian so that all Reeb chords are
non-degenerate.  It turns out that the bijection of the generators needed in
the argument for the proof of \autoref{res:surgery-map} when $n>1$ in
\cite{EL17} does not hold.  It is therefore necessary to consider the
cohomology of the algebras in more detail.  In order to stay consistent with
the next section, we here switch notation and denote the Weinstein manifold as
$V$ and the Legendrian in $\partial V$ as $\partial \Lambda$. 

We will also allow the Legendrian to consist of stops as well as $0$-spheres.
In this dimension that means that $\partial \Lambda$ is a finite set of points
in $\partial V$ that is partitioned into two-point $0$-spheres (the attaching
spheres of the one-handles) and one-point
stops. We then get a Weinstein sector $V_{\partial \Lambda}$ by attaching
Weinstein $1$-handles at the $0$-spheres and half-handles at the stops. The
stops should be thought of as singular Legendrian skeleta of 'degenerate'
one-point Weinstein hypersurfaces in $\partial V$. Through their respective
surgery formulas, the singular and smooth Chekanov--Eliashberg dg-algebras
correspond to wrapped and partially wrapped Floer cohomology, respectively.  In
higher dimensions they require separate treatments, the singular one being
defined in terms of the smooth one. However, when $n=1$ the theory for the
'singular' stops and the 'smooth'  $0$-spheres is essentially the same, and we
will prove the surgery formula as one theorem. Write $CE^{*}( \partial
\Lambda;V )$ for the Chekanov--Eliashberg dg-algebra of $\partial \Lambda$ and
$CW^{*}( C;V_{\partial \Lambda} )$ for the partially wrapped Floer cohomology
of the co-cores $C$ produced by the surgery. 
\begin{theorem}
\label{res:surgery-map-surface}
	Let $V$ be a Weinstein surface and $\partial \Lambda$ an embedded
	collection of $0$-spheres and points, as described above. There is an
	$A_{\infty}$-quasi-isomorphism
	\[
		\Phi_{\partial \Lambda} :CW^{*}( C;V_{\partial \Lambda} )
		\isomto
		CE^{*}( \partial \Lambda;V ).
	\] 
\end{theorem}

To prove this, we begin by giving a combinatorial description of $CE^{*}(
\partial \Lambda;V )$. The manifold $V$ is a surface, so $\partial V$ is a
disjoint union of copies of $\mathbb{S}^{1}$. There is one idempotent for each
$0$-sphere and stop. For each component of $\partial V$ we pick a base point
not in $\partial \Lambda$ and write $c_{ij}^{p}$ for the Reeb chord which
starts at $i \in \partial \Lambda$, ends at $j \in \partial \Lambda$, and
passes though the base point $p$ times.

\begin{proposition}
\label{res:zero-dim-diff}
	Let $n=1$. For a chord $c_{ij}^{p}$, let
	\[
		\partial_{0}( c_{ij}^{p} )  = \sum \pm c_{kj}^{l}c_{ik}^{p-l}
	\]
	with the sum taken over all $k$ and $l$ for which the chords on the
	right hand side exist. Let $\partial_{-1}( c_{ij}^{p} ) = e_{k}$ if
	$c_{ij}^{p}$ is a chord with $i=j$, $p=1$, and which lives in the
	boundary of a simply connected component of $V$, where $e_{k}$ is the
	corresponding idempotent of the component of $i$, and let
	$\partial_{-1}(c_{ij}^{p} )=0$ otherwise. The differential $\partial$
	of $CE^{*}( \partial \Lambda;V )$ is then given by $\partial =
	\partial_{0} + \partial_{-1}$, where each $\partial_{i}$ has been
	extended to the whole algebra by the Leibniz rule.
\end{proposition}
\begin{proof}
	The terms $\partial_{0}$ and $\partial_{-1}$ correspond to unanchored
	and anchored symplectization disks respectively, illustrated in
	\autoref{fig:ce-diff-symplectization-disks}.
\begin{figure}[H]
    \centering
    
    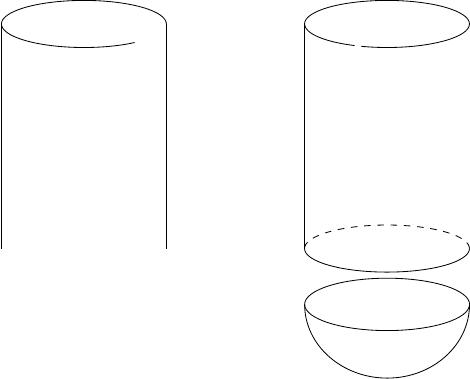

    \caption{The two types of symplectization disks contributing to $\partial$,
    with the unanchored disk to the right and the anchored disk to the left.
    The blue parts are the boundary segments of the disk which map to $\partial
    \Lambda \times \R$.}
    \label{fig:ce-diff-symplectization-disks}
\end{figure}
	Having more than one slit in the symplectization disk would result in
	it being non-rigid, and if an internal puncture were to wind more
	than once around the boundary, the anchor at the puncture would be
	non-rigid, as it would need to have a branch point.
\end{proof}

Note that $\partial_{0}$ is action preserving, while $\partial_{-1}$ is
strictly action decreasing, and that
$\partial_{0}^{2}=\partial_{-1}^{2}=\partial_{0}\circ \partial_{-1} =
\partial_{-1} \circ \partial_{0} = 0$. We will compute the
$\partial_{0}$-cohomology of $CE^{*}( \partial \Lambda;V)$, for which we need
some additional terminology. 

\begin{definition}
\label{def:short-chord-def}
	Given two Reeb chords $c_{ik}^{p}$ and $c_{kj}^{q}$ of $\partial
	\Lambda$ whose endpoint and starting point agree, their
	\emph{concatenation} is
	\[
		c_{kj}^{q} * c_{ik}^{p} := c_{ij}^{p+q}.
	\]
	A chord is called \emph{short} if it is not the concatenation of any
	other chords. We also consider the idempotents to be short chords.  Let
	$\bbc = c_{i_mj_m}^{p_m}\ldots c_{i_{1}j_{1}}^{p_{1}} \in A$ be a word
	of chords and let $l_1,\ldots,l_k$ be the subsequence $1,\ldots,m-1$ of
	all $l_{h}$ such that $j_{l_{h}}\neq i_{l_{h}+1}$. The \emph{total
	concatenation}, $\bbc_{\text{tc}}$, of $\bbc$ is the word
	\[
		\bbc_{\text{tc}}:=
		(c_{i_mj_m}^{p_m} * \ldots *
		c_{i_{l_{k}+1}j_{l_{k}+1}}^{p_{l_{k}+1}})
		\ldots
		(c_{i_{l_{h+1}}j_{l_{h+1}}}^{p_{l_{h+1}}} * \ldots *
		c_{i_{l_{h}}j_{l_{h}}}^{p_{l_h}})
		\ldots
		(c_{i_{l_{1}}j_{l_{1}}}^{p_{l_1}} * \ldots *
		c_{i_{1}j_{1}}^{p_{1}})
	\]
	i.e. the word obtained by concatenating all adjacent chords in $\bbc$
	whose respective starting point and endpoint agree. The \textit{total
	splitting}, $\bbc_{\text{ts}}$, of $\bbc$ is the unique word consisting
	only of short chords such that $(\bbc_{\text{ts}})_{\text{tc}} =
	\bbc_{\text{tc}}$, i.e. the word obtained by splitting the chords of
	$\bbc$ into two whenever they pass a point of $\partial \Lambda$. A
	word $\bbc$ is called \textit{unconcatable} if $\bbc=\bbc_{\text{tc}}$
	or if $\bbc = e_{i}$.
\end{definition}

\begin{lemma}
\label{res:short-chords-generate}
	The residue classes of the unconcatable words which only consist of
	short chords, i.e. words $\bbc$ with $\bbc = \bbc_{\textup{tc}} =
	\bbc_{\textup{ts}}$, form a basis of $H^{*}(CE( \partial
	\Lambda;V),\partial_{0})$.
\end{lemma}
\begin{proof}
	Let $\bbc$ be a word in $CE^{*}( \partial \Lambda;V )$, and let $C(
	\bbc )$ be the span of all words $\bbc'$ such that $\bbc_{\text{tc}}' =
	\bbc_{\text{tc}}$. Note that, as a complex, $(CE^{*}( \partial
	\Lambda;V ), \partial_{0})$ splits as a direct sum of the idempotents
	and the summands $C( \bbc )$ for all words $\bbc$ such that $\bbc =
	\bbc_{\text{ts}}$. Let $k$
	be the number of concatenation points of $\bbc_{\text{ts}}$ and note
	that $k=0$ if and only if $\bbc$ is an unconcatable word of short
	chords. In this case, $\partial_{0}$ vanishes on $C( \bbc ) =
	\text{Span}( \bbc )$. On the
	other hand, if $k > 0$, then one can readily compute that $C( \bbc )$
	is as a chain complex isomorphic to the tensor product of $k$ copies of
	the exact complex 
\[\begin{tikzcd}
		{0} & {\textbf{k}} & {\textbf{k}} & {0}  
		\arrow[from=1-1, to=1-2]
		\arrow["\sim",from=1-2, to=1-3]
		\arrow[from=1-3, to=1-4]
\end{tikzcd}\]
	and since we are working over a field, it then follows from the Künneth
	formula that $C(\bbc)$ itself is exact. The result follows.
\end{proof}
	We next compute the cohomology of $CW^{*}( C;V_{\partial \Lambda} )$.

\begin{lemma}
\label{res:wrapped-floer-diff}
	Suppose that $n=1$ and let $C'$ be the set of all null-homotopic
	components of $C$.  By null-homotopy we mean a contraction of the
	component keeping the endpoints in $\partial V_{\partial \Lambda}$.
	Then the inclusion,
	\[
		CW^{*}( C \setminus C';V_{\partial \Lambda} ) \to 
		CW^{*}( C;V_{\partial \Lambda} ) 
	\] 
	is a quasi-isomorphism, and moreover, $\mu_{1}$ vanishes on $CW^{*}( C
	\setminus C';V_{\partial \Lambda} )$.
\end{lemma}
\begin{proof}
	If a component is null-homotopic, it will together with some segment of
	the boundary form a contractible loop, which bounds some simply
	connected subset of $V_{\partial \Lambda}$.  This boundary segment is a
	Reeb chord $c$ of $\partial C$. There is then by the Riemann mapping
	theorem a unique holomorphic disk which takes $c$ as input, and by
	definition of the differential this disk will output the
	self-intersection point corresponding to the null-homotopic component.
	The self-intersection points act as idempotents in $CW^{*}(
	C;V_{\partial \Lambda} )$, so it follows that the inclusion is a
	quasi-isomorphism.

	To show that $\mu_{1}$ vanishes on $CW^{*}( C \setminus C';V_{\partial
	\Lambda})$, we first note that there are no intersection points between
	the different cores, while the differential vanishes on the idempotents
	(which can be though of as the unique self-intersection points of a
	small push off).  If there is a Reeb chord $c$ such that $\mu_{1}'( c )
	\neq 0$, then $c$ must begin and end at the same component. Since the
	output is non-zero, there will be a filling disk taking $c$ as input
	which outputs an intersection point. This filling disk then gives rise
	to a null-homotopy of the component at which $c$ has its endpoints. On
	the other hand, if $\mu_{1}''( c ) \neq 0$, there will be a partial
	holomorphic building taking $c$ as input, which outputs some Reeb
	chord. In order for the primary disk of this building to be rigid, it
	must have one positive puncture and two negative punctures. The
	building must therefore have a secondary filling disk with a single
	puncture meeting one of the two negative punctures of the primary disk.
	This disk then induces a null-homotopy of one of the components which
	the input has an endpoint on. 
\end{proof}

We now have everything we need to prove the surgery formula.

\begin{proof}[Proof of \autoref{res:surgery-map-surface}]
Recall that an $A_{\infty}$-morphism consists of an infinite sequence of
morphisms $\Phi_{\partial \Lambda}^{i}$. This sequence of maps was defined in
\cite{EL17} by counting cobordism disks.  Since an $A_{\infty}$-morphism is a
quasi-isomorphism if and only if the chain map $\Phi_{\partial \Lambda}^{1}$ is
a quasi-isomorphism of complexes, we will in the following restrict our
attention to this map and denote it by $\Phi_{\partial \Lambda}$ for
simplicity.  We define the \emph{action} of a chord in the boundary of $V$,
$V_{\partial \Lambda}$, or $V_{\partial \Lambda}^{\circ}$ to be the integral of
the contact form along the chord. We define the \emph{limit action} of a chord
to be the limit of its action when the handles and half-handles are shrunk.
Chords at different stages of the shrinking process can be canonically
identified, so this is well-defined.  It follows from Stokes' theorem that this
limit action is a filtration that is respected by the differentials of $CW^{*}(
C;V_{\partial \Lambda} )$ and $CE^{*}(\partial \Lambda;V )$, and that
$\Phi_{\partial \Lambda}$ is filtration preserving.

We write $\Phi_{\partial \Lambda} = \Phi_{\partial \Lambda}' + \Phi_{\partial
\Lambda}''$, where for each chord $c \in CW^{*}( C;V_{\partial \Lambda} )$, the
term $\Phi_{\partial \Lambda}(c)'$ consists of the words in $\Phi_{\partial
\Lambda}( c )$ of limit action equal to that of $c$, and $\Phi_{\partial
\Lambda}''( c )$ consists of the words of strictly lower limit action. Again by
Stokes' theorem, $\Phi_{\partial \Lambda}'$ and $\Phi''_{\partial \Lambda}$ are
the parts of $\Phi_{\partial \Lambda}$ corresponding to unanchored and anchored
cobordism disks respectively. 

We will view $\Phi_{\partial \Lambda}'$ as going between the chain complexes
$CW^{*}( C;V_{\partial \Lambda} )$ and $CE^{*}(\partial \Lambda;V )$ equipped
with the limit action preserving part of their differentials. This limit action
preserving part vanishes for $CW^{*}( C;V_{\partial \Lambda} )$ and for
$CE^{*}(\partial \Lambda;V )$ it is the differential $\partial_{0}$ from above.
We will show that with respect to these differentials, $\Phi_{\partial
\Lambda}'$  is a quasi-isomorphism.

Let $u:D^{2}\setminus \{z_1,\ldots,z_k\} \to V_{\Lambda}^{\circ}$ be a rigid
unanchored cobordism disk, as in the definition of $\Phi_{\partial \Lambda}$.
We claim that no interior points of $D^{2}\setminus \{z_1,\ldots,z_k\}$ get
mapped to $L$. This is indeed the case: The curve $u$ will be a branched cover
and the rigidity of $u$ implies that there are no branch points. The preimage
$u^{-1}\mid_{\text{int} D^{2}}(L)$ of the curve will thus be a smoothly
embedded curve in $D^{2}$, which will then either be a closed loop, or have
endpoints at two of the negative punctures, as illustrated in
\autoref{fig:rigid-curve-preimage}.  
\begin{figure}[!htb]
    \centering
    
    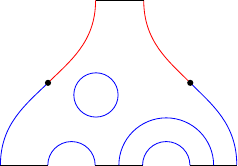

    \caption{The preimage of $L$ under $u\mid_{\text{int} D^{2}}$ will be a
	    smooth curve whose
	    components are either closed or have endpoints at two negative
	    punctures.}
    \label{fig:rigid-curve-preimage}
\end{figure}
The first situation is impossible as $u$ has no critical points and thus is a
locally injective immersion. The second situation is also impossible. This is
because the half-annuls region in \autoref{fig:rigid-curve-preimage} bounded by
the preimage curve and the boundary will have negative area by Stokes' theorem.
Thus, $u$ cannot map any interior points to $L$. 

We can therefore consider $u$ as a map to the surface $\widetilde{V}_{\partial
\Lambda}^{\circ}$ obtained from 'opening up' $V_{\partial \Lambda}^{\circ}$
along $L$ and inserting two boundary components. (Topologically this is the
same as removing a small open neighborhood of $L$.) The universal cover of
$\widetilde{V}_{\partial \Lambda}^{\circ}$ is a disjoint union of surfaces as
illustrated in the center of \autoref{fig:universal-cover}, either continuing
indefinitely, or ending at a half-handle as to the left or right in the figure.
Chords of $C$ correspond to segments of the upper component of the boundary,
and unconcatable words of short chords of $\Lambda$ correspond to segments of
the bottom component of the boundary, shown in the aforementioned image.  It
thus follows from the Riemann mapping theorem that $\Phi_{\partial \Lambda}'$
is a bijection from the set of Reeb chords on $C$ onto the set of unconcatable
words of short chords, and is hence a quasi-isomorphism by
\autoref{res:short-chords-generate}.

\begin{figure}[!htb]
    \centering
    
    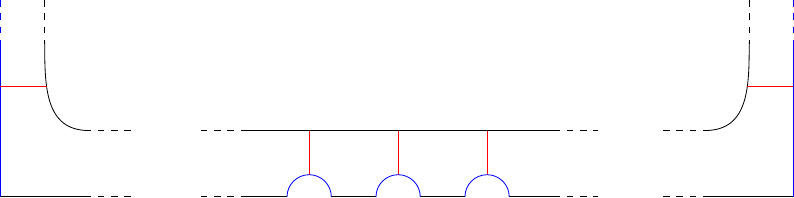

    \caption{The universal cover of $\widetilde{V}_{\partial \Lambda}^{\circ}$
    	consists of a disjoint union of surfaces as in the center, either
    	continuing indefinitely, or stopping at half-handle as to the left and
    	right.}
    \label{fig:universal-cover}
\end{figure}

Finally, we consider the mapping cone  $\text{Cone}(\Phi_{\partial \Lambda})$
of $\Phi_{\partial \Lambda}$.  Since $\Phi_{\partial \Lambda}$ preserves the
filtration, $\text{Cone}(\Phi_{\partial \Lambda})$ inherits the limit action
filtration. Considering the spectral sequence arising from this filtration, one
sees that the first page is isomorphic to $\text{Cone}(\Phi_{\partial
\Lambda}')$. Since  $\Phi_{\partial \Lambda}'$ is a quasi-isomorphism, the
spectral sequence vanishes on the second page, making $\Phi_{\partial \Lambda}$
a quasi-isomorphism.
\end{proof}
	Using the above, we can now prove \autoref{res:zero-dim-leg-min-mod}
	from the introduction.  
\begin{proof}[Proof of \autoref{res:zero-dim-leg-min-mod}]
	The assumption that no co-core is null-homotopic guarantees that
	$\partial$ vanishes on the short chords and hence they are all cycles.
	If we then consider the inclusion of the chain complex of the
	unconcatable words of short chords into $CE^{*}(\partial \Lambda;V )$,
	this is by \autoref{res:short-chords-generate} a quasi-isomorphism on
	the action preserving level, and the differential is action decreasing,
	so it follows that it is a quasi-isomorphism onto the full chain
	complex (consider the mapping cone of the inclusion, filter by action,
	and take the spectral sequence of this filtration).  This shows that
	the unconcatable words of short chords form a basis of $H^{*}CE(
	\partial \Lambda;V )$. 	

	The multiplication is induced by the multiplication on $CE^{*}(
	\partial \Lambda;V )$, and comes from the boundary $\partial( c_2 *
	c_1)=c_2c_1 + \partial_{-1}( c_2 * c_1 )$, which is satisfied whenever
	$c_{2}$ and $c_{1}$ can be concatenated. Note that the second term
	indeed is a cycle, since $\partial_{-1}^{2} = 0$ holds and since
	$\partial_0( c_2c_1 ) = 0$ by the assumption that the $c_{i}$ are short
	chords. Since the differential of
	$CW^{*}( C;V_{\partial \Lambda} )$ is zero in view of
	\autoref{res:wrapped-floer-diff} (by the assumption that no co-cores
	are null-homotopic) the higher operations can be obtained using
	\autoref{res:surgery-map}. They have been  described combinatorially in
	\cite{HKK17}, and are also easy to see in our setup.  If one takes a
	filling disk and shrinks the handles, its input and output will
	converge to a word whose total concatenation goes once around the
	boundary of $V$. When the number of inputs of a partial holomorphic
	building is greater than two, the symplectization disk can only have
	one positive puncture, as it would otherwise be non-rigid.  There
	exists such a rigid unanchored symplectization disk taking a chord $c$
	as input and outputting a word $c_2c_1$ if and only if $c=c_2*c_1$.
	Since any handle attachment results in a space not homeomorphic to a
	disk, the symplectization disks do not have any anchors. Taken
	together, these observations give the desired result. See
	\autoref{fig:higher-operations}.
\begin{figure}[!htb]
    \centering
    
    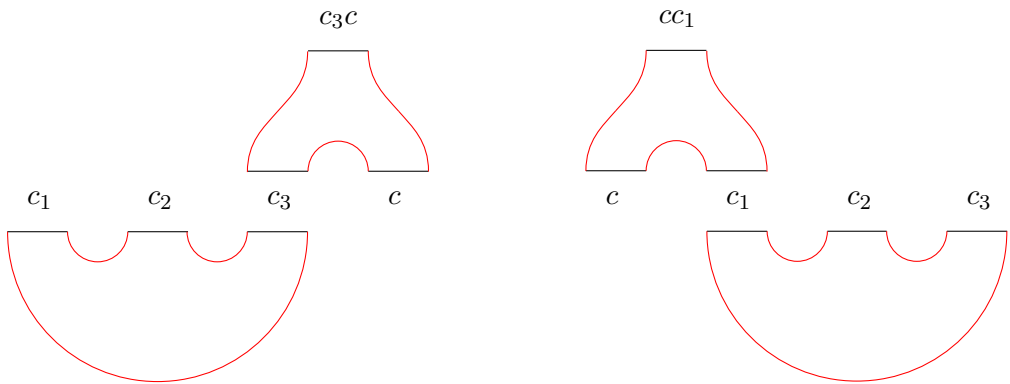

    \caption{The partial holomorphic buildings contributing to the higher
    operations.}
    \label{fig:higher-operations}
\end{figure}

	Finally, the last statement follows from
	\autoref{res:wrapped-floer-diff} and \autoref{res:commutative-square}
	combined with the fact that the horizontal arrows in
	\autoref{res:commutative-square} are quasi-isomorphisms by
	\autoref{res:surgery-map}.
\end{proof}

%% file: figures/cobordism-disk.pdf_tex
\begingroup%
  \makeatletter%
  \providecommand\color[2][]{%
    \errmessage{(Inkscape) Color is used for the text in Inkscape, but the package 'color.sty' is not loaded}%
    \renewcommand\color[2][]{}%
  }%
  \providecommand\transparent[1]{%
    \errmessage{(Inkscape) Transparency is used (non-zero) for the text in Inkscape, but the package 'transparent.sty' is not loaded}%
    \renewcommand\transparent[1]{}%
  }%
  \providecommand\rotatebox[2]{#2}%
  \newcommand*\fsize{\dimexpr\f@size pt\relax}%
  \newcommand*\lineheight[1]{\fontsize{\fsize}{#1\fsize}\selectfont}%
  \ifx\svgwidth\undefined%
    \setlength{\unitlength}{114.1133358bp}%
    \ifx\svgscale\undefined%
      \relax%
    \else%
      \setlength{\unitlength}{\unitlength * \real{\svgscale}}%
    \fi%
  \else%
    \setlength{\unitlength}{\svgwidth}%
  \fi%
  \global\let\svgwidth\undefined%
  \global\let\svgscale\undefined%
  \makeatother%
  \begin{picture}(1,0.69681174)%
    \lineheight{1}%
    \setlength\tabcolsep{0pt}%
    \put(0,0){\includegraphics[width=\unitlength,page=1]{cobordism-disk.pdf}}%
  \end{picture}%
\endgroup%

%% file: figures/commutative-square.pdf_tex
\begingroup%
  \makeatletter%
  \providecommand\color[2][]{%
    \errmessage{(Inkscape) Color is used for the text in Inkscape, but the package 'color.sty' is not loaded}%
    \renewcommand\color[2][]{}%
  }%
  \providecommand\transparent[1]{%
    \errmessage{(Inkscape) Transparency is used (non-zero) for the text in Inkscape, but the package 'transparent.sty' is not loaded}%
    \renewcommand\transparent[1]{}%
  }%
  \providecommand\rotatebox[2]{#2}%
  \newcommand*\fsize{\dimexpr\f@size pt\relax}%
  \newcommand*\lineheight[1]{\fontsize{\fsize}{#1\fsize}\selectfont}%
  \ifx\svgwidth\undefined%
    \setlength{\unitlength}{284.91097758bp}%
    \ifx\svgscale\undefined%
      \relax%
    \else%
      \setlength{\unitlength}{\unitlength * \real{\svgscale}}%
    \fi%
  \else%
    \setlength{\unitlength}{\svgwidth}%
  \fi%
  \global\let\svgwidth\undefined%
  \global\let\svgscale\undefined%
  \makeatother%
  \begin{picture}(1,0.29500287)%
    \lineheight{1}%
    \setlength\tabcolsep{0pt}%
    \put(0,0){\includegraphics[width=\unitlength,page=1]{commutative-square.pdf}}%
    \put(0.77678248,0.21427767){\color[rgb]{0,0,0}\makebox(0,0)[lt]{\lineheight{1.25}\smash{\begin{tabular}[t]{l}\footnotesize$\Phi_{\Lambda \slash \Pi}$\end{tabular}}}}%
    \put(0.66469777,0.06228908){\color[rgb]{0,0,0}\makebox(0,0)[lt]{\lineheight{1.25}\smash{\begin{tabular}[t]{l}\footnotesize$\Psi_{\Pi}$\end{tabular}}}}%
    \put(0.86434737,0.06228908){\color[rgb]{0,0,0}\makebox(0,0)[lt]{\lineheight{1.25}\smash{\begin{tabular}[t]{l}\footnotesize$\Psi_{\Pi}$\end{tabular}}}}%
    \put(0.16120645,0.14305175){\color[rgb]{0,0,0}\makebox(0,0)[lt]{\lineheight{1.25}\smash{\begin{tabular}[t]{l}\footnotesize$\Phi_{\Lambda} \circ \iota$\end{tabular}}}}%
  \end{picture}%
\endgroup%

%% file: figures/ce-diff-symplectization-disks.pdf_tex
\begingroup%
  \makeatletter%
  \providecommand\color[2][]{%
    \errmessage{(Inkscape) Color is used for the text in Inkscape, but the package 'color.sty' is not loaded}%
    \renewcommand\color[2][]{}%
  }%
  \providecommand\transparent[1]{%
    \errmessage{(Inkscape) Transparency is used (non-zero) for the text in Inkscape, but the package 'transparent.sty' is not loaded}%
    \renewcommand\transparent[1]{}%
  }%
  \providecommand\rotatebox[2]{#2}%
  \newcommand*\fsize{\dimexpr\f@size pt\relax}%
  \newcommand*\lineheight[1]{\fontsize{\fsize}{#1\fsize}\selectfont}%
  \ifx\svgwidth\undefined%
    \setlength{\unitlength}{225.59328335bp}%
    \ifx\svgscale\undefined%
      \relax%
    \else%
      \setlength{\unitlength}{\unitlength * \real{\svgscale}}%
    \fi%
  \else%
    \setlength{\unitlength}{\svgwidth}%
  \fi%
  \global\let\svgwidth\undefined%
  \global\let\svgscale\undefined%
  \makeatother%
  \begin{picture}(1,0.80539048)%
    \lineheight{1}%
    \setlength\tabcolsep{0pt}%
    \put(0,0){\includegraphics[width=\unitlength,page=1]{ce-diff-symplectization-disks.pdf}}%
    \put(0.79847808,0.64489842){\color[rgb]{0,0,0}\makebox(0,0)[lt]{\lineheight{1.25}\smash{\begin{tabular}[t]{l}$c_{ii}^{1}$\end{tabular}}}}%
    \put(0.15358389,0.64489842){\color[rgb]{0,0,0}\makebox(0,0)[lt]{\lineheight{1.25}\smash{\begin{tabular}[t]{l}$c^{p}_{ij}$\end{tabular}}}}%
    \put(-0.00156959,0.18056195){\color[rgb]{0,0,0}\makebox(0,0)[lt]{\lineheight{1.25}\smash{\begin{tabular}[t]{l}$c_{im}^{l}$\end{tabular}}}}%
    \put(0.13038804,0.16036416){\color[rgb]{0,0,0}\makebox(0,0)[lt]{\lineheight{1.25}\smash{\begin{tabular}[t]{l}$c_{mj}^{p-l}$\end{tabular}}}}%
    \put(0,0){\includegraphics[width=\unitlength,page=2]{ce-diff-symplectization-disks.pdf}}%
  \end{picture}%
\endgroup%

%% file: figures/rigid-curve-preimage.pdf_tex
\begingroup%
  \makeatletter%
  \providecommand\color[2][]{%
    \errmessage{(Inkscape) Color is used for the text in Inkscape, but the package 'color.sty' is not loaded}%
    \renewcommand\color[2][]{}%
  }%
  \providecommand\transparent[1]{%
    \errmessage{(Inkscape) Transparency is used (non-zero) for the text in Inkscape, but the package 'transparent.sty' is not loaded}%
    \renewcommand\transparent[1]{}%
  }%
  \providecommand\rotatebox[2]{#2}%
  \newcommand*\fsize{\dimexpr\f@size pt\relax}%
  \newcommand*\lineheight[1]{\fontsize{\fsize}{#1\fsize}\selectfont}%
  \ifx\svgwidth\undefined%
    \setlength{\unitlength}{113.91432755bp}%
    \ifx\svgscale\undefined%
      \relax%
    \else%
      \setlength{\unitlength}{\unitlength * \real{\svgscale}}%
    \fi%
  \else%
    \setlength{\unitlength}{\svgwidth}%
  \fi%
  \global\let\svgwidth\undefined%
  \global\let\svgscale\undefined%
  \makeatother%
  \begin{picture}(1,0.69802409)%
    \lineheight{1}%
    \setlength\tabcolsep{0pt}%
    \put(0,0){\includegraphics[width=\unitlength,page=1]{rigid-curve-preimage.pdf}}%
  \end{picture}%
\endgroup%

%% file: figures/universal-cover.pdf_tex
\begingroup%
  \makeatletter%
  \providecommand\color[2][]{%
    \errmessage{(Inkscape) Color is used for the text in Inkscape, but the package 'color.sty' is not loaded}%
    \renewcommand\color[2][]{}%
  }%
  \providecommand\transparent[1]{%
    \errmessage{(Inkscape) Transparency is used (non-zero) for the text in Inkscape, but the package 'transparent.sty' is not loaded}%
    \renewcommand\transparent[1]{}%
  }%
  \providecommand\rotatebox[2]{#2}%
  \newcommand*\fsize{\dimexpr\f@size pt\relax}%
  \newcommand*\lineheight[1]{\fontsize{\fsize}{#1\fsize}\selectfont}%
  \ifx\svgwidth\undefined%
    \setlength{\unitlength}{381.07930341bp}%
    \ifx\svgscale\undefined%
      \relax%
    \else%
      \setlength{\unitlength}{\unitlength * \real{\svgscale}}%
    \fi%
  \else%
    \setlength{\unitlength}{\svgwidth}%
  \fi%
  \global\let\svgwidth\undefined%
  \global\let\svgscale\undefined%
  \makeatother%
  \begin{picture}(1,0.24847306)%
    \lineheight{1}%
    \setlength\tabcolsep{0pt}%
    \put(0,0){\includegraphics[width=\unitlength,page=1]{universal-cover.pdf}}%
  \end{picture}%
\endgroup%

%% file: figures/higher-operations.pdf_tex
\begingroup%
  \makeatletter%
  \providecommand\color[2][]{%
    \errmessage{(Inkscape) Color is used for the text in Inkscape, but the package 'color.sty' is not loaded}%
    \renewcommand\color[2][]{}%
  }%
  \providecommand\transparent[1]{%
    \errmessage{(Inkscape) Transparency is used (non-zero) for the text in Inkscape, but the package 'transparent.sty' is not loaded}%
    \renewcommand\transparent[1]{}%
  }%
  \providecommand\rotatebox[2]{#2}%
  \newcommand*\fsize{\dimexpr\f@size pt\relax}%
  \newcommand*\lineheight[1]{\fontsize{\fsize}{#1\fsize}\selectfont}%
  \ifx\svgwidth\undefined%
    \setlength{\unitlength}{376.89900556bp}%
    \ifx\svgscale\undefined%
      \relax%
    \else%
      \setlength{\unitlength}{\unitlength * \real{\svgscale}}%
    \fi%
  \else%
    \setlength{\unitlength}{\svgwidth}%
  \fi%
  \global\let\svgwidth\undefined%
  \global\let\svgscale\undefined%
  \makeatother%
  \begin{picture}(1,0.37240777)%
    \lineheight{1}%
    \setlength\tabcolsep{0pt}%
    \put(0,0){\includegraphics[width=\unitlength,page=1]{higher-operations.pdf}}%
    \put(0.38098774,0.17779287){\color[rgb]{0,0,0}\makebox(0,0)[lt]{\lineheight{1.25}\smash{\begin{tabular}[t]{l}$c$\end{tabular}}}}%
    \put(0.59857981,0.17779287){\color[rgb]{0,0,0}\makebox(0,0)[lt]{\lineheight{1.25}\smash{\begin{tabular}[t]{l}$c$\end{tabular}}}}%
    \put(0.02009369,0.17779287){\color[rgb]{0,0,0}\makebox(0,0)[lt]{\lineheight{1.25}\smash{\begin{tabular}[t]{l}$c_1$\end{tabular}}}}%
    \put(0.7190419,0.17779287){\color[rgb]{0,0,0}\makebox(0,0)[lt]{\lineheight{1.25}\smash{\begin{tabular}[t]{l}$c_1$\end{tabular}}}}%
    \put(0.1405544,0.17779287){\color[rgb]{0,0,0}\makebox(0,0)[lt]{\lineheight{1.25}\smash{\begin{tabular}[t]{l}$c_2$\end{tabular}}}}%
    \put(0.83999774,0.17779287){\color[rgb]{0,0,0}\makebox(0,0)[lt]{\lineheight{1.25}\smash{\begin{tabular}[t]{l}$c_2$\end{tabular}}}}%
    \put(0.26002977,0.17779287){\color[rgb]{0,0,0}\makebox(0,0)[lt]{\lineheight{1.25}\smash{\begin{tabular}[t]{l}$c_3$\end{tabular}}}}%
    \put(0.95947357,0.17779287){\color[rgb]{0,0,0}\makebox(0,0)[lt]{\lineheight{1.25}\smash{\begin{tabular}[t]{l}$c_3$\end{tabular}}}}%
    \put(0.3121164,0.35827891){\color[rgb]{0,0,0}\makebox(0,0)[lt]{\lineheight{1.25}\smash{\begin{tabular}[t]{l}$c_3c$\end{tabular}}}}%
    \put(0.6515318,0.35925473){\color[rgb]{0,0,0}\makebox(0,0)[lt]{\lineheight{1.25}\smash{\begin{tabular}[t]{l}$cc_1$\end{tabular}}}}%
  \end{picture}%
\endgroup%

%% file: singcealgebras.tex
\section{Chekanov--Eliashberg dg-algebras for singular Legendrians}
In this section, we give an overview of the Asplund--Ekholm construction of the
Chekanov--Eliashberg dg-algebra for singular Legendrians, and how to compute it
in $\R^{3}$ using the combinatorial description by An--Bae \cite{AB20}. The
definition immediately gives singular versions of the surgery and cobordism
maps for smooth Legendrians. We also review the singular Legendrian
Reidemeister moves and the Ng resolution. 

\subsection{Chekanov--Eliashberg dg-algebras for singular Legendrians}
\label{ssec:ceforsing}
This material follows \cite{AE21}, to which we refer for details.  Let $W$ be a
$2n$-dimensional Weinstein manifold, and let $V \subset \partial W$ be a
$(2n-2)$-dimensional Weinstein hypersurface. We will also allow the case
$W=B^{2n} \setminus \text{pt}$ for some $\text{pt} \in \partial B^{2n} =
S^{2n-1}$. The latter corresponds to the Weinstein manifold with a stop at a
point in the boundary, and $\partial W = \R^{2n+1}$ can be identified with the
standard contact vector space, i.e. the standard contact Darboux ball. We
denote the singular Legendrian skeleton of $V$ as $\Lambda$, and write $V_{0}$
for the subcritical part of $V$ and $\partial \Lambda \subset \partial V_{0}$
for the Legendrian attaching attaching spheres in the handle decomposition of
$V$. We construct a cobordism $W_{V}^{0}$ as follows.  Let
$D^{*}_{\epsilon}[-1,1]$ be an $\epsilon$-disk subbundle of $T^{*}[-1,1]$. Then
$W_{V}^{0}$ is obtained by attaching the handle $V_{0}\times
D^{*}_{\epsilon}[-1,1]$ to $W \sqcup (\R\times ( \R\times V ))$ along
$V_{0}\times D^{*}_{\epsilon}[-1,1]\mid_{\{-1,1\}} \subset \partial W \sqcup (
\R\times V )$. We have $\partial_{-}W_{V_{0}} = \R \times V$.  We construct a
smooth Legendrian link $\Sigma$ of $n$-spheres in $\partial W_{V}^{0}$ by
joining the two copies of the top strata $\Lambda \setminus V_{0}$ in $\partial
W$ and $\R \times V$ across the handle by $\partial \Lambda \times [-1,1]$, as
shown in \autoref{fig:singular-definition}. For each Legendrian handle $\Pi$ in
$\Lambda$ we denote the corresponding Legendrian attaching sphere by $\Sigma(
\Pi ) \subset \Sigma$.
\begin{figure}[!htb]
    \centering
    
    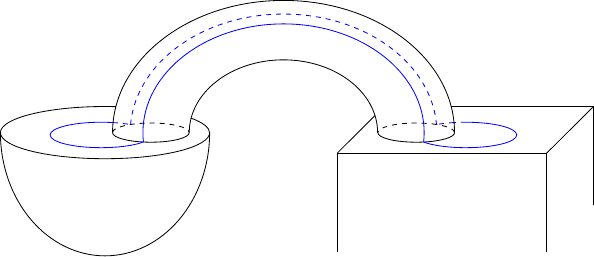

    \caption{The cobordism $W_{V}^{0}$ in the Asplund--Ekholm construction.}
    \label{fig:singular-definition}
\end{figure}

\begin{definition}[{\cite[Definition 3.2]{AE21}}]
The \emph{singular Chekanov--Eliashberg dg-algebra} of $\Lambda$ in $\partial
W$, denoted by $CE^{*}( \Lambda;V_{0};W)$, is the ordinary smooth
Chekanov--Eliashberg dg-algebra as in \autoref{ssec:smooth-ce} of $\Sigma$ in
$\partial W_{V}^{0}$, i.e.
\[
	CE^{*}( \Lambda;V_{0};W ):=CE^{*}( \Sigma;W_{V}^{0} ).
\] 
\end{definition}

\begin{lemma}[\cite{AE21}]
\label{res:ce-sing-subalg}
	The Chekanov--Eliashberg dg-algebra 
	$CE^{*}( \Lambda;V_{0};W )$ contains $CE^{*}( \partial
	\Lambda;V_{0} )$ as a dg-subalgebra.
\end{lemma}
\begin{proof}
	See \cite[Section 3.3]{AE21}.
\end{proof}
The definition immediately gives an analogue of the surgery formula for
singular Legendrians.  We write $W_{V}:=W_{V,\Sigma}^{0}$ for the Weinstein
manifold obtained by handle attachment to $W_{V}^{0}$ to along $\Sigma$, and
denote the co-cores of $W_{V}$ as $C$. 

\begin{theorem}[{\cite[Theorem 1.1]{AE21}}]
\label{res:sing-surgery-map}
	There is an $A_{\infty}$-quasi-isomorphism	
	\[
		CE^{*}( \Lambda;V_{0};W ) \isomto CW^{*}( C;W_{V} ).
	\] 
\end{theorem}
\begin{proof}
	Follows from \autoref{res:surgery-map} and the definition of $CE^{*}(
	\Lambda;V_{0};W )$.
\end{proof}
	This surgery formula is the same as the one in \cite[Conjecture
	3]{EL17} \cite{Syl19}. The wrapped Floer cohomology of $W_{V}$, can
	equivalently be thought of as the partially wrapped Floer cohomology of
	$W$ stopped at $V$. 
\begin{remark}	
	A smooth Legendrian $\Lambda$ can be considered as a singular
	Legendrian with a single top handle, and the singular
	Chekanov--Eliashberg dg-algebra is then quasi-isomorphic to the
	Chekanov--Eliashberg dg-algebra with based loop space coefficients, see
	\cite[Theorem 1.2, Section 4]{AE21}. 
\end{remark}

	We will also consider a relative version of the algebra. 

	\begin{definition}
	Let $\Pi \subset \Lambda$ be a subset of the Legendrian cores of the
	top handles of $V$. Let $W_{V,\Sigma( \Pi )}^{0}$ be the cobordism
	obtained by handle attachment along $\Sigma( \Pi )$ in $W_{V}^{0}$. The
	Chekanov--Eliashberg dg-algebra of $\Lambda \setminus \Pi$ \emph{relative}
	$\Pi$ is 
	\[
		CE^{*}( \Lambda \setminus \Pi;V_{0,\partial \Pi};W ):= 
		CE^{*}( \Sigma \setminus \Sigma_{\Pi};W_{V,\Sigma_\Pi}^{0}).  
	\]
	\end{definition}
	This is a slightly different construction than in
	\cite[Section 6.1]{AE21}. However, since there are no Reeb chords in
	the negative part of $\partial W^{0}_{V,\Sigma( \Pi )}$ that
	corresponds to $V \times \R$, and no pseudo-holomorphic curves can
	enter there (\cite[Lemma 3.5]{AE21}), it results in a dg-algebra that
	is isomorphic to the original relative algebra construction by
	Asplund--Ekholm. This definition gives an analogue of the cobordism map.

\begin{theorem}
\label{res:sing-cobordism-map}
	There is a dg-algebra morphism
	\[
		CE^{*}( \Lambda \setminus \Pi;V_{0,\partial \Pi};W ) \to 
		CE^{*}( \Lambda;V_{0};W )
	\]
	which is a quasi-isomorphism onto 
	$CE^{*}( \Lambda;V_{0};W )[\Lambda \setminus \Pi,\Lambda \setminus
	\Pi]$.
\end{theorem}
\begin{proof}
	Follows from \autoref{res:cobordism-map} and the definition of $CE^{*}(
	\Lambda;V_{0};W )$ and $CE^{*}( \Lambda \setminus \Pi;V_{0,\partial
	\Pi};W )$.
\end{proof}
As suggested by the notation, there is also a relative version of
\autoref{res:ce-sing-subalg}.
\begin{lemma}[\cite{AE21}]
	The Chekanov--Eliashberg dg-algebra $CE^{*}( \Lambda \setminus
	\Pi;V_{0,\partial \Pi};W )$ contains $CE^{*}( \partial \Lambda
	\setminus \partial \Pi;V_{0, \partial \Pi} )$ as a dg-subalgebra.
\end{lemma}
\begin{proof}
	See \cite[Section 3.3]{AE21}.
\end{proof}
Asplund--Ekholm give a surgery presentation of this algebra. Since
\autoref{res:cobordism-map} fails to be an isomorphism for $n=2$ this case
needs to be treated separately.

\begin{lemma}[{\cite[Proposition 6.2]{AE21}}]
\label{res:surgery-presentation}
	For $n > 2$, the algebra $CE^{*}( \Lambda \setminus \Pi;V_{0,\partial
	\Pi};W)$ is generated by composable words of Reeb chords of $\Lambda$
	in $\partial W$ and $\partial \Lambda$ in $V_{0}$, without an endpoint
	on $\Pi$ or $\partial \Pi$. The subalgebra $CE^{*}( \partial \Lambda
	\setminus \partial \Pi; V_{0,\partial \Pi})$ is generated by composable
	words of Reeb chords of $\partial \Lambda$ in $V_{0}$, without
	endpoints on $\partial \Pi$. For $n=2$, the algebra is quasi-isomorphic
	to the algebra described above, and isomorphic to the algebra generated
	by composable words of Reeb chords of $\Lambda$ in $\partial W$ and
	$\partial \Lambda \setminus \partial \Pi$ in $V_{0,\partial \Pi}$,
	where the words have no endpoint on $\Pi$.  
\end{lemma}

\subsection{The singular algebra in $\R^{3}$}	
In \cite[Section 7.1]{AE21}, Asplund--Ekholm show that in $\R^{3}$, their
invariant is isomorphic to the combinatorial version by An--Bae in \cite{AB20}.
We will use this below, and give a brief summary of the method. 

Let $V \subset \R^{3}$ be a Weinstein hypersurface and let $\pi: \R^{3} \to
\R^{2}$, $( x,y,z ) \to ( x,y )$ be the Lagrangian projection. We assume that
$V$ is in generic position so that $\pi|_{V_{0}}$ is an embedding and let
$\R^{2,\circ} := (\R^{2} \setminus \pi( V_{0} ) )\sqcup_{\partial V_{0}}
(\partial V_{0} \times (-\infty, 0 ] )$, i.e. the surface obtained by cutting
out the image of the subcritical part of the hypersurface and attaching the
negative end of the symplectization of $\partial V_{0}$ instead.  The image
$\widetilde{\Lambda} := \pi( \Lambda \setminus (\Lambda \cap V_{0}))$ of the
cores of the top handles of $V$ forms an exact Lagrangian, with negative end
$\partial \Lambda$ in $\partial_{-} \R^{2, \circ} \cong \partial V_{0}$. 

We can then consider the version of the Legendrian dg-algebra $CE^{*}(
\widetilde{\Lambda};\R \times \R^{2,\circ})$ as described by An--Bae
\cite{AB20}. It is generated by double points of $\widetilde{\Lambda}$ and Reeb
chords of $\partial \Lambda$, and its differential is given by counting disks
in $\R^{2,\circ}$ with boundary on $\widetilde{\Lambda}$, whose punctures
converge to either a double point or to a chord in the boundary, similarly to
in the classical Chekanov--Eliashberg dg-algebra of smooth knots in $\R^{3}$.

The grading is easiest to compute in the front projection. Let $S$ be the
subset of $\Lambda$ consisting of all singularities along with the points
which are cusps in the front (i.e. the points with $y$-coordinate zero). We
choose a \emph{Maslov potential} $m:\Lambda \setminus S \to \Z$, such that $m$
is locally constant on $\Lambda \setminus S$, and increases with one when we
pass an up cusp and decreases with one when we pass a down cusp. The degree of
Reeb chord $a$ of $\Lambda$ with starting point $a^{-}$ and endpoint $a^{+}$
then has grading $ | a | = -m( a^{+} ) + m( a^{-} )$. The chords of
$c_{ij}^{p}$ of $\Lambda$ are graded by $| c_{ij}^{p} | = q + 1 + (- m( j ) +
m( i ))$, where $q$ is the number of times $c_{ij}^{p}$ passes through the
$z$-axis in the front (or equivalently, the $y$-axis in the Lagrangian
projection).

\begin{proposition}
\label{res:ce-in-contactization}
	There is an isomorphism of dg-algebras,
	\[
		CE^{*}(\widetilde{\Lambda};\R \times \R^{2,\circ}) 
		\cong CE^{*}( \Lambda;V_{0};\R^{4}  ) 
	\] 
\end{proposition}
\begin{proof}
	This is a special case of \cite[Lemma 7.1]{AE21}.	
\end{proof}
We give an example of how to compute the algebra. It is adapted from
\cite[Section 4.6.2]{AB20}.

\begin{example}
	Let $\Lambda$ be the singular Legendrian illustrated in
	\autoref{fig:pinched-figure-eight}. 
\begin{figure}[!htb]
    \centering
    
    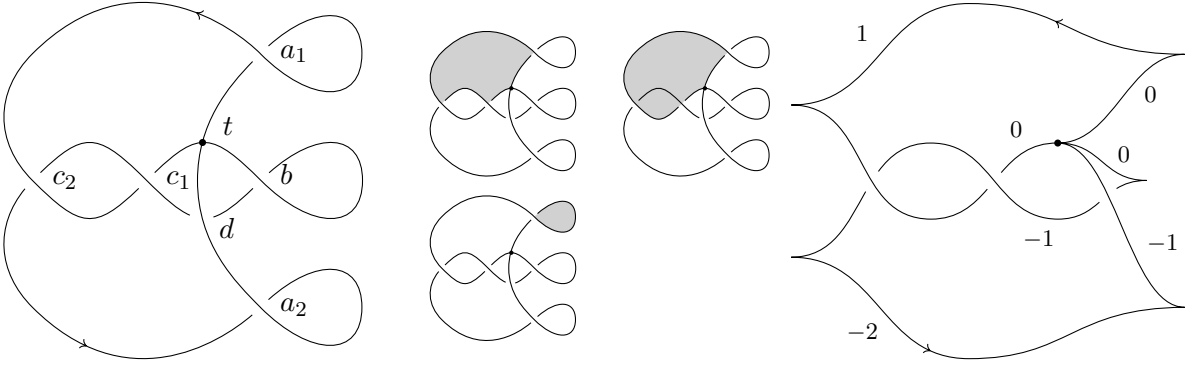

    \caption{To the left is the Lagrangian projection of a singular Legendrian
    $\Lambda$. In the center are the disks contributing to $\partial a_1$. To
    the right is the front projection, with a choice of Maslov potential.}
    \label{fig:pinched-figure-eight}
\end{figure}
	It has six Reeb chords
	$a_1,a_2,b,c_1,c_2$, and $d$, and an infinite number of
	chords $t_{ij}^{p}$ of $\partial \Lambda$. We label the points of
	$\partial \Lambda$ as $1,2,3,4,$ going counter-clockwise from the top. 
	Giving the strands of the front
	projections Maslov potentials as shown in
	\autoref{fig:pinched-figure-eight}, we get the degrees
	\begin{equation*}
	\begin{aligned}[c]
		| a_1 | &= -1,\\
		| a_2 | &= -1,\\
		| b | &= -1,
	\end{aligned}
	\quad\quad\quad\quad
	\begin{aligned}[c]
		| c_1 | &= -1,\\
		| c_2 | &= 1,\\
		| d | &= 0,
	\end{aligned}
	\quad\quad\quad\quad
	\begin{aligned}[c]
		| t_{ij}^{p} | &= -2p+1.
	\end{aligned}
	\end{equation*}
	Counting the disks in the Lagrangian projection we get the differential,
	\begin{equation*}
	\begin{aligned}[c]
		\partial a_1  &= c_1c_2t_{12}^{0} + t_{12}^{0} + e_1,\\
		\partial a_2  &= dc_2c_1 + t_{23}^{0}c_1 + d + e_2,\\
		\partial b  &= dt_{34}^{0} + e_2,
	\end{aligned}
	\quad\quad\quad\quad
	\begin{aligned}[c]
		\partial  c_1  &= 0,  \\
		\partial  c_2  &= 0, \\
		\partial  d  &= 0,
	\end{aligned}
	\end{equation*}
	and as before
	\[
		\partial t_{ij}^{p} = \sum_{k,l} t_{kj}^{l}t_{ik}^{p-l}.
	\] 
	The disks contributing to $\partial a_{1}$ are illustrated in 
	\autoref{fig:pinched-figure-eight}. 
\end{example}
\subsection{Weinstein isotopy and invariance}
\label{ssec:isotopy}
There are two notions of isotopy of $V$. The first is \emph{Weinstein isotopy},
which is a smooth isotopy (not necessarily preserving the contact distribution) of 
$V$ through Weinstein hypersurfaces, see
\cite{Eli18}. For a generic Weinstein isotopy, the skeleton undergoes an
isotopy, together with handle-slides and introduction and contraction of
handles. See \autoref{fig:weinstein-reidemeister}. 
The second notion is that of \emph{Legendrian isotopy} of a Weinstein skeleton,
which is an ambient contact isotopy of this singular Legendrian.
In particular, a Legendrian isotopy of a Weinstein skeleton induces a Weinstein
isotopy of the corresponding thickening, preserving the handle
decomposition. Recall that any Weinstein hypersurface
has a naturally defined
skeleton that is independent of the contact form; see \cite[Section 2]{Eli18}.
The important difference between these notions
is that the quasi-isomorphism type of $CE^{*}( \Lambda;V_{0};W )$ is invariant
under Legendrian isotopy, but not Weinstein isotopy. 

\begin{theorem}[Asplund--Ekholm \cite{AE21}]
\label{res:invariance}
	The Chekanov--Eliashberg dg-algebra $CE^{*}( \Lambda;V_{0};W )$ is up to
	quasi-isomorphism invariant under Legendrian isotopy of $\Lambda$.
\end{theorem}
\begin{proof}
	See \cite[Section 2.1]{AE21}. Legendrian isotopies of $V$ induce
	Legendrian isotopies of the link $\Sigma$ in the definition of $CE^{*}(
	\Lambda;V_{0};W )$, so the result follows from the isotopy invariance
	of the Chekanov--Eliashberg dg-algebra for smooth Legendrians.  
\end{proof}

\begin{remark}
For $\partial W=\R^{3}$, An--Bae have introduced the stronger notion of
\emph{generalized stable tame isomorphism} and shown that two singular
Legendrian knots in $\R^{3}$ are Legendrian isotopic if and only if their
Chekanov--Eliashberg dg-algebras are generalized stable tame isomorphic
\cite[Theorem 5.1]{AB20}.
\end{remark}

Below we will explore several concrete examples of how invariance fails for
Weinstein isotopies. Also see \cite[Examples 7.4 and 7.5]{AE21}. A
Weinstein isotopy will typically alter the handle decomposition of $W_{V}$.
Since the co-cores $C$ generate the wrapped Fukaya category
$\mathcal{W}(W_{V})$ \cite{CRGG17, GPS18} one expects the category
$\text{Perf}(CE^{*}( \Lambda;V_0;W ))$ of perfect complexes of dg-modules, i.e.
semi-free dg-modules over $CE^{*}( \Lambda;V_{0};W )$, to be invariant up to
quasi-equivalence under Weinstein isotopy. See also \cite[Remark 1.5]{AE21}.

When $\partial W=\R^{3}$, one can see that the algebra changes in a predicable way when
performing a handle contraction.

\begin{corollary}
\label{res:weinstein-isotopy-cor}
	Let $V \subset \R^{3}$ be a Weinstein hypersurface with skeleton
	$\Lambda$ and let $V' \subset \R^{3}$ be a Weinstein isotopic
	hypersurface with skeleton $\Lambda'$, obtained by performing a
	Legendrian isotopy and contracting a single handle $\Pi$ of $V$ as
	in \autoref{fig:weinstein-reidemeister}. Then there is a dg-algebra
	morphism
	\[
		CE^{*}( \Lambda';V_{0}';\R^{4} ) \to CE^{*}(
		\Lambda;V_{0};\R^{4} )
	\]
	which is a quasi-isomorphism onto $CE^{*}( \Lambda;V_{0};\R^{4}
	)[\Lambda\setminus\Pi,\Lambda\setminus\Pi]$.
\end{corollary}
\begin{proof}
	By \autoref{res:invariance}, Legendrian isotopy preserves the
	quasi-isomorphism type, so we need only consider the handle
	contraction. It follows from \autoref{res:ce-in-contactization} and
	\autoref{res:surgery-presentation} that $CE^{*}(
	\Lambda';V_{0}';\R^{4} )$ is isomorphic to the relative algebra
	$CE^{*}( \Lambda \setminus \Pi;V_{0,\partial \Pi};\R^{4} )$, and by
	\autoref{res:sing-cobordism-map} there is a dg-algebra morphism
	$CE^{*}( \Lambda \setminus \Pi;V_{0,\partial \Pi};\R^{4} ) \to CE^{*}(
	\Lambda;V_{0};\R^{4} )$ which is a quasi-isomorphism onto $CE^{*}(
	\Lambda;V_{0};\R^{4})[\Lambda\setminus\Pi,\Lambda\setminus\Pi]$. The result follows.
\end{proof}

\subsection{Reidemeister moves and the Ng resolution} 
In $\R^{3}$, there are versions of the Reidemeister moves for singular
Legendrians, extending the moves for smooth Legendrians introduced in
\cite{Swi92}.

\begin{proposition}[\cite{BI09}]
	Two singular Legendrians in $\R^{3}$ are Legendrian isotopic (after
	deforming the angles of the strands at the singular points) if and
	only if their front diagrams are related by a sequence of planar
	isotopies and singular Legendrian Reidemeister moves as shown in
	\autoref{fig:sing-legendrian-reidemeister}.
\end{proposition}
\begin{proof}
	See \cite[Section 4]{BI09}.
\end{proof}
\begin{figure}[!htb]
    \centering
    
    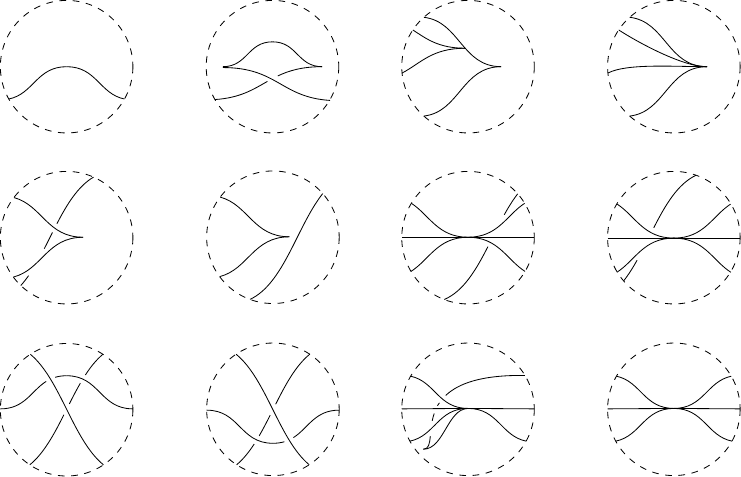

    \caption{The singular Legendrian Reidemeister moves, introduced in
	    \cite{BI09}. Reflections of these moves, with the crossings adjusted
	    accordingly, are also allowed. In the
	    moves involving singularities, any number of strands are allowed.}
    \label{fig:sing-legendrian-reidemeister}
\end{figure}
\begin{remark}
	The set of moves in \autoref{fig:sing-legendrian-reidemeister} is not
	minimal; IV is a special case of VI.
\end{remark}
There is an analogous result for Weinstein isotopies.
\begin{proposition}
\label{res:weinstein-lemma}
	Two singular Legendrians in $\R^{3}$ have Weinstein isotopic
	thickenings if and only
	if their front diagrams are related by a sequence of planar isotopies,
	singular Legendrian Reidemeister moves as shown in
	\autoref{fig:sing-legendrian-reidemeister}, and Weinstein handle
	introductions and contractions as shown in
	\autoref{fig:weinstein-reidemeister}.
\end{proposition}
\begin{remark}
	The set of moves in \autoref{fig:weinstein-reidemeister} is not minimal
	either; any one of the three types of moves can can be
	obtained from any other by composing with Legendrian isotopies. 
\end{remark}
\begin{figure}[!htb]
    \centering
    
    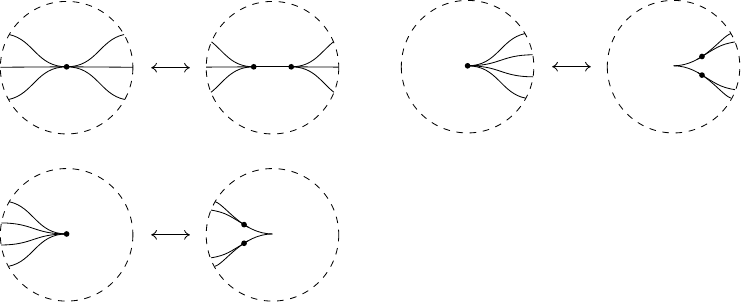

    \caption{Introduction and contraction of Weinstein handles. The
	    two groups of strands that become separated after inserting the new
    strand (i.e. located to the left and right in the top left picture) do not
    necessarily consist of the same number of components, they can even be
    empty.}
    \label{fig:weinstein-reidemeister}
\end{figure}

\begin{proof}
	We first prove that the moves applied to the skeleton can be realized
	by a Weinstein isotopy of the thickening. Consider a small deformation
	of a singular Legendrian $\Lambda^0$ with
	thickening $V^0$ by one of the moves in \autoref{fig:weinstein-reidemeister} 
	to produce a singular
	Legendrian $\Lambda^1$. 

	Recall that the Lagrangian projection of the Weinstein hypersurface to $\R^{2}$
	is an immersed Weinstein surface. If $\Lambda^{0}$ is the skeleton of a
	Weinstein thickening, one can readily find a Weinstein homotopy inside that
	surface that deforms the skeleton to the Lagrangian projection of
	$\Lambda^{1}$. The Weinstein homotopy deforms the Liouville form by the
	addition of an exact term $dh_{t}$ where $h_{0} = 0$. The Weinstein isotopy
	can then be obtained from the initial Weinstein embedding by addition of the
	function $h_{t}$ to the $z$-coordinate. 


	We now conversely show that a generic Weinstein isotopy can be assumed to
	induce a deformation of the skeleton that can be realized by the moves.
	For a Weinstein isotopy $(V^t,\lambda^t,\phi^t)$ for which $\phi^t$ remains
	Morse, one can readily deform the skeleton so that it is induced by an
	ambient contact isotopy. First, one perform an explicit deformation near
	the 0-skeleton to ensure that there is an ambient contact isotopy that
	induces that deformation. Then, one deforms the remaining part of the
	skeleton by using the standard result that a Legendrian isotopy is
	induced by an ambient contact isotopy (see \cite{Gei08}). Such
	deformations are induced by the Reidemeister moves.

	A generic Weinstein isotopy consists of generic times when $\phi^t$
	remains Morse, together with a finite number of singular moments that
	are handle-slides, or births or deaths. Note that both these singular
	moments can be realized by the moves from Figure 12 together with a
	Legendrian isotopy if one assume the handle-slides and birth-deaths to
	be generic. More precisely, the birth-death is given by any of the
	moves in \autoref{fig:weinstein-reidemeister} where the new singularity
	is one-valent. Handle-slides can be performed by a contraction of an
	arc as in \autoref{fig:weinstein-reidemeister}, followed by an
	introduction of a new arc. The statement follows from this.
\end{proof}
Since one typically computes the Chekanov--Eliashberg dg-algebra in the Lagrangian
projection but constructs isotopies in the front projection, it is useful to be able to
go back and forth between them. This is done using the so called
Ng resolution, which was introduced by Ng for smooth Legendrians in
\cite{Ng03} and generalized to singular Legendrians by An--Bae--Su in
\cite{ABS22}.
\begin{definition}
	Let $\Lambda \subset \R^{3}$ be a singular Legendrian. The \emph{Ng
	resolution} of $\Lambda$ is the singular Legendrian whose Lagrangian
	projection (up to an isotopy of $\R^{2}$ which corrects
	the areas bounded by the different regions in the projections and the
	angles of the strands at the singular points)
	is obtained  by performing the operations illustrated in
	\autoref{fig:legendrian-to-lagrangian} to the front diagram of
	$\Lambda$.
	\begin{figure}[!htb]
	    \centering
	    
    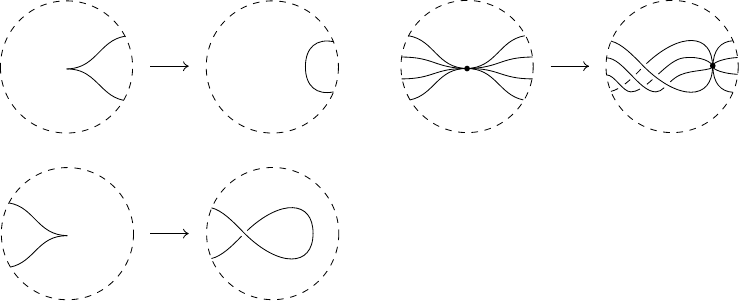

	    \caption{Obtaining the Lagrangian projection of an isotopic
		    Legendrian from the front projection. In the operation to the 
		    upper right, the singularity is allowed to have any
	    	valency, and the operation consists of performing a twist of the
    		strands to left of the singularity.}
	    \label{fig:legendrian-to-lagrangian}
	\end{figure}
\end{definition}
\begin{proposition}[\cite{ABS22}]
	The Ng resolution is well defined in the sense that, after a planar
	isotopy, the Lagrangian projection indeed lifts to a Legendrian.
	Moreover, every singular Legendrian in $\R^{3}$ is Legendrian isotopic
	to its Ng resolution after deforming the angles at the singularities.
\end{proposition}
\begin{proof}
	This is \cite[Lemma 2.2.16]{ABS22}.
\end{proof}

%% file: figures/singular-definition.pdf_tex
\begingroup%
  \makeatletter%
  \providecommand\color[2][]{%
    \errmessage{(Inkscape) Color is used for the text in Inkscape, but the package 'color.sty' is not loaded}%
    \renewcommand\color[2][]{}%
  }%
  \providecommand\transparent[1]{%
    \errmessage{(Inkscape) Transparency is used (non-zero) for the text in Inkscape, but the package 'transparent.sty' is not loaded}%
    \renewcommand\transparent[1]{}%
  }%
  \providecommand\rotatebox[2]{#2}%
  \newcommand*\fsize{\dimexpr\f@size pt\relax}%
  \newcommand*\lineheight[1]{\fontsize{\fsize}{#1\fsize}\selectfont}%
  \ifx\svgwidth\undefined%
    \setlength{\unitlength}{285.09281429bp}%
    \ifx\svgscale\undefined%
      \relax%
    \else%
      \setlength{\unitlength}{\unitlength * \real{\svgscale}}%
    \fi%
  \else%
    \setlength{\unitlength}{\svgwidth}%
  \fi%
  \global\let\svgwidth\undefined%
  \global\let\svgscale\undefined%
  \makeatother%
  \begin{picture}(1,0.43147167)%
    \lineheight{1}%
    \setlength\tabcolsep{0pt}%
    \put(0,0){\includegraphics[width=\unitlength,page=1]{singular-definition.pdf}}%
    \put(0.12677687,0.33781367){\color[rgb]{0,0,0}\makebox(0,0)[lt]{\lineheight{1.25}\smash{\begin{tabular}[t]{l}\color{blue}\footnotesize$\Sigma( \Lambda )$\end{tabular}}}}%
    \put(-0.04862647,0.2488417){\color[rgb]{0,0,0}\makebox(0,0)[lt]{\lineheight{1.25}\smash{\begin{tabular}[t]{l}\footnotesize$\partial W$\end{tabular}}}}%
    \put(0.1557651,0.08413319){\color[rgb]{0,0,0}\makebox(0,0)[lt]{\lineheight{1.25}\smash{\begin{tabular}[t]{l}\footnotesize$W$\end{tabular}}}}%
    \put(0.65896484,0.08413319){\color[rgb]{0,0,0}\makebox(0,0)[lt]{\lineheight{1.25}\smash{\begin{tabular}[t]{l}\footnotesize$\R \times ( \R \times V )$\end{tabular}}}}%
  \end{picture}%
\endgroup%

%% file: figures/pinched-figure-eight.pdf_tex
\begingroup%
  \makeatletter%
  \providecommand\color[2][]{%
    \errmessage{(Inkscape) Color is used for the text in Inkscape, but the package 'color.sty' is not loaded}%
    \renewcommand\color[2][]{}%
  }%
  \providecommand\transparent[1]{%
    \errmessage{(Inkscape) Transparency is used (non-zero) for the text in Inkscape, but the package 'transparent.sty' is not loaded}%
    \renewcommand\transparent[1]{}%
  }%
  \providecommand\rotatebox[2]{#2}%
  \newcommand*\fsize{\dimexpr\f@size pt\relax}%
  \newcommand*\lineheight[1]{\fontsize{\fsize}{#1\fsize}\selectfont}%
  \ifx\svgwidth\undefined%
    \setlength{\unitlength}{454.30413194bp}%
    \ifx\svgscale\undefined%
      \relax%
    \else%
      \setlength{\unitlength}{\unitlength * \real{\svgscale}}%
    \fi%
  \else%
    \setlength{\unitlength}{\svgwidth}%
  \fi%
  \global\let\svgwidth\undefined%
  \global\let\svgscale\undefined%
  \makeatother%
  \begin{picture}(1,0.2966216)%
    \lineheight{1}%
    \setlength\tabcolsep{0pt}%
    \put(0,0){\includegraphics[width=\unitlength,page=1]{pinched-figure-eight.pdf}}%
    \put(0.22968764,0.25032748){\color[rgb]{0,0,0}\makebox(0,0)[lt]{\lineheight{1.25}\smash{\begin{tabular}[t]{l}$a_1$\end{tabular}}}}%
    \put(0.22965393,0.03973315){\color[rgb]{0,0,0}\makebox(0,0)[lt]{\lineheight{1.25}\smash{\begin{tabular}[t]{l}$a_2$\end{tabular}}}}%
    \put(0.22965393,0.14528719){\color[rgb]{0,0,0}\makebox(0,0)[lt]{\lineheight{1.25}\smash{\begin{tabular}[t]{l}$b$\end{tabular}}}}%
    \put(0.18191419,0.18618633){\color[rgb]{0,0,0}\makebox(0,0)[lt]{\lineheight{1.25}\smash{\begin{tabular}[t]{l}$t$\end{tabular}}}}%
    \put(0.17907214,0.10073597){\color[rgb]{0,0,0}\makebox(0,0)[lt]{\lineheight{1.25}\smash{\begin{tabular}[t]{l}$d$\end{tabular}}}}%
    \put(0.13514277,0.14528719){\color[rgb]{0,0,0}\makebox(0,0)[lt]{\lineheight{1.25}\smash{\begin{tabular}[t]{l}$c_1$\end{tabular}}}}%
    \put(0.0406402,0.14504431){\color[rgb]{0,0,0}\makebox(0,0)[lt]{\lineheight{1.25}\smash{\begin{tabular}[t]{l}$c_2$\end{tabular}}}}%
    \put(0.70729705,0.26409309){\color[rgb]{0,0,0}\makebox(0,0)[lt]{\lineheight{1.25}\smash{\begin{tabular}[t]{l}\footnotesize$1$\end{tabular}}}}%
    \put(0.94875673,0.0905119){\color[rgb]{0,0,0}\makebox(0,0)[lt]{\lineheight{1.25}\smash{\begin{tabular}[t]{l}\footnotesize$-1$\end{tabular}}}}%
    \put(0.84576045,0.09404591){\color[rgb]{0,0,0}\makebox(0,0)[lt]{\lineheight{1.25}\smash{\begin{tabular}[t]{l}\footnotesize$-1$\end{tabular}}}}%
    \put(0.94656457,0.21340517){\color[rgb]{0,0,0}\makebox(0,0)[lt]{\lineheight{1.25}\smash{\begin{tabular}[t]{l}\footnotesize$0$\end{tabular}}}}%
    \put(0.83520382,0.18411337){\color[rgb]{0,0,0}\makebox(0,0)[lt]{\lineheight{1.25}\smash{\begin{tabular}[t]{l}\footnotesize$0$\end{tabular}}}}%
    \put(0.92435099,0.16392324){\color[rgb]{0,0,0}\makebox(0,0)[lt]{\lineheight{1.25}\smash{\begin{tabular}[t]{l}\footnotesize$0$\end{tabular}}}}%
    \put(0.69949935,0.01659539){\color[rgb]{0,0,0}\makebox(0,0)[lt]{\lineheight{1.25}\smash{\begin{tabular}[t]{l}\footnotesize$-2$\end{tabular}}}}%
  \end{picture}%
\endgroup%

%% file: figures/sing-legendrian-reidemeister.pdf_tex
\begingroup%
  \makeatletter%
  \providecommand\color[2][]{%
    \errmessage{(Inkscape) Color is used for the text in Inkscape, but the package 'color.sty' is not loaded}%
    \renewcommand\color[2][]{}%
  }%
  \providecommand\transparent[1]{%
    \errmessage{(Inkscape) Transparency is used (non-zero) for the text in Inkscape, but the package 'transparent.sty' is not loaded}%
    \renewcommand\transparent[1]{}%
  }%
  \providecommand\rotatebox[2]{#2}%
  \newcommand*\fsize{\dimexpr\f@size pt\relax}%
  \newcommand*\lineheight[1]{\fontsize{\fsize}{#1\fsize}\selectfont}%
  \ifx\svgwidth\undefined%
    \setlength{\unitlength}{355.46604187bp}%
    \ifx\svgscale\undefined%
      \relax%
    \else%
      \setlength{\unitlength}{\unitlength * \real{\svgscale}}%
    \fi%
  \else%
    \setlength{\unitlength}{\svgwidth}%
  \fi%
  \global\let\svgwidth\undefined%
  \global\let\svgscale\undefined%
  \makeatother%
  \begin{picture}(1,0.64330392)%
    \lineheight{1}%
    \setlength\tabcolsep{0pt}%
    \put(0,0){\includegraphics[width=\unitlength,page=1]{sing-legendrian-reidemeister.pdf}}%
    \put(0.22090751,0.56842621){\color[rgb]{0,0,0}\makebox(0,0)[lt]{\lineheight{1.25}\smash{\begin{tabular}[t]{l}I\end{tabular}}}}%
    \put(0.76037983,0.33681141){\color[rgb]{0,0,0}\makebox(0,0)[lt]{\lineheight{1.25}\smash{\begin{tabular}[t]{l}V\end{tabular}}}}%
    \put(0.75338624,0.10534578){\color[rgb]{0,0,0}\makebox(0,0)[lt]{\lineheight{1.25}\smash{\begin{tabular}[t]{l}VI\end{tabular}}}}%
    \put(0.20824806,0.10575925){\color[rgb]{0,0,0}\makebox(0,0)[lt]{\lineheight{1.25}\smash{\begin{tabular}[t]{l}III\end{tabular}}}}%
    \put(0.21519589,0.33669593){\color[rgb]{0,0,0}\makebox(0,0)[lt]{\lineheight{1.25}\smash{\begin{tabular}[t]{l}II\end{tabular}}}}%
    \put(0.75606115,0.56791284){\color[rgb]{0,0,0}\makebox(0,0)[lt]{\lineheight{1.25}\smash{\begin{tabular}[t]{l}IV\end{tabular}}}}%
    \put(0,0){\includegraphics[width=\unitlength,page=2]{sing-legendrian-reidemeister.pdf}}%
  \end{picture}%
\endgroup%

%% file: figures/weinstein-reidemeister.pdf_tex
\begingroup%
  \makeatletter%
  \providecommand\color[2][]{%
    \errmessage{(Inkscape) Color is used for the text in Inkscape, but the package 'color.sty' is not loaded}%
    \renewcommand\color[2][]{}%
  }%
  \providecommand\transparent[1]{%
    \errmessage{(Inkscape) Transparency is used (non-zero) for the text in Inkscape, but the package 'transparent.sty' is not loaded}%
    \renewcommand\transparent[1]{}%
  }%
  \providecommand\rotatebox[2]{#2}%
  \newcommand*\fsize{\dimexpr\f@size pt\relax}%
  \newcommand*\lineheight[1]{\fontsize{\fsize}{#1\fsize}\selectfont}%
  \ifx\svgwidth\undefined%
    \setlength{\unitlength}{355.26236226bp}%
    \ifx\svgscale\undefined%
      \relax%
    \else%
      \setlength{\unitlength}{\unitlength * \real{\svgscale}}%
    \fi%
  \else%
    \setlength{\unitlength}{\svgwidth}%
  \fi%
  \global\let\svgwidth\undefined%
  \global\let\svgscale\undefined%
  \makeatother%
  \begin{picture}(1,0.40734333)%
    \lineheight{1}%
    \setlength\tabcolsep{0pt}%
    \put(0,0){\includegraphics[width=\unitlength,page=1]{weinstein-reidemeister.pdf}}%
  \end{picture}%
\endgroup%

%% file: figures/legendrian-to-lagrangian.pdf_tex
\begingroup%
  \makeatletter%
  \providecommand\color[2][]{%
    \errmessage{(Inkscape) Color is used for the text in Inkscape, but the package 'color.sty' is not loaded}%
    \renewcommand\color[2][]{}%
  }%
  \providecommand\transparent[1]{%
    \errmessage{(Inkscape) Transparency is used (non-zero) for the text in Inkscape, but the package 'transparent.sty' is not loaded}%
    \renewcommand\transparent[1]{}%
  }%
  \providecommand\rotatebox[2]{#2}%
  \newcommand*\fsize{\dimexpr\f@size pt\relax}%
  \newcommand*\lineheight[1]{\fontsize{\fsize}{#1\fsize}\selectfont}%
  \ifx\svgwidth\undefined%
    \setlength{\unitlength}{354.54120065bp}%
    \ifx\svgscale\undefined%
      \relax%
    \else%
      \setlength{\unitlength}{\unitlength * \real{\svgscale}}%
    \fi%
  \else%
    \setlength{\unitlength}{\svgwidth}%
  \fi%
  \global\let\svgwidth\undefined%
  \global\let\svgscale\undefined%
  \makeatother%
  \begin{picture}(1,0.40634165)%
    \lineheight{1}%
    \setlength\tabcolsep{0pt}%
    \put(0,0){\includegraphics[width=\unitlength,page=1]{legendrian-to-lagrangian.pdf}}%
  \end{picture}%
\endgroup%

%% file: stoppedsubalg.tex
\section{Stopped subalgebras} 
\label{sec:stopped}
In this section we prove our main result, which provides simplified models of
the singular Chekanov--Eliashberg dg-algebra in $\R^{3}$.
\subsection{Bordered Legendrians}
We will be working with bordered Legendrians, of which we here give an
overview. A more detailed account of this topic can
be found in \cite{ABS22}.   

\begin{definition}
A \emph{bordered Legendrian} in $\R^{3}$ is a subset
$\Gamma \subset  \R^{3}$ of the form $\Gamma = \Lambda \cap \{ x \in \R^{3} : | x | \leq M\}$
where $\Lambda$ is a proper embedding of a singular Legendrian skeleton of a Weinstein hypersector 
$V \subset \R^{3}$ such that for $| x | \geq M$, each point of $\Lambda$ has 
$y$-coordinate equal to $0$ (possibly with finitely many two-valent singularities). If $\Lambda$ does not have any singularities
inside $\{  x \in \R^{3} : | x | \leq M\}$ we call $\Gamma$ \emph{smooth} and
otherwise we say that
$\Gamma$ is \emph{singular}. If $\{ x \in \R^{3} : x=-M\} \cap \Gamma = \emptyset$ or 
$\{ x \in \R^{3} : x=M\} \cap \Gamma = \emptyset$ we say that $\Gamma$ has no
\emph{left ends} or \emph{right ends}, respectively.
\end{definition}

\begin{definition}
Let $\Gamma$ be a bordered Legendrian with no left ends. Let $\Lambda'$
be the singular Legendrian obtained by closing up the left ends of $\Gamma$
into a singularity consisting of a single $0$-cell, as shown to the left in
\autoref{fig:closing}. Let $V'$ be the
thickening of $\Lambda'$ to a Weinstein surface with skeleton $\Lambda'$.
The \emph{bordered Chekanov--Eliashberg dg-algebra} of $\Gamma$,
denoted by $CE^{*}( \Gamma;V_{0};\R^{4} )$, is the unital dg-subalgebra of 
the Chekanov--Eliashberg dg-algebra $CE^{*}( \Lambda';V_{0}';\R^{4} )$
generated by all chords which through the Ng resolution correspond to crossings
in the front of $\Gamma$, and the chords of $\partial \Lambda'$ which belong to
singularities of $\Gamma$. 
\begin{figure}[!htb]
    \centering
    
    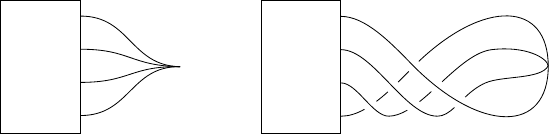

    \caption{The closure $\Lambda'$ of $\Lambda$, illustrated in the front
    projection to the left and the Lagrangian projection to the right.}
    \label{fig:closing}
\end{figure}
\end{definition}
\begin{lemma}
	The algebra $CE^{*}( \Gamma;V_{0};\R^{4} )$ is well-defined in the sense
	that it is closed under the differential.
\end{lemma}
\begin{proof}
	Using the Ng resolution, one sees that the chords of $\Lambda'$
	which are excluded from $CE^{*}( \Gamma;V_{0};\R^{4} )$ are all of
	higher action than the ones included, and can therefore not occur in
	the differential of any chord in $CE^{*}( \Gamma;V_{0};\R^{4} )$. The
	chords of $\partial \Lambda'$ at the singularity in
	\autoref{fig:closing} can not occur in the differential as no
	disks contributing to the differential of $CE^{*}( \Gamma;V_{0};\R^{4}
	)$ can cross the twist to the right in \autoref{fig:closing}.
\end{proof}

This algebra was first introduced in \cite{Siv11} for smooth bordered Legendrians and
generalized to singular bordered Legendrians in \cite{ABS22}. 
In \cite{ABS22}, it is shown that
$CE^{*}(\Gamma;V_{0};\R^{4} )$ is invariant up to quasi-isomorphism
under Legendrian isotopy of $\Gamma$, by which we mean an isotopy of $\Lambda$ which is
constant outside $\{  x \in \R^{3} : | x | < M\}$. 

\begin{theorem}[{\cite[Theorem 3.3.15]{ABS22}}]
	The dg-algebra $CE^{*}( \Gamma;V_{0};\R^{4})$ is up to quasi-isomorphism
	invariant under Legendrian isotopy of $\Gamma$ supported in some fixed
	compact subset.
\end{theorem}
\begin{proof}
	See {\cite[Theorem 3.3.15]{ABS22}}.
\end{proof}

\begin{definition}
Let $\Gamma$ be a bordered Legendrian. The \emph{reflection} of
$\Gamma$ is the bordered Legendrian ${}^{*}\Gamma$ whose front is the
reflection of the front of $\Gamma$ over the horizontal axis.
\end{definition}

\begin{lemma}
\label{res:push-twist}
	Let $\Gamma$ be a bordered Legendrian. Let $\Lambda$
	be a Legendrian whose front contains $\Gamma$ as a
	tangle, along with a positive or negative half-twist 
	as illustrated to the left in \autoref{fig:half-twists}. 
	Then the operations 
	in \autoref{fig:half-twists} preserve the Legendrian isotopy type of
	$\Lambda$.
\begin{figure}[!htb]
    \centering
    
    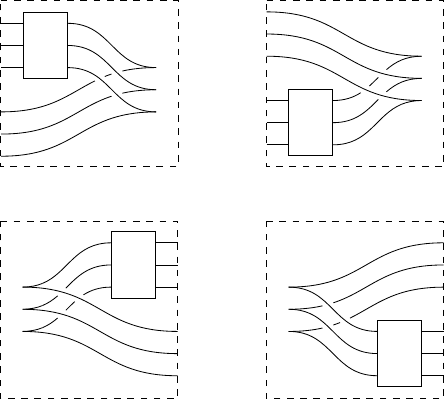

    \caption{Pushing $\Gamma$ through the cusp of a half-twist.}
    \label{fig:half-twists}
\end{figure}
\end{lemma}
\begin{remark}
	Note that $\Gamma$ need not have the same number of left ends as right
	ends, or may lack left or right ends altogether. 
\end{remark}
\begin{proof}	
	We perform Reidemeister moves and push the left cusps, right cusps,
	crossings, and singularities of $\Gamma$ through the half-twist, as shown in
	\autoref{fig:elementary-permutation}. Doing this turns left cusps into
	right cusps and vice versa, flips the singularities, and reverses the
	order, so we get the bordered Legendrian ${}^{*}\Gamma$ on the other
	side. See the proof of \cite[Lemma 5.1.4]{Ng01} for a similar
	construction.
\begin{figure}[!htb]
    \centering
    
    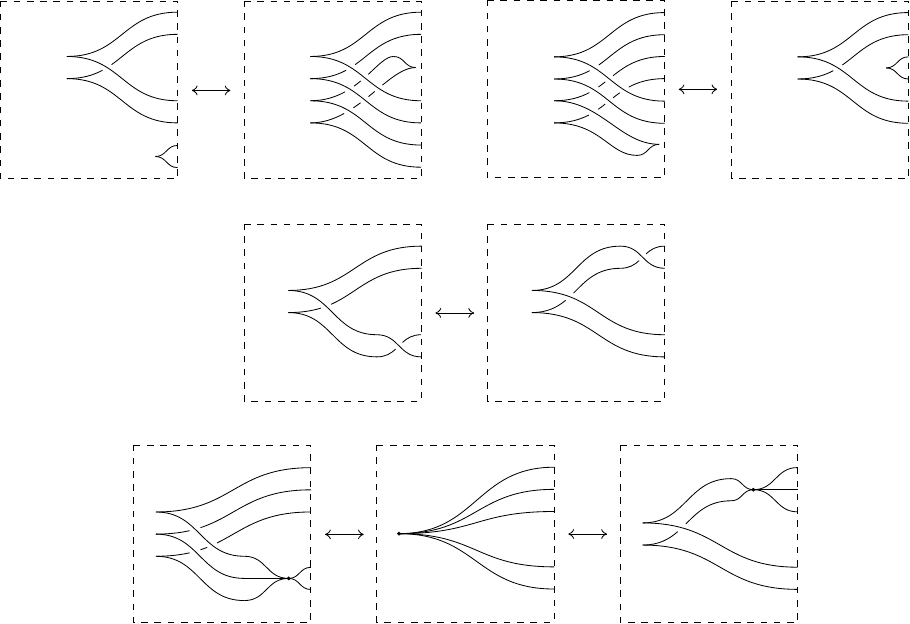

    \caption{Pushing cusps, crossings, and singularities through a half-twist.}
    \label{fig:elementary-permutation}
\end{figure}
\end{proof}

\subsection{Stopped subalgebras}
Let $\Lambda$ be the singular Legendrian skeleton of some Weinstein
hypersurface $V \subset \R^{3}$. 

\begin{definition}
If there are two values $x_0 < x_1$ such that all singularities of $\Lambda$ have
$x$-coordinates equal to either $x_0$ or $x_1$ and all other points of $\Lambda$ have
$x$-coordinates in the interval $( x_0,x_1 )$ we say that $\Lambda$ has 
\emph{all singularities to the sides}. We call the singularities with
$x$-coordinate equal to $x_0$ \emph{left singularities} and those with
$x$-coordinate equal to $x_{1}$ \emph{right singularities}.
\end{definition}

Every singular Legendrian in $\R^{3}$ is isotopic to a Legendrian with all
singularities to the sides and we now assume that $\Lambda$ is in such a
position. We construct a bordered Legendrian $\Lambda
\cup \Omega$ by attaching one strand $\Omega_{i}$ to each left singularity of
$\Lambda$, numbering from bottom to top, as shown in
\autoref{fig:exact-stops}. 
\begin{figure}[!htb]
    \centering
    
    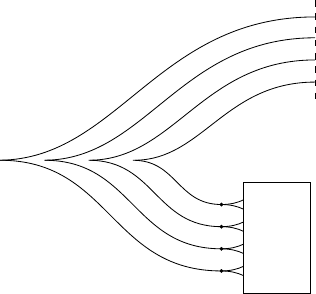

    \caption{Adding strands $\Omega = \Omega_{1} \cup
	    \Omega_{2} \cup \ldots \cup \Omega_{k}$ to $\Lambda$. There are no
    Reeb chords of $\Lambda \cup \Omega$ with endpoints on $\Omega$.}
    \label{fig:exact-stops}
\end{figure}
Note that $\partial \Lambda$ consists of two-point
spheres and $\partial \Omega$ of one-point stops.
The bordered Legendrian $\Lambda \cup \Omega$ has the property that there are
no Reeb chords
of $\Lambda \cup \Omega$ with an endpoint on $\Omega$ in the algebra 
$CE^{*}( \Lambda \cup \Omega;V_{0};\R^{4} )$. Note however that there are Reeb
chords of $\partial \Lambda \cup \partial \Omega$ with endpoints on $\partial
\Omega$.

\begin{definition}
The \emph{stopped subalgebra} of $CE^{*}( \Lambda \cup \Omega;V_{0};\R^{4} )$, denoted by 
$CE^{*}(\Lambda;V_{0,\partial \Omega};\R^{4} )$, is the unital dg-subalgebra generated by all
chords of $\Lambda$, and the chords of $\partial \Lambda$ which do not pass
though $\partial \Omega$.
\end{definition}
\begin{lemma}
	The stopped subalgebra is a subcomplex, and hence itself a unital
	dg-algebra.
\end{lemma}
\begin{proof}
	By \autoref{res:zero-dim-diff}, the differential of a chord of
	$\partial \Lambda \cup \partial \Omega$ which does not pass through
	$\partial \Omega$, nor has any endpoint on $\partial \Omega$, 
	will not contain any chords which
	pass through or have endpoints on $\partial \Omega$. For
	a chord $c$ of $\Lambda$, it is clear from the Ng resolution that 
	$\partial c$ does not contain any chords of $\partial \Lambda \cup
	\partial \Omega$ which pass through or have endpoints on $\partial
	\Omega$
\end{proof}
If the base points of $\partial V_{0}$ are placed near $\partial \Omega$ then 
the chords of $\partial \Lambda$ which do not pass though $\partial \Omega$
are precisely the chords of the form $c_{ij}^{0}$. In particular, $CE^{*}(
\Lambda \cup \Pi;V_{0};W )$ is finitely generated. Note that $CE^{*}(
\Lambda;V_{0,\partial \Omega};W )$ also embeds as subalgebra of $CE^{*}(
\Lambda;V_{0};W )$.

The motivation for the notation comes from considering $CE^{*}(
\Lambda;V_{0,\partial \Omega};W )$ as being a
version of the relative algebra of \autoref{ssec:ceforsing}, 
but where $\Sigma( \Sigma )$ is a
non-compact stop at which we attach a 'non-compact half-handle', as shown in 
\autoref{fig:singular-definition-stop}. 
\begin{figure}[!htb]
    \centering
    
    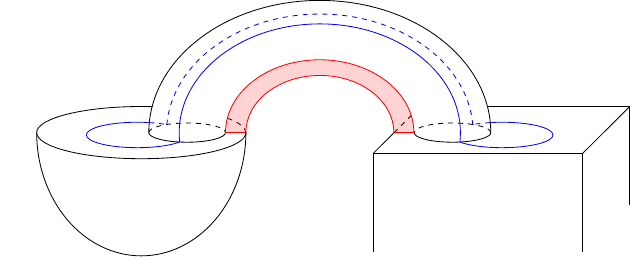

    \caption{The cobordism $W_{V,\Sigma(\Omega)}^{0}$, where $L_{\Sigma( \Omega )}$ 
    is the core of the half-handle attached at $\Sigma(\Omega)$.}
    \label{fig:singular-definition-stop}
\end{figure}
We do not give the full details of the
geometric construction, but note that attaching the non-compact half-handle
corresponds to putting a stop diffeomorphic to a Legendrian arc in the boundary
of the Weinstein cobordism $W^0_V$.

\begin{lemma}
\label{res:stopped-subalgebra-prop}
	The canonical inclusion	
	\[
		CE^{*}( \Lambda;V_{0,\partial \Omega};\R^{4} ) 
		\hookrightarrow
		CE^{*}( \Lambda \cup \Omega;V_{0};\R^{4} )
	\]
	is a quasi-isomorphism onto 
	$CE^{*}( \Lambda \cup \Omega;V_{0};\R^{4} )[\Lambda,\Lambda]$.	
\end{lemma}
\begin{remark}
	Note that this lemma says 
	something different than \autoref{res:sing-cobordism-map}, the reason being 
	that $V_{0,\partial\Omega}$ is a Weinstein hypersector and not a hypersurface.
\end{remark}
\begin{proof}
	We will use an argument similar to that in the proof of
	\autoref{res:surgery-map-surface}.
	We filter $CE^{*}( \Lambda \cup \Omega;V_{0};\R^{4} )$ by Reeb
	chord action and let $\widetilde{\partial}$ be the action preserving
	component of the differential $\partial$. Since $\partial$ is strictly
	action decreasing on
	chords of $\Lambda$, $\widetilde{\partial}$ will act as
	$\partial_{0}$ from \autoref{res:zero-dim-diff} on chords of $\partial
	\Lambda \cup \partial \Omega$ and vanish on chords of $\Lambda \cup
	\Omega$. Let 
\[
	\tilde{\iota}:(CE^{*}( \Lambda;V_{0,\partial 
	\Omega};\R^{4} ),\widetilde{\partial}) \to 
	(CE^{*}( \Lambda \cup \Omega;V_{0};\R^{4}),\widetilde{\partial}) 
\] 
	be the inclusion considered as going between the same algebras as
	$\iota$ but with $\widetilde{\partial}$ as differential.
	It follows from \autoref{res:short-chords-generate} that the inclusion
	\[
		(CE^{*}( \partial \Lambda;V_{0,\partial \Omega}
		),\partial_{0} )
		\hookrightarrow
		(CE^{*}( \partial \Lambda \cup \partial \Omega;V_{0}
		),\partial_{0})
	\]
	is a quasi-isomorphism onto 
	$(CE^{*}( \partial \Lambda \cup \partial \Omega;V_{0}
	),\partial_{0})[\partial \Lambda,\partial \Lambda]$. 
	Since $\tilde{\iota}$ simply extends this inclusion by the identity on
	the remaining generators, it follows that
	the whole inclusion $\tilde{\iota}$ is a
	quasi-isomorphism onto $(CE^{*}( \Lambda \cup \Omega;V_{0};\R^{4}
	)[\Lambda,\Lambda],\widetilde{\partial})$. If we
	consider the mapping cone of $\iota$ we see that the
	first page of the spectral sequence arising from the action filtration
	of the cone is isomorphic to the mapping cone of
	$\tilde{\iota}$. Consequently, the sequence vanishes on the second
	page and $\iota$ is a quasi-isomorphism
	onto $CE^{*}( \Lambda \cup \Omega;V_{0};\R^{4}
	)[\Lambda,\Lambda]$. 
\end{proof}
\begin{lemma}
\label{res:stopped-lemma}
	There is a quasi-isomorphism
	\[
		CE^{*}( \Lambda \cup \Omega;V_{0};\R^{4} ) \to
		CE^{*}( \Lambda;V_{0};\R^{4} ),
	\] 
	and moreover, the restriction of this 
	quasi-isomorphism to $CE^{*}( \Lambda \cup \Omega;V_{0};\R^{4}
	)[\Lambda,\Lambda]$ is also a quasi-isomorphism.
\end{lemma}
\begin{proof}
	First, we order the strands of $\Omega$ as
	$\Omega_{1},\ldots,\Omega_{k}$ going from top to bottom. 
	By applying \autoref{res:push-twist} twice we can perform an isotopy of 
	$\Lambda \cup \Omega$ to get a bordered Legendrian of the form shown
	in \autoref{fig:exact-stops-isotopy}. 
\begin{figure}[!htb]
    \centering
    
    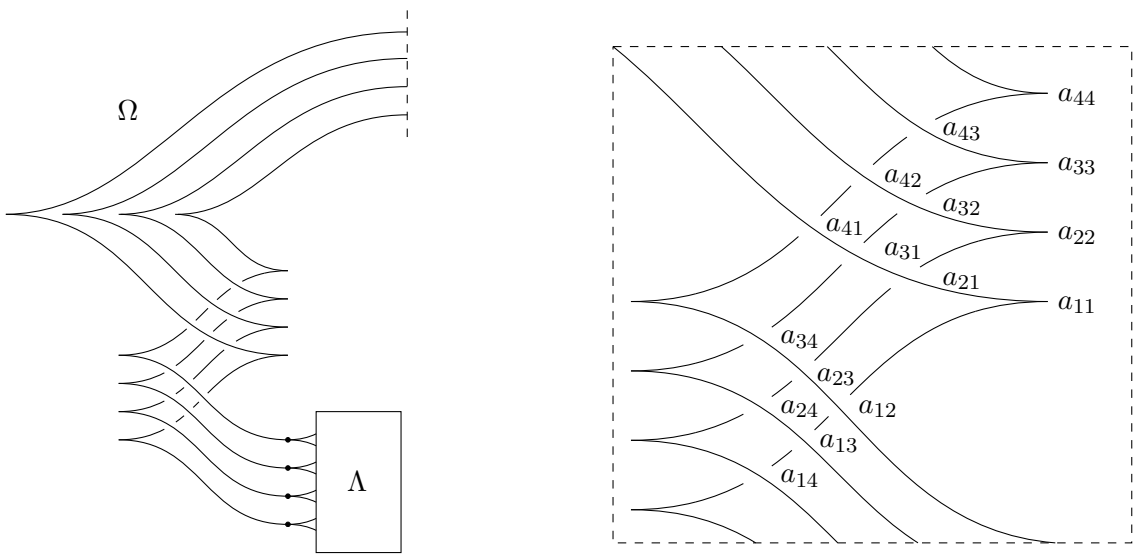

    \caption{The bordered Legendrian $\Lambda \cup \Omega$ after a Legendrian
	    isotopy which introduces Reeb chords from $\Omega$ to itself.}
    \label{fig:exact-stops-isotopy}
\end{figure}
	This isotopy introduces one Reeb chord
	generator $a_{ij}$ from $\Omega_{i}$ to $\Omega_{j}$ 
	for each pair $1 \leq i,j \leq k$. The differential
	acts on these by $\partial a_{ij} = \pm \delta_{ij}e_{i} + 
	\sum_{m > i} \pm a_{mj}a_{im}$, where $e_{i}$ is the idempotent
	corresponding to $\Omega_{i}$. In particular, we have 
	$\partial a_{kk} = \pm e_{k}$. Moreover, $a_{kk}$ does
	not occur in the differential of any other chords. 
	By \autoref{res:exact-removal}, there is then a quasi-isomorphism 
	\[
		CE^{*}( \Lambda \cup \Omega;V_{0};\R^{4} ) \isomto
		CE^{*}( \Lambda \cup \Omega_{1} \cup \ldots \cup
		\Omega_{k-1};V_{0};\R^{4} ).
	\]
	By induction we then get the desired quasi-isomorphism 
	$CE^{*}( \Lambda \cup \Omega;V_{0};\R^{4} ) \isomto
	CE^{*}( \Lambda;V_{0};\R^{4} )$. 
	By construction (see the proof of
	\autoref{res:exact-removal}) all words with an endpoint on $\Omega$
	are sent to zero. As a chain complex,
	$CE^{*}( \Lambda \cup \Omega;V_{0};\R^{4} )$ splits into a direct sum of
	the subcomplex $CE^{*}( \Lambda \cup \Omega;V_{0};\R^{4}
	)[\Lambda,\Lambda]$ and the subcomplex of words with an endpoint on
	$\Omega$. This implies that the restriction to 
	$CE^{*}( \Lambda \cup \Omega;V_{0};\R^{4})[\Lambda,\Lambda]$ is also a
	quasi-isomorphism.
\end{proof}

Combining these results, we get that the stopped subalgebra is quasi-isomorphic
to the Chekanov--Eliashberg dg-algebra of $\Lambda$.

\begin{theorem}
\label{res:finite-stopped-models}
	Let $\Lambda \subset \R^{3}$ be a singular Legendrian. 
	Then the canonical inclusion 
	\[
		CE^{*}( \Lambda;V_{0,\partial \Omega};\R^{4})
		\hookrightarrow
		CE^{*}( \Lambda;V_{0};\R^{4} )
	\] 
	is a quasi-isomorphism.
\end{theorem}
\begin{proof}
	Composing the maps \autoref{res:stopped-subalgebra-prop} and
	\autoref{res:stopped-lemma} we get a quasi-isomorphism,
	\[
		CE^{*}( \Lambda;V_{0,\partial \Omega};\R^{4})
		\isomto
		CE^{*}( \Lambda \cup \Omega;V_{0};\R^{4} )[\Lambda,\Lambda]
		\isomto
		CE^{*}( \Lambda;V_{0};\R^{4} ),
	\]
	and it is clear from the construction of these maps that 
	the composition is the canonical inclusion.
\end{proof}
\begin{remark}
\label{rm:omega-subset}
	To simplify the exposition, we have only considered the case when we
	attach one strand $\Omega_i$ at each left singularity. However, one can
	equally well attach them only at a subset of the left singularities.
	If one does this and lets $\Omega$ be the union of the strands attached
	at this subset, one can in the same way define the stopped subalgebra 
	by excluding the chords of $\partial \Lambda$ passing though the $\partial \Omega$.
	The results \autoref{res:stopped-subalgebra-prop}, \autoref{res:stopped-lemma}, and
	\autoref{res:finite-stopped-models} can also be applied in this more
	general setting, and the proofs are
	the same.
\end{remark}

\subsection{Bordered Legendrians from opening up singular Legendrians}
The stopped subalgebra can be realized as the Chekanov--Eliashberg dg-algebra of 
a bordered Legendrian. This allows us to exploit the isotopy invariance and
further simplify the algebra.
\begin{definition}
\label{def:opening}
Let $\Lambda \subset \R^{3}$ be a singular Legendrian 
with all singularities to the sides. 
The \emph{opening} of $\Lambda$ is the
smooth bordered Legendrian $\Lambda^{\circ}$ obtained by removing the singularities of 
$\Lambda$ and separating the strands, as shown in 
\autoref{fig:singular-resolution}.
Given a left or right singularity $t$ of
$\Lambda$, we define $\Lambda^{\circ,t}$ to be the (possibly singular) bordered
Legendrian obtained by opening only at $t$.
\begin{figure}[!htb]
    \centering
    
    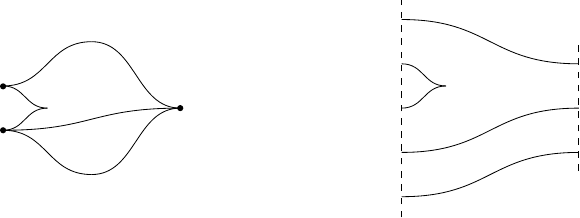

    \caption{The front projection of a singular Legendrian $\Lambda$ and its
    opening $\Lambda^{\circ}$.}
    \label{fig:opening}
\end{figure}
\end{definition}

\begin{definition}
\label{def:resolution}
Let $\Lambda \subset \R^{3}$ be a singular Legendrian 
with all singularities
to the sides. The \emph{resolution} of $\Lambda$ is the
smooth bordered Legendrian $\Lambda^{\bullet}$ constructed as shown in
\autoref{fig:singular-resolution}. It is obtained from the opening
$\Lambda^{\circ}$ by performing a negative half-twist of the ends corresponding
to each left singularity of $\Lambda$, and a positive half twist followed by a
negative half-twist of the ends corresponding to each right singularity, in
such a way that no new chords are introduced between ends corresponding to different
singularities. We are free to choose the order of the twists, as well as
whether they go up or down. Given a left or right singularity $t$ of
$\Lambda$ we define $\Lambda^{\bullet,t}$ to be the (possibly singular) 
bordered Legendrian obtained by resolving only at $t$.
\begin{figure}[!htb]
    \centering
    
    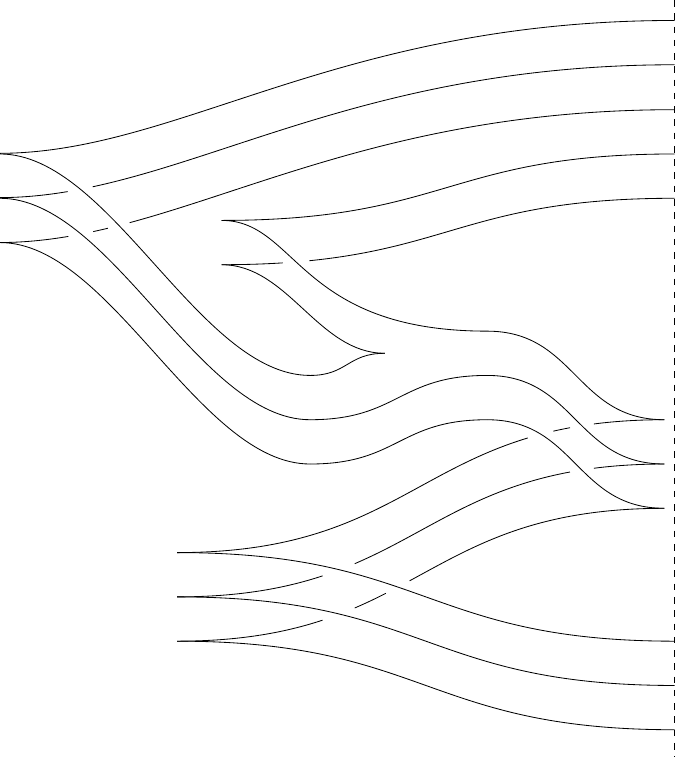

    \caption{The front projection of the resolution $\Lambda^{\bullet}$ of the
	    Legendrian $\Lambda$ in
    \autoref{fig:opening}.}
    \label{fig:singular-resolution}
\end{figure}
\end{definition}
\begin{remark}
	Note that due to the choices involved in the definition, the 
	resolution $\Lambda^{\bullet}$ is not canonically defined up
	to Legendrian isotopy. However, for any two different choices of order
	and direction of the twists in the construction of $\Lambda^{\bullet}$,
	there is a canonical identification of the
	respective Reeb chords. The Chekanov--Eliashberg dg-algebra 
	$CE^{*}( \Lambda^{\bullet};\R^{4} )$ is thus well-defined up to
	isomorphism.
\end{remark}
The idea of the following theorem is the easily verified fact that, in the case
when the singularities are stopped, one can replace the finitely many internal
generators by Reeb chords introduced by removing the singularities and wrapping
the strands.
\begin{theorem}
\label{res:resolution}
	Let $\Lambda \subset \R^{3}$ be a singular Legendrian 
	with all singularities to the sides. Then there is a
	quasi--isomorphism 
	\[
		CE^{*}( \Lambda;V_{0};\R^{4} ) \cong 
		CE^{*}( \Lambda^{\bullet};\R^{4} ).
	\]
	Moreover, if $t$ is an arbitrary singularity of $\Lambda$ there is a
	quasi-isomorphism
	\[
		CE^{*}( \Lambda;V_{0};\R^{4} ) \cong 
		CE^{*}( \Lambda^{\bullet,t};V^{\bullet,t}_{0};\R^{4} ).
	\]
\end{theorem}
\begin{proof}
	We prove the first quasi-isomorphism; the proof of the second is
	similar in light of \autoref{rm:omega-subset}.
	By performing a positive half twist at each right singularity of
	$\Lambda$, using Reidemeister VI moves, 
	we can obtain a Legendrian $\Lambda'$ with only left singularities such
	that $(\Lambda')^{\bullet} = \Lambda^{\bullet}$. Since
	$\Lambda$ and $\Lambda'$ are Legendrian isotopic by construction, 
	we get a quasi-isomorphism
	\[
		CE^{*}( \Lambda;V_{0};\R^{4} ) \cong 
		CE^{*}( \Lambda';V_{0};\R^{4} ).
	\]	
	It is thus sufficient to prove the theorem assuming that 
	$\Lambda$ only has left singularities.

	For each left singularity $t$ of valency $n$ of $\Lambda$ the stopped
	subalgebra $CE^{*}(\Lambda;V_{0,\partial \Omega};\R^{4})$ has one Reeb
	chord generator $t_{ij}^{0}$ for each $1 \leq i < j \leq n$, labeling
	counter-clockwise from a base point placed to the left.  The
	differential acts by
	\[
		\partial t_{ij}^{0} = \sum_{i < k < j} \pm t_{kj}^{0}t_{ik}^{0}.
	\] 
	Looking at the corresponding negative half-twist in
	$\Lambda^{\bullet}$ one sees that there is 
	one Reeb chord generator $a_{ij}$ for each $1 \leq i < j \leq n$, using
	the same labeling of the strands as for the points of $\partial
	\Lambda$. The differential applied to these chords are given by 
	 \[
		\partial a_{ij}= \sum_{i < k < j} \pm a_{ik}a_{kj},
	\]
	so the canonical identification of generators, mapping $a_{ij}$ to 
	$t_{ij}^{0}$ is a unital isomorphism of algebras. Using the Ng
	resolution to compute the differential it is clear that this is
	compatible with the differential, and we thus have a dg-algebra
	isomorphism 
	\[
		CE^{*}(\Lambda;V_{0,\partial \Omega};\R^{4}) \cong 
		CE^{*}( \Lambda^{\bullet};\R^{4} ).
	\] 
	The result now follows from
	\autoref{res:finite-stopped-models}.
\end{proof}
\begin{remark} 
	Any singular Legendrian can be moved by a Legendrian isotopy into a position with all
	singularities to the sides, so \autoref{res:resolution} produces finitely
	generated models for all singular Legendrians in $\R^{3}$. The algebra 
	$CE^{*}( \Lambda^{\bullet};\R^{4})$ will depend up to
	quasi-isomorphism on the choice of isotopy.
\end{remark}
\begin{theorem}
\label{res:opening}
	Let $\Lambda \subset \R^{3}$ be a singular Legendrian and let $t$ be
	an arbitrary right singularity of $\Lambda$. Then there is an isotopy
	\[
		\Lambda^{\bullet,t} \sim \Lambda^{\circ,t}
	\] 
	and in particular, there are then quasi-isomorphisms
	\[
		CE^{*}( \Lambda;V_{0};\R^{4} ) \cong 
		CE^{*}(\Lambda^{\bullet,t};\R^{4}) \cong
		CE^{*}(\Lambda^{\circ,t};\R^{4}).
	\]
\end{theorem}
\begin{proof}
	The resolution $\Lambda^{\bullet,t}$ is obtained from
	$\Lambda^{\circ,t}$ by performing two half-twists. By applying 
	\autoref{res:push-twist} (with $\Gamma = \Lambda^{\circ,t}$) we
	get an isotopy $\Lambda^{\bullet,t} \sim (^{*}\Lambda)^{\bullet,t}$. Doing
	this again (with $\Gamma = {}^{*} (\Lambda^{\circ,t} )$) we obtain
	an isotopy $(^{*}\Lambda)^{\bullet,t} \sim \Lambda^{\circ,t}$.
	Compare \autoref{fig:exact-stops} and \autoref{fig:exact-stops-isotopy}.
	The result then follows by \autoref{res:resolution}.
\end{proof}
It is important to note that the above theorem is not true if one opens at more
than one right singularity, as it is then not possible to perform the isotopy.

%% file: figures/closing.pdf_tex
\begingroup%
  \makeatletter%
  \providecommand\color[2][]{%
    \errmessage{(Inkscape) Color is used for the text in Inkscape, but the package 'color.sty' is not loaded}%
    \renewcommand\color[2][]{}%
  }%
  \providecommand\transparent[1]{%
    \errmessage{(Inkscape) Transparency is used (non-zero) for the text in Inkscape, but the package 'transparent.sty' is not loaded}%
    \renewcommand\transparent[1]{}%
  }%
  \providecommand\rotatebox[2]{#2}%
  \newcommand*\fsize{\dimexpr\f@size pt\relax}%
  \newcommand*\lineheight[1]{\fontsize{\fsize}{#1\fsize}\selectfont}%
  \ifx\svgwidth\undefined%
    \setlength{\unitlength}{264.97974546bp}%
    \ifx\svgscale\undefined%
      \relax%
    \else%
      \setlength{\unitlength}{\unitlength * \real{\svgscale}}%
    \fi%
  \else%
    \setlength{\unitlength}{\svgwidth}%
  \fi%
  \global\let\svgwidth\undefined%
  \global\let\svgscale\undefined%
  \makeatother%
  \begin{picture}(1,0.24157684)%
    \lineheight{1}%
    \setlength\tabcolsep{0pt}%
    \put(0,0){\includegraphics[width=\unitlength,page=1]{closing.pdf}}%
    \put(0.06167841,0.10380619){\color[rgb]{0,0,0}\makebox(0,0)[lt]{\lineheight{1.25}\smash{\begin{tabular}[t]{l}$\Gamma$\end{tabular}}}}%
    \put(0.53358731,0.10380635){\color[rgb]{0,0,0}\makebox(0,0)[lt]{\lineheight{1.25}\smash{\begin{tabular}[t]{l}$\Gamma$\end{tabular}}}}%
    \put(0,0){\includegraphics[width=\unitlength,page=2]{closing.pdf}}%
  \end{picture}%
\endgroup%

%% file: figures/half-twists.pdf_tex
\begingroup%
  \makeatletter%
  \providecommand\color[2][]{%
    \errmessage{(Inkscape) Color is used for the text in Inkscape, but the package 'color.sty' is not loaded}%
    \renewcommand\color[2][]{}%
  }%
  \providecommand\transparent[1]{%
    \errmessage{(Inkscape) Transparency is used (non-zero) for the text in Inkscape, but the package 'transparent.sty' is not loaded}%
    \renewcommand\transparent[1]{}%
  }%
  \providecommand\rotatebox[2]{#2}%
  \newcommand*\fsize{\dimexpr\f@size pt\relax}%
  \newcommand*\lineheight[1]{\fontsize{\fsize}{#1\fsize}\selectfont}%
  \ifx\svgwidth\undefined%
    \setlength{\unitlength}{213.18031131bp}%
    \ifx\svgscale\undefined%
      \relax%
    \else%
      \setlength{\unitlength}{\unitlength * \real{\svgscale}}%
    \fi%
  \else%
    \setlength{\unitlength}{\svgwidth}%
  \fi%
  \global\let\svgwidth\undefined%
  \global\let\svgscale\undefined%
  \makeatother%
  \begin{picture}(1,0.89854764)%
    \lineheight{1}%
    \setlength\tabcolsep{0pt}%
    \put(0,0){\includegraphics[width=\unitlength,page=1]{half-twists.pdf}}%
    \put(0.08816361,0.77510932){\color[rgb]{0,0,0}\makebox(0,0)[lt]{\lineheight{1.25}\smash{\begin{tabular}[t]{l}$\Gamma$\end{tabular}}}}%
    \put(0.28604201,0.28086576){\color[rgb]{0,0,0}\makebox(0,0)[lt]{\lineheight{1.25}\smash{\begin{tabular}[t]{l}$\Gamma$\end{tabular}}}}%
    \put(0.66444651,0.60823938){\color[rgb]{0,0,0}\makebox(0,0)[lt]{\lineheight{1.25}\smash{\begin{tabular}[t]{l}${}^{*}\Gamma$\end{tabular}}}}%
    \put(0.86426754,0.08825736){\color[rgb]{0,0,0}\makebox(0,0)[lt]{\lineheight{1.25}\smash{\begin{tabular}[t]{l}${}^{*}\Gamma$\end{tabular}}}}%
    \put(0,0){\includegraphics[width=\unitlength,page=2]{half-twists.pdf}}%
  \end{picture}%
\endgroup%

%% file: figures/elementary-permutation.pdf_tex
\begingroup%
  \makeatletter%
  \providecommand\color[2][]{%
    \errmessage{(Inkscape) Color is used for the text in Inkscape, but the package 'color.sty' is not loaded}%
    \renewcommand\color[2][]{}%
  }%
  \providecommand\transparent[1]{%
    \errmessage{(Inkscape) Transparency is used (non-zero) for the text in Inkscape, but the package 'transparent.sty' is not loaded}%
    \renewcommand\transparent[1]{}%
  }%
  \providecommand\rotatebox[2]{#2}%
  \newcommand*\fsize{\dimexpr\f@size pt\relax}%
  \newcommand*\lineheight[1]{\fontsize{\fsize}{#1\fsize}\selectfont}%
  \ifx\svgwidth\undefined%
    \setlength{\unitlength}{436.37851024bp}%
    \ifx\svgscale\undefined%
      \relax%
    \else%
      \setlength{\unitlength}{\unitlength * \real{\svgscale}}%
    \fi%
  \else%
    \setlength{\unitlength}{\svgwidth}%
  \fi%
  \global\let\svgwidth\undefined%
  \global\let\svgscale\undefined%
  \makeatother%
  \begin{picture}(1,0.68528091)%
    \lineheight{1}%
    \setlength\tabcolsep{0pt}%
    \put(0,0){\includegraphics[width=\unitlength,page=1]{elementary-permutation.pdf}}%
  \end{picture}%
\endgroup%

%% file: figures/exact-stops.pdf_tex
\begingroup%
  \makeatletter%
  \providecommand\color[2][]{%
    \errmessage{(Inkscape) Color is used for the text in Inkscape, but the package 'color.sty' is not loaded}%
    \renewcommand\color[2][]{}%
  }%
  \providecommand\transparent[1]{%
    \errmessage{(Inkscape) Transparency is used (non-zero) for the text in Inkscape, but the package 'transparent.sty' is not loaded}%
    \renewcommand\transparent[1]{}%
  }%
  \providecommand\rotatebox[2]{#2}%
  \newcommand*\fsize{\dimexpr\f@size pt\relax}%
  \newcommand*\lineheight[1]{\fontsize{\fsize}{#1\fsize}\selectfont}%
  \ifx\svgwidth\undefined%
    \setlength{\unitlength}{151.45558395bp}%
    \ifx\svgscale\undefined%
      \relax%
    \else%
      \setlength{\unitlength}{\unitlength * \real{\svgscale}}%
    \fi%
  \else%
    \setlength{\unitlength}{\svgwidth}%
  \fi%
  \global\let\svgwidth\undefined%
  \global\let\svgscale\undefined%
  \makeatother%
  \begin{picture}(1,0.93024509)%
    \lineheight{1}%
    \setlength\tabcolsep{0pt}%
    \put(0,0){\includegraphics[width=\unitlength,page=1]{exact-stops.pdf}}%
    \put(0.84956713,0.15444858){\color[rgb]{0,0,0}\makebox(0,0)[lt]{\lineheight{1.25}\smash{\begin{tabular}[t]{l}$\Lambda$\end{tabular}}}}%
    \put(0.27751066,0.65653802){\color[rgb]{0,0,0}\makebox(0,0)[lt]{\lineheight{1.25}\smash{\begin{tabular}[t]{l}$\Omega$\end{tabular}}}}%
  \end{picture}%
\endgroup%

%% file: figures/singular-definition-stop.pdf_tex
\begingroup%
  \makeatletter%
  \providecommand\color[2][]{%
    \errmessage{(Inkscape) Color is used for the text in Inkscape, but the package 'color.sty' is not loaded}%
    \renewcommand\color[2][]{}%
  }%
  \providecommand\transparent[1]{%
    \errmessage{(Inkscape) Transparency is used (non-zero) for the text in Inkscape, but the package 'transparent.sty' is not loaded}%
    \renewcommand\transparent[1]{}%
  }%
  \providecommand\rotatebox[2]{#2}%
  \newcommand*\fsize{\dimexpr\f@size pt\relax}%
  \newcommand*\lineheight[1]{\fontsize{\fsize}{#1\fsize}\selectfont}%
  \ifx\svgwidth\undefined%
    \setlength{\unitlength}{302.4969963bp}%
    \ifx\svgscale\undefined%
      \relax%
    \else%
      \setlength{\unitlength}{\unitlength * \real{\svgscale}}%
    \fi%
  \else%
    \setlength{\unitlength}{\svgwidth}%
  \fi%
  \global\let\svgwidth\undefined%
  \global\let\svgscale\undefined%
  \makeatother%
  \begin{picture}(1,0.40664688)%
    \lineheight{1}%
    \setlength\tabcolsep{0pt}%
    \put(0,0){\includegraphics[width=\unitlength,page=1]{singular-definition-stop.pdf}}%
    \put(0.47710229,0.23123067){\color[rgb]{0,0,0}\makebox(0,0)[lt]{\lineheight{1.25}\smash{\begin{tabular}[t]{l}\color{red}\footnotesize$L_{\Sigma( \Omega )}$\end{tabular}}}}%
    \put(0.1770177,0.31837749){\color[rgb]{0,0,0}\makebox(0,0)[lt]{\lineheight{1.25}\smash{\begin{tabular}[t]{l}\color{blue}\footnotesize$\Sigma( \Lambda )$\end{tabular}}}}%
    \put(0.01170624,0.23452452){\color[rgb]{0,0,0}\makebox(0,0)[lt]{\lineheight{1.25}\smash{\begin{tabular}[t]{l}\footnotesize$\partial W$\end{tabular}}}}%
    \put(0.20433809,0.07929254){\color[rgb]{0,0,0}\makebox(0,0)[lt]{\lineheight{1.25}\smash{\begin{tabular}[t]{l}\footnotesize$W$\end{tabular}}}}%
    \put(0.67858619,0.07929254){\color[rgb]{0,0,0}\makebox(0,0)[lt]{\lineheight{1.25}\smash{\begin{tabular}[t]{l}\footnotesize$\R \times ( \R \times V )$\end{tabular}}}}%
  \end{picture}%
\endgroup%

%% file: figures/exact-stops-isotopy.pdf_tex
\begingroup%
  \makeatletter%
  \providecommand\color[2][]{%
    \errmessage{(Inkscape) Color is used for the text in Inkscape, but the package 'color.sty' is not loaded}%
    \renewcommand\color[2][]{}%
  }%
  \providecommand\transparent[1]{%
    \errmessage{(Inkscape) Transparency is used (non-zero) for the text in Inkscape, but the package 'transparent.sty' is not loaded}%
    \renewcommand\transparent[1]{}%
  }%
  \providecommand\rotatebox[2]{#2}%
  \newcommand*\fsize{\dimexpr\f@size pt\relax}%
  \newcommand*\lineheight[1]{\fontsize{\fsize}{#1\fsize}\selectfont}%
  \ifx\svgwidth\undefined%
    \setlength{\unitlength}{429.10543126bp}%
    \ifx\svgscale\undefined%
      \relax%
    \else%
      \setlength{\unitlength}{\unitlength * \real{\svgscale}}%
    \fi%
  \else%
    \setlength{\unitlength}{\svgwidth}%
  \fi%
  \global\let\svgwidth\undefined%
  \global\let\svgscale\undefined%
  \makeatother%
  \begin{picture}(1,0.47696997)%
    \lineheight{1}%
    \setlength\tabcolsep{0pt}%
    \put(0,0){\includegraphics[width=\unitlength,page=1]{exact-stops-isotopy.pdf}}%
    \put(0.92382646,0.21474707){\color[rgb]{0,0,0}\makebox(0,0)[lt]{\lineheight{1.25}\smash{\begin{tabular}[t]{l}$a_{11}$\end{tabular}}}}%
    \put(0.74880367,0.12330302){\color[rgb]{0,0,0}\makebox(0,0)[lt]{\lineheight{1.25}\smash{\begin{tabular}[t]{l}$a_{12}$\end{tabular}}}}%
    \put(0.71464401,0.0928219){\color[rgb]{0,0,0}\makebox(0,0)[lt]{\lineheight{1.25}\smash{\begin{tabular}[t]{l}$a_{13}$\end{tabular}}}}%
    \put(0.68024192,0.06234048){\color[rgb]{0,0,0}\makebox(0,0)[lt]{\lineheight{1.25}\smash{\begin{tabular}[t]{l}$a_{14}$\end{tabular}}}}%
    \put(0.82260896,0.23576954){\color[rgb]{0,0,0}\makebox(0,0)[lt]{\lineheight{1.25}\smash{\begin{tabular}[t]{l}$a_{21}$\end{tabular}}}}%
    \put(0.92382646,0.27570971){\color[rgb]{0,0,0}\makebox(0,0)[lt]{\lineheight{1.25}\smash{\begin{tabular}[t]{l}$a_{22}$\end{tabular}}}}%
    \put(0.71213029,0.1515337){\color[rgb]{0,0,0}\makebox(0,0)[lt]{\lineheight{1.25}\smash{\begin{tabular}[t]{l}$a_{23}$\end{tabular}}}}%
    \put(0.68024132,0.12330302){\color[rgb]{0,0,0}\makebox(0,0)[lt]{\lineheight{1.25}\smash{\begin{tabular}[t]{l}$a_{24}$\end{tabular}}}}%
    \put(0.7712343,0.26201258){\color[rgb]{0,0,0}\makebox(0,0)[lt]{\lineheight{1.25}\smash{\begin{tabular}[t]{l}$a_{31}$\end{tabular}}}}%
    \put(0.82260896,0.30068026){\color[rgb]{0,0,0}\makebox(0,0)[lt]{\lineheight{1.25}\smash{\begin{tabular}[t]{l}$a_{32}$\end{tabular}}}}%
    \put(0.92382646,0.33667234){\color[rgb]{0,0,0}\makebox(0,0)[lt]{\lineheight{1.25}\smash{\begin{tabular}[t]{l}$a_{33}$\end{tabular}}}}%
    \put(0.68024192,0.18426565){\color[rgb]{0,0,0}\makebox(0,0)[lt]{\lineheight{1.25}\smash{\begin{tabular}[t]{l}$a_{34}$\end{tabular}}}}%
    \put(0.72023077,0.28345971){\color[rgb]{0,0,0}\makebox(0,0)[lt]{\lineheight{1.25}\smash{\begin{tabular}[t]{l}$a_{41}$\end{tabular}}}}%
    \put(0.7712343,0.32572994){\color[rgb]{0,0,0}\makebox(0,0)[lt]{\lineheight{1.25}\smash{\begin{tabular}[t]{l}$a_{42}$\end{tabular}}}}%
    \put(0.82260896,0.36715366){\color[rgb]{0,0,0}\makebox(0,0)[lt]{\lineheight{1.25}\smash{\begin{tabular}[t]{l}$a_{43}$\end{tabular}}}}%
    \put(0.92382646,0.39763488){\color[rgb]{0,0,0}\makebox(0,0)[lt]{\lineheight{1.25}\smash{\begin{tabular}[t]{l}$a_{44}$\end{tabular}}}}%
    \put(0,0){\includegraphics[width=\unitlength,page=2]{exact-stops-isotopy.pdf}}%
    \put(0.29986023,0.05451377){\color[rgb]{0,0,0}\makebox(0,0)[lt]{\lineheight{1.25}\smash{\begin{tabular}[t]{l}$\Lambda$\end{tabular}}}}%
    \put(0.0979492,0.38036333){\color[rgb]{0,0,0}\makebox(0,0)[lt]{\lineheight{1.25}\smash{\begin{tabular}[t]{l}$\Omega$\end{tabular}}}}%
    \put(0,0){\includegraphics[width=\unitlength,page=3]{exact-stops-isotopy.pdf}}%
  \end{picture}%
\endgroup%

%% file: figures/opening.pdf_tex
\begingroup%
  \makeatletter%
  \providecommand\color[2][]{%
    \errmessage{(Inkscape) Color is used for the text in Inkscape, but the package 'color.sty' is not loaded}%
    \renewcommand\color[2][]{}%
  }%
  \providecommand\transparent[1]{%
    \errmessage{(Inkscape) Transparency is used (non-zero) for the text in Inkscape, but the package 'transparent.sty' is not loaded}%
    \renewcommand\transparent[1]{}%
  }%
  \providecommand\rotatebox[2]{#2}%
  \newcommand*\fsize{\dimexpr\f@size pt\relax}%
  \newcommand*\lineheight[1]{\fontsize{\fsize}{#1\fsize}\selectfont}%
  \ifx\svgwidth\undefined%
    \setlength{\unitlength}{278.05310828bp}%
    \ifx\svgscale\undefined%
      \relax%
    \else%
      \setlength{\unitlength}{\unitlength * \real{\svgscale}}%
    \fi%
  \else%
    \setlength{\unitlength}{\svgwidth}%
  \fi%
  \global\let\svgwidth\undefined%
  \global\let\svgscale\undefined%
  \makeatother%
  \begin{picture}(1,0.37532408)%
    \lineheight{1}%
    \setlength\tabcolsep{0pt}%
    \put(0,0){\includegraphics[width=\unitlength,page=1]{opening.pdf}}%
  \end{picture}%
\endgroup%

%% file: figures/singular-resolution.pdf_tex
\begingroup%
  \makeatletter%
  \providecommand\color[2][]{%
    \errmessage{(Inkscape) Color is used for the text in Inkscape, but the package 'color.sty' is not loaded}%
    \renewcommand\color[2][]{}%
  }%
  \providecommand\transparent[1]{%
    \errmessage{(Inkscape) Transparency is used (non-zero) for the text in Inkscape, but the package 'transparent.sty' is not loaded}%
    \renewcommand\transparent[1]{}%
  }%
  \providecommand\rotatebox[2]{#2}%
  \newcommand*\fsize{\dimexpr\f@size pt\relax}%
  \newcommand*\lineheight[1]{\fontsize{\fsize}{#1\fsize}\selectfont}%
  \ifx\svgwidth\undefined%
    \setlength{\unitlength}{324.15156807bp}%
    \ifx\svgscale\undefined%
      \relax%
    \else%
      \setlength{\unitlength}{\unitlength * \real{\svgscale}}%
    \fi%
  \else%
    \setlength{\unitlength}{\svgwidth}%
  \fi%
  \global\let\svgwidth\undefined%
  \global\let\svgscale\undefined%
  \makeatother%
  \begin{picture}(1,1.11990362)%
    \lineheight{1}%
    \setlength\tabcolsep{0pt}%
    \put(0,0){\includegraphics[width=\unitlength,page=1]{singular-resolution.pdf}}%
  \end{picture}%
\endgroup%

%% file: examples.tex
\section{Examples and computations}
\label{sec:examples}
We here give a number of examples of how the results in the previous sections can be used to
compute the cohomology and minimal model of the Chekanov–Eliashberg dg-algebra. 

\begin{example}
\label{ex:unknot}
	Let $\Lambda$ be the standard Legendrian unknot with a single top
	handle, illustrated in \autoref{fig:unknot}.
\begin{figure}[!htb]
    \centering
    
    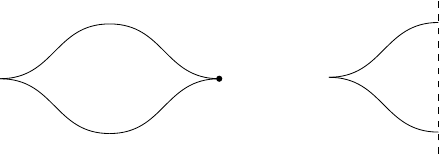

    \caption{The front projection of the standard Legendrian unknot $\Lambda$
	    with a single top handle, and its opening $\Lambda^{\circ}$.}
    \label{fig:unknot}
\end{figure}
	The opening of $\Lambda$ lacks Reeb chords, so its Chekanov--Eliashberg
	dg-algebra is isomorphic to the ground field $\textbf{k}$.
	\autoref{res:opening} then says that there is a quasi-isomorphism
	$CE^{*}( \Lambda;V_{0};\R^{4} )\cong \textbf{k}$. This has been shown
	in \cite[Example 7.4]{AE21} using other methods. The result is expected
	in light of the surgery formula, as the corresponding co-core consists
	of a single section in $W_{V} \cong T^{*}\R^{2}$ whose wrapped Floer
	cohomology is generated by a single self-intersection point, see
	\cite[Section 1.1]{AE21}.
\end{example}

\subsection{Rainbow connected sums}
We here describe a way of connecting two singular Legendrians so that 
the Chekanov--Eliashberg dg-algebra of the resulting Legendrian is
quasi-isomorphic to the direct product of the Chekanov--Eliashberg dg-algebras of the
initial Legendrians.
\begin{definition}
Let $\Lambda ^{1}$ and $\Lambda^{2}$ be two unlinked singular Legendrians in
$\R^{3}$, such that $\Lambda^{1}$ and $\Lambda^{2}$
both have $k$ left singularities. The
\textit{rainbow connected sum} of $\Lambda^{1}$ and $\Lambda^{2}$ is the
Legendrian $\Lambda^{1} \# \Lambda^{2}$ obtained by connecting $\Lambda^{1}$ and
$\Lambda^{2}$ by $k$ handles $\Pi$ as in \autoref{fig:rainbow-sum}.
\begin{figure}[!hbt]
    \centering
    
    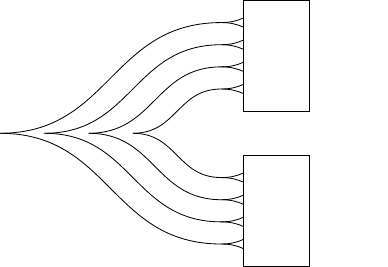

    \caption{The Rainbow connected sum of two singular Legendrians
    $\Lambda^{1}$ and $\Lambda^{2}$ in the front, with $k=4$.}
    \label{fig:rainbow-sum}
\end{figure}
\end{definition}
\begin{remark}
\label{rm:rainbow-one}
In particular, if one connects two unlinked Legendrians $\Lambda^{1}$ and
$\Lambda^{2}$ with only one
handle in some arbitrary way, the resulting 
Legendrian will always be isotopic to a rainbow connected
sum of two Legendrians isotopic to $\Lambda^{1}$ and $\Lambda^{2}$.
\end{remark}

\begin{proposition}
\label{res:rainbow-map}
	There is a quasi-isomorphism 
	\[
		CE^{*}( \Lambda^{1} \# \Lambda^{2};V_{0}^{1} \cup
		V_{0}^{2};\R^{4} ) \isomto 
		CE^{*}( \Lambda^{1};V_{0}^{1};\R^{4}) \times 
		CE^{*}( \Lambda^{2};V_{0}^{2};\R^{4}) 
	\]
	where '$\times$' denotes the direct product of dg-algebras.
\end{proposition}
\begin{proof}
	Similarly to in the proof of \autoref{res:stopped-lemma}, one can
	perform an isotopy of $\Lambda^{1} \# \Lambda^{2}$ introducing 
	chords as in \autoref{fig:exact-stops-isotopy} and use 
	\autoref{res:exact-removal} to successively remove the handles in $\Pi$. 
	We then obtain a quasi-isomorphism 
	$CE^{*}( \Lambda^{1}\#\Lambda^{2};V_{0};\R^{4} )  \isomto
	CE^{*}( \Lambda^{1}\cup \Lambda^{2};V_{0}^{1} \cup 
	V_{0}^{2};\R^{4} )$, and since $\Lambda^{1}$ and $\Lambda^{2}$ are
	unlinked there is a natural isomorphism $CE^{*}( \Lambda^{1}\cup
	\Lambda^{2};V_{0}^{1}\cup V_{0}^{2};\R^{4} ) \cong
	CE^{*}( \Lambda^{1};V_{0}^{1};\R^{4}) \times 
	CE^{*}( \Lambda^{2};V_{0}^{2};\R^{4})$.
\end{proof}

Using this, we can compute the cohomology of the handcuff graph.
\begin{example}
	Let $\Lambda$ be the singular Legendrian illustrated in
	\autoref{fig:handcuff-graph},
	obtained by connecting two unknots with a single handle. 
\begin{figure}[!htb]
    \centering
    
    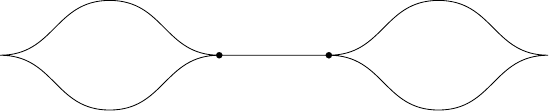

    \caption{The front projection of two unlinked unknots connected by a handle.}
    \label{fig:handcuff-graph}
\end{figure}
	As noted in \autoref{rm:rainbow-one}, $\Lambda$ is isotopic to a
	rainbow connected sum of two unknots. By \autoref{res:rainbow-map} and
	\autoref{ex:unknot} we then have
	\[
		CE^{*}( \Lambda;V_{0};\R^{4} ) \cong \textbf{k} \times
		\textbf{k}.
	\] 
	This result was expected by An--Bae in \cite[Section 8]{AB20}.
\end{example}
\subsection{Finite dimensional models}
We here construct an infinite class of Legendrians for which 
\autoref{res:opening} gives finite dimensional models.

\begin{definition}
\label{def:permutation-legendrians}
	A \emph{handle permutation of order $n$} is a word
	$\sigma=\sigma_1\sigma_2\ldots\sigma_2$ of elements $\sigma_{i} \in
	\{1,2,\ldots,n\}$ such that each $k \in \{1,2,\ldots,n\}$ appears
	precisely two times in $\sigma$. We write $\sigma( i^{-} )$ and
	$\sigma( i^{+} )$ for the unique numbers such that 
	$\sigma_{\sigma( i^{-} )}=\sigma_{\sigma( i^{+} )}=i$ and $\sigma(
	i^{-} ) < \sigma( i^{+} )$. 
	If two handle permutations are related by a permutation (i.e.
	relabeling) of
	$\{1,2,\ldots,n\}$ we consider them to be equivalent.
\end{definition}
\begin{definition}
	Let the \emph{standard unknot of size $l$} be the knot whose front
	diagram is obtained by a uniform scaling of the front of the standard
	unknot illustrated to the left in \autoref{fig:unknot}, such that the
	length of the Reeb chord becomes $l$. Given a handle permutation
	$\sigma$ of order $n$, we define $\Lambda^{\circ}_{\sigma}$ to be the
	smooth bordered Legendrian consisting of $n$ strands
	$\Lambda_1^{\circ}, \ldots ,\Lambda_n^{\circ}$ such that each strand
	$\Lambda_i^{\circ}$ is a copy of the left half of the standard unknot
	of size $\sigma( i^{+} ) - \sigma( i^{-} )$ translated so that the
	upper and lower ends of $\Lambda_{i}^{\circ}$ have the respective front
	coordinates $( 0,\sigma( i^{+} ))$ and $( 0,\sigma( i^{-} ))$. We then
	define $\Lambda_{\sigma}$ to be the unique singular Legendrian with
	opening $\Lambda^{\circ}_{\sigma}$ such that $\Lambda_{\sigma}$ only
	has one singularity, and call it the \emph{permutation Legendrian} of
	$\sigma$. We denote by $V_{\sigma}$ the Weinstein thickening of
	$\Lambda_{\sigma}$. Note that this construction is independent of the
	choice of representative of the handle permutation, as defined above.
\end{definition}

The algebra $CE^{*}( \Lambda_{\sigma}^{ \circ};\R^{4} )$ has $n$
idempotents and one Reeb chords generator $a_{ij}$ from $i$ to $j$ for each pair 
$i$ and $j$ such that $\sigma( i^{-} ) < \sigma( j^{-} ) < \sigma( i^{+} ) <
\sigma( j^{+} )$. If one
gives the handles the same Maslov potential (under the canonical identification
by translation and rescaling) then $| a_{ij} | = 1$ for all $i$ and $j$.
The differential is given by
\[
	\partial a_{ij} = \sum \pm a_{kj}a_{ik}
\] 
where the sum is taken over all $1 < k < n$ for which the chords $a_{kj}$ and
$a_{ik}$ exist, i.e. all $k$ such that $\sigma( i^{-} ) < \sigma( k^{-} ) <
\sigma( i^{+} ) < \sigma( k^{+} )$ and $\sigma( k^{-} ) < \sigma( j^{-} ) <
\sigma( k^{+} ) < \sigma( j^{+} )$. The
finite dimensionality implies that the cohomology can be readily computed.
\begin{example}
\label{ex:standard-an}
	Let $n > 0$ and let $\sigma$ be the handle permutation of order $n$ given by
	$\sigma(i^{-}) = i$ and $\sigma(i^{+}) = 2i$. We call $\Lambda_{A}^{n}
	:= \Lambda_{\sigma}$ the \emph{standard $A_{n}$-Legendrian}. 
\begin{figure}[!htb]
    \centering
    
    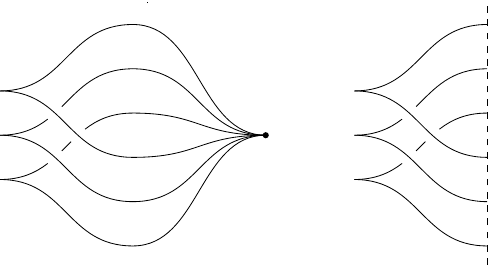

    \caption{The front projection of the 
	    standard $A_{n}$-Legendrian $\Lambda_{A}^{n}$ and its opening
    $\Lambda_{A}^{n,\circ}$.} 
    \label{fig:an-pos-neg-sing}
\end{figure}
	The algebra $CE^{*}( \Lambda_{A}^{n,\circ};\R^{4} )$ has one
	generator $a_{ij}$ for each $1 \leq i < j \leq n$, with  
	\[
		\partial a_{ij} = \sum_{i < k < j} \pm a_{kj}a_{ik}.
	\] 
	\begin{proposition}
	The minimal model of $CE^{*}( \Lambda_{A}^{n};V_{A,0}^{n};\R^{4} )$ is
	isomorphic to the path algebra of the $A_{n}$-quiver
\[\begin{tikzcd}
	{\stackrel{1}{\bullet}} & {\stackrel{2}{\bullet}} & \stackrel{3}{\bullet} & 
	{\ldots} & {\stackrel{n-1}{\bullet}} & 
	{\stackrel{n}{\bullet}} 
	\arrow["{\alpha_1}", from=1-1, to=1-2]
	\arrow["{\alpha_2}", from=1-2, to=1-3]
	\arrow["{\alpha_3}", from=1-3, to=1-4]
	\arrow["{\alpha_{n-1}}", from=1-4, to=1-5]
	\arrow["{\alpha_{n-2}}", from=1-5, to=1-6]
\end{tikzcd}\]
	bound by the relations $\alpha_{i+1}\alpha_{i}=0$ for $1 \leq i \leq
	n-2$, without any higher operations.
	\end{proposition}
	\begin{proof}
		The differential of $CE^{*}( \Lambda_{A}^{n,\circ};\R^{4} )$ 
		is essentially the same differential as $\partial_{0}$ in
		\autoref{ssec:surgery-surfaces} and one can use the same argument as in
		\autoref{res:short-chords-generate} to see that the cohomology is
		generated by the residues of the chords of the form $a_{i,i+1}$ for $1
		\leq i < n$, with zero multiplication. The cohomology is thus
		isomorphic to the path algebra. There is a dg-algebra quasi-isomorphism
		$CE^{*}( \Lambda_{A}^{n,\circ};\R^{4} ) \to
		H^{*}CE( \Lambda_{A}^{n,\circ};\R^{4} )$ given by
		$a_{i,i+1} \mapsto [a_{i,i+1}]$ and $a_{ij} \mapsto 0$ for 
		$j-i > 1$, which shows that this indeed is the minimal model
		of $CE^{*}( \Lambda_{A}^{n,\circ};\R^{4} )$, and 
		then by \autoref{res:opening} also the minimal model of $CE^{*}(
		\Lambda_{A}^{n};V_{A,0}^{n};\R^{4} )$
	\end{proof}
	
	Partial computations of this example were done by An--Bae in
	\cite[Section 8]{AB20}, where they expected this result. These
	computations are related to the computations using the technology of
	microlocal sheaf theory; see Nadler's computations for the $A_n$-arboreal
	Legendrian from \cite{Nad17}.

	As a geometric application, we can use this computation to obstruct the
	existence of certain isotopies, compare \cite{Mis03}.
	\begin{corollary}
		A Legendrian isotopy from $\Lambda_{A}^{n}$ to itself cannot
		produce a non-trivial permutation of the handles. 
	\end{corollary}
	\begin{proof}
		An isotopy from $\Lambda_{A}^{n}$ to itself permuting 
		the handles induces an algebra automorphism of $H^{*}CE(
		\Lambda_{A}^{n};V_{A,0}^{n};\R^{4} )$ permuting the corresponding 
		idempotents in the same way, and no non-trivial
		permutations of the idempotents of 
		$H^{*}CE( \Lambda_{A}^{n};V_{A,0}^{n};\R^{4} )$ can be produced by
		an algebra automorphism.
	\end{proof}
\end{example}

\subsubsection{Mutations}
Consider an abstract Weinstein surface $V$ such that $V_{0}$ is connected
and let $\partial \Lambda \subset V_{0}$ be the critical attaching spheres. By
fixing a base point in $\partial V_{0}$ and following the Reeb flow around
$\partial V_{0}$ we get a handle permutation $\sigma$ such that $V$
and $V_{\sigma}$ are isomorphic as Weinstein manifolds and
there is an embedding $V \hookrightarrow \R^{3}$ whose image is $V_{\sigma}$. The
handle permutation $\sigma$ depends on the base point and will in general result in
non-Legendrian isotopic $V_{\sigma}$. However, given a fixed handle
decomposition, $\sigma$ is unique up to cyclic
permutations and $V_{\sigma}$ is unique up to Weinstein isotopy.  

\begin{definition}
	Let $\sigma$ be a handle permutation of order $n$ and let $k \in
	\Z$. The \emph{shift of $\sigma$ by $k$} is the
	permutation
	\[
		\sigma[k]:= \sigma_{1-k}\ldots\sigma_{i-k}\ldots\sigma_{n-k},
	\] 
	counting mod $n$ in the indices.
\end{definition}

\begin{proposition}
\label{res:mutation}
	Let $\sigma$ be a handle permutation of order $n$ and let $k \in \Z$.
	There is a Weinstein isotopy
	\[
		V_{\sigma} \sim
		V_{\sigma[k]}.
	\] 
\end{proposition}
\begin{proof}
	It is sufficient to show this for $k=1$. The isotopy is illustrated in
	\autoref{fig:weinstein-isotopy-proof} and is performed as follows. We
	relabel so that $\sigma_{2n} = n$ and 
	give the ends of the handles $\Lambda_i$ of $V_{\sigma}$ the labeling
	$i^{\pm}$ where $i=1,\ldots,n$, as specified in
	\autoref{def:permutation-legendrians}. In the first step, we perform a
	Weinstein isotopy introducing a handle with a right cusp as in
	\autoref{res:weinstein-lemma}, which splits
	the singularity in two singularities such that an end $i^{\pm}$
	belongs to the lower singularity if $\sigma( i^{\pm} ) \leq \sigma(
	n^{-} )$ and to the upper singularity if $\sigma( i^{\pm}) > \sigma(
	n^{-} )$. In the second and third steps, we perform
	Reidemeister IV moves at the ends $n^{-}$ and $n^{+}$ which moves the
	handle $\Lambda_{n}$ into a position slightly below the introduced
	handle, obtaining the Legendrian in the fourth picture in
	\autoref{fig:weinstein-isotopy-proof}. This
	Legendrian can be obtained from $V_{\sigma[1]}$ by an analogous
	Weinstein isotopy, as shown in the two last pictures of the figure. 
	We thus have $V_{\sigma} \sim V_{\sigma[1]}$.
\begin{figure}[!htb]
    \centering
    
    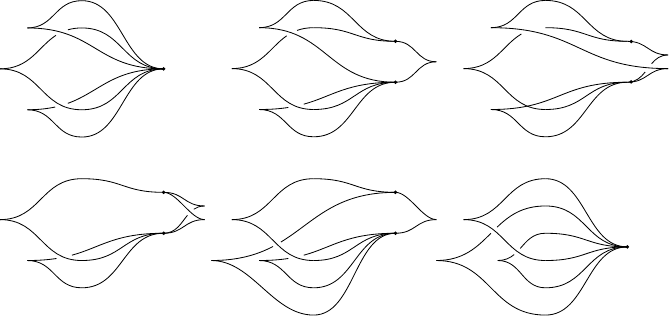

    \caption{A Weinstein isotopy between the permutation Legendrians
	    $\Lambda_{ 121323 }$ and $\Lambda_{ 312132 }$. 
	    The figure should be read from left to right and top
    	    to bottom.}
    \label{fig:weinstein-isotopy-proof}
\end{figure}
\end{proof}

Viewing the Chekanov--Eliashberg dg-algebras as path algebras,
the quiver of $CE^{*}( \Lambda_{ \sigma[1] }^{\circ};\R^{4} )$ can be obtained
from that of $CE^{*}( \Lambda_{\sigma}^{\circ};\R^{4} )$ by flipping all
arrows with target on the vertex corresponding to the $\sigma_{2n}$:th handle.
The isotopy can therefore be viewed as geometric realization of the 
Bernstein--Gelfand--Ponomarev reflection functors \cite{BGP73}.
\begin{example}
\label{ex:an-no-relations}
	Consider the handle permutation $\sigma = 121323$. 
\begin{figure}[!htb]
    \centering
    
    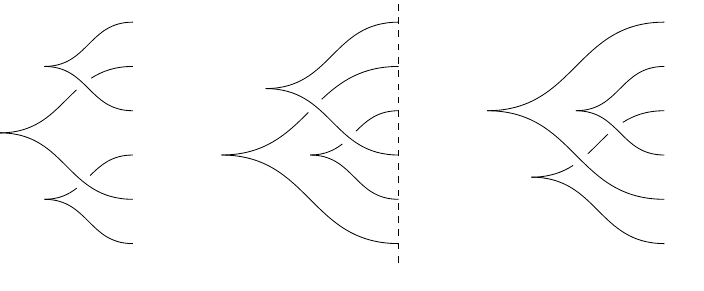

    \caption{The openings of the three Legendrians obtained by shifting
    $\sigma=121323$.}
    \label{fig:source-mutation-example}
\end{figure}
	The openings of 
	$\Lambda_{\sigma}$ and the shifts $\Lambda_{ \sigma[1] }$ and
	$\Lambda_{ \sigma[2] }$ are illustrated in \autoref{fig:source-mutation-example}.
	The Chekanov--Eliashberg dg-algebras of the openings of all three have
	vanishing differential and are isomorphic to the respective path
	algebras of the following quivers:
\[\begin{tikzcd}
	\bullet & \bullet & \bullet, && \bullet & \bullet & \bullet, && \bullet &
	\bullet & \bullet.
	\arrow[from=1-1, to=1-2]
	\arrow[from=1-2, to=1-3]
	\arrow[from=1-7, to=1-6]
	\arrow[from=1-5, to=1-6]
	\arrow[from=1-10, to=1-11]
	\arrow[from=1-10, to=1-9]
\end{tikzcd}\]
\end{example}

The Weinstein isotopies in \autoref{res:mutation} are, however, not all
Weinstein isotopies which exist between the permutation Legendrians, as the
following example shows.
\begin{example}
	Let $n > 1$ and let $\sigma$ the handle permutation of order $n$,
	given by
	\begin{align*}
		\sigma( 1^{-} ) = 1, && \sigma( i^{-} ) = 2i-2, &&& \sigma(
		n^{-} ) = 2n-2,\\
		\sigma( 1^{+} ) = 3, && \sigma( i^{-} ) = 2i+1, &&& \sigma(
		n^{-} )
		= 2n,
	\end{align*}
	for $i \neq 1,n$. The permutation Legendrian $\Lambda_{A'}^{n}:=
	\Lambda_{\sigma}$ and its opening are
	illustrated in
	\autoref{fig:an-no-relations}.
\begin{figure}[!htb]
    \centering
    
    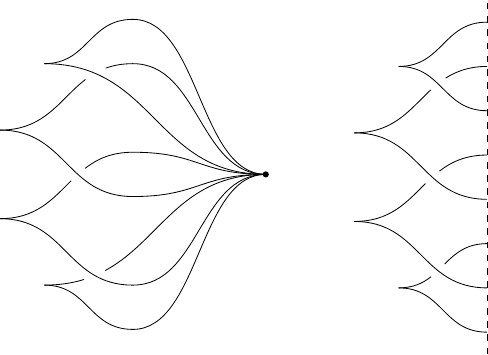

    \caption{The front projection of the Legendrian $\Lambda_{A'}^{n}$ and its
	    opening $\Lambda_{A'}^{n, \circ}$.}
    \label{fig:an-no-relations}
\end{figure}
The Chekanov--Eliashberg dg-algebra of $\Lambda_{A'}^{n, \circ}$ 
has differential zero and is isomorphic to the path
algebra of the $A_{n}$-quiver 
\[\begin{tikzcd}
	{\stackrel{1}{\bullet}} & {\stackrel{2}{\bullet}} & \stackrel{3}{\bullet} & 
	{\ldots} & {\stackrel{n-1}{\bullet}} & 
	{\stackrel{n}{\bullet}} 
	\arrow["{\alpha_1}", from=1-1, to=1-2]
	\arrow["{\alpha_2}", from=1-2, to=1-3]
	\arrow["{\alpha_3}", from=1-3, to=1-4]
	\arrow["{\alpha_{n-1}}", from=1-4, to=1-5]
	\arrow["{\alpha_{n-2}}", from=1-5, to=1-6]
\end{tikzcd}\]
without any relations. We expect that this Weinstein surface can be naturally 
realized a Milnor fiber of the $A_{n}$-singularity in the boundary of $B^{4}$,
see \cite[Corollary 1.17]{GPS18}\cite{Nad17}.
\begin{proposition}
\label{res:an-isotopy}
	There is a Weinstein isotopy,
	\[
		V_{A}^{n} \sim V_{ A' }^{n}.
	\] 
\end{proposition}
\begin{proof}
	The isotopy is illustrated in \autoref{fig:an-isotopy} for $n=4$. The
	procedure is the same for all $n$. 
\begin{figure}[!htb]
    \centering
    
    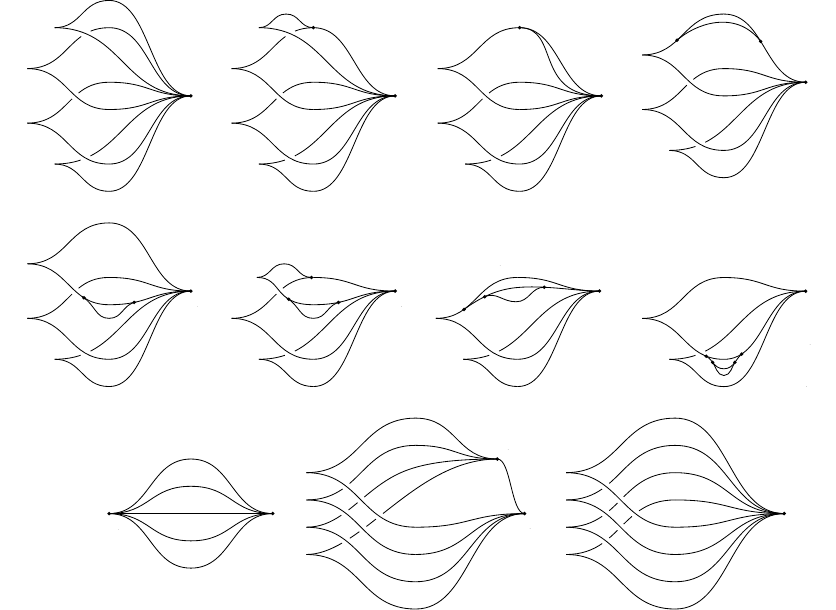

    \caption{A Weinstein isotopy between $V_{A}^{n}$ and $V_{ A' }^{n}$ for
	    $n=4$. The figure should be read from left to right and top to bottom.}
    \label{fig:an-isotopy}
\end{figure}
	The first step is a Weinstein handle
	introduction as in \autoref{res:weinstein-lemma}. The second step is a
	Reidemeister VI move. The third step is another handle introduction.
	The fourth step is an application of \autoref{res:push-twist}. The
	fifth, sixth, and seventh steps are performed in the same way as the
	first four: by introducing a handle, performing a Reidemeister VI move,
	introducing another handle, and then applying 
	\autoref{res:push-twist}. By repeating this procedure for each
	handle and then performing a series of Weinstein handle contractions one 
	then obtains the Legendrian in the ninth picture. The second to last step is
	then a Reidemeister IV move, followed by another application of 
	\autoref{res:push-twist}, and the final step is a Weinstein
	handle contraction.
\end{proof}
\end{example}

Examples \ref{ex:standard-an} and \ref{ex:an-no-relations} 
demonstrate that the Chekanov--Eliashberg dg-algebra can undergo
significant changes under Weinstein isotopy. Recall that this corresponds to
different choices of generators of the partially wrapped Fukaya category
obtained by stopping at the Weinstein surface. It
is interesting to note that the path algebra of the $A_{n}$-quiver with all
quadratic relations quotiented out is the Koszul dual of the path algebra of
the same quiver without relations.  We intend to investigate
bifurcations of the Chekanov--Eliashberg dg-algebra under Weinstein isotopy further in the future.
\subsection{Infinite dimensional models}
In addition to the permutation Legendrians, there are several families of
singular Legendrians with infinite dimensional models for which the 
cohomology can be computed. 
\begin{example}
\label{ex:theta}
	Let $n > 1$ and let $\Lambda_{\theta}^{n}$ be the singular
	Legendrian with one right
	singularity and one left singularity such that when opening
	$\Lambda_{\theta}^{n}$ at both singularities one obtains a bordered Legendrian
	consisting of $n$ strands without cusps, crossings, or singularities.
	It is illustrated for $n=4$ in \autoref{fig:an-plus-one-construction}.
\begin{figure}[!htb]
    \centering
    
    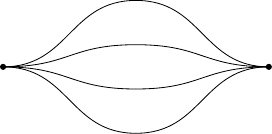

    \caption{The front projection of the Legendrian $\Lambda_{\theta}^{n}$ for
    $n=4$.}
    \label{fig:an-plus-one-construction}
\end{figure}
	We call $\Lambda_{\theta}^{n}$ the \emph{$\theta_{n}$-Legendrian}. From
	the proof of \autoref{res:an-isotopy} we see that
	$\Lambda_{\theta}^{n}$ is Weinstein isotopic to $\Lambda_{A}^{n-1}$ and
	$\Lambda_{A'}^{n-1}$.  
	\begin{proposition}
	\label{res:theta-minimal}
	The minimal model of $CE^{*}(
	\Lambda_{\theta}^{n};V_{\theta,0}^{n};\R^{4} )$ is isomorphic to
	the $A_{\infty}$-algebra whose underlying module is the path algebra of
	the quiver with one vertex for each $i \in \Z_{n}$ and one arrow
	$\alpha_{i}$ from $i$ to $i+1$ for each $i$, with all compositions
	quotiented out, and whose $A_{\infty}$-operations $\mu_{k}$ act
	by
	 \[
		 \mu_{n}( \alpha_{i+n-1} \otimes\ldots \otimes \alpha_{i+1}
		 \otimes \alpha_{i} ) = \varepsilon_{i}
	 \] 
	for each $i \in \Z_{n}$, and vanish on all other words.
	\end{proposition}
	\begin{proof}
		By \autoref{res:opening}, we have a quasi-isomorphism 
		\[
			CE^{*}(
			\Lambda;V_{0};\R^{4} ) \cong CE^{*}(
			(\Lambda_{\theta}^{n})^{\circ,t};V_{0};\R^{4} ),
		\] 
		where $t$ is
		the right singularity. The algebra $CE^{*}(
		(\Lambda_{\theta}^{n})^{\circ,t};V_{0};\R^{4} )$ is isomorphic
		to its internal algebra, i.e. the Chekanov--Eliashberg dg-algebra
		of $n$ one-point stops in the boundary of the disk. The result
		then follows from 
		\autoref{res:zero-dim-leg-min-mod}.
	\end{proof}
	We thus have a third, non-formal representative of the
	Chekanov--Eliashberg dg-algebras of the Weinstein isotopy class of 
	$\Lambda_{A}^{n}$ and $\Lambda_{A'}^{n}$.

	\begin{remark}
	Geometrically, the fact that $\Lambda_{\theta}^{n}$ and the $n$-point
	stop have quasi-isomorphic Chekanov--Eliashberg dg-algebras can be
	explained by noting that $\Lambda_\theta^{n}$ is the $(
	T^{*}D^{1},T^{*}S^{0} )$-stabilization of the $n$-point stop. This
	stabilization construction is a special case of the product
	construction defined in \cite[Section 3.2]{Eli18}, and for which
	Ganatra--Pardon--Shende \cite{GPS18} have constructed a Künneth
	embedding of the corresponding partially wrapped Fukaya categories. In
	light of \autoref{res:sing-surgery-map}, one therefore expects there to
	be a Künneth formula for Chekanov--Eliashberg dg-algebras, from which the
	above quasi-isomorphism would follow.  
	\end{remark}

	Similarly to in \autoref{ex:standard-an}, we can use the above computation 
	to obstruct certain isotopies. 

	\begin{corollary}
		There exists a Legendrian isotopy from $\Lambda_\theta^{n}$ to
		itself realizing a
		given permutation of the handles if and only if the
		permutation is cyclic.
	\end{corollary}	
	\begin{proof}
		The only permutations of the idempotents of the minimal model
		in \autoref{res:theta-minimal} which can be realized by an
		algebra automorphism are the cyclic ones. This implies that
		there are no isotopies of $\Lambda$ producing non-cyclic
		permutations of the handles. The cyclic permutations are easy
		to produce using the Reidemeister moves.
	\end{proof}
\end{example}
\begin{example}
Let $n > 0$. We construct a singular Legendrian $\Lambda_{\text{cyc}}^{n}$ as
follows. 
\begin{figure}[!htb]
    \centering
    
    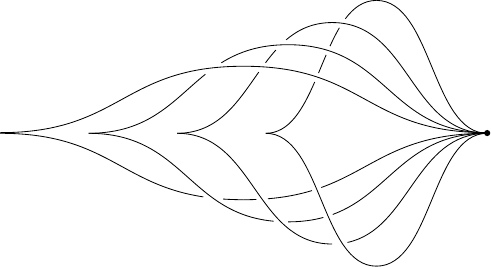

    \caption{The singular Legendrian $\Lambda_{\text{cyc}}^{n}$, defined as the
    union of $n$ unknots successively more squeezed in the $x$-direction and
    stretched in the $z$-direction.} 
    \label{fig:vanishing-example}
\end{figure}
First, we consider a standard unknot with a single top handle and with the
singularity and right cusp placed at the origin in $\R^{3}$. We construct a
second unknot by applying a transformation of the form $( x,z ) \to ( x \slash
M, M z)$ for some $M > 1$ to the front diagram of the first unknot.  We
construct a third unknot by applying the same transformation to the second
unknot, and continue like this until we have $n$ unknots. Then
$\Lambda_{\text{cyc}}^{n}$ is defined to be the union of these unknots, see
\autoref{fig:vanishing-example}.

\begin{proposition}
	The minimal model of $CE^{*}(
	\Lambda_{\textup{cyc}}^{n};V_{\textup{cyc},0}^{n};\R^{4} )$ is 
	isomorphic to
	the path algebra of the quiver with $n$ vertices $1,\ldots,n$ and one
	arrow from $i$ to $j$ for each ordered pair $( i,j ) \in
	\{1,\ldots,n\}^{2}$ such that $i \neq j$.
\end{proposition}
\begin{proof}
	For each ordered pair $(i,j)$ of distinct components of
	$\Lambda_{\textup{cyc}}^{n,\circ}$, there is one
	Reeb chord $a_{ij}$ from $i$ to $j$. If one gives the handles the same
	Maslov potential then each $a_{ij}$ has degree zero, and the
	differential thus vanishes. The result then follows by \autoref{res:opening}.
\end{proof}
\begin{proposition}
	The Legendrian $\Lambda_{\textup{cyc}}^{n}$ is not Weinstein
	isotopic to any permutation Legendrian.
\end{proposition}
\begin{proof}
	As noted in \cite[Remark 1.5]{AE21}, the Hochschild homology of the
	singular Chekanov--Eliashberg dg-algebra is a Weinstein isotopy invariant.
	The Hochschild homology is derived invariant, and in particular,
	invariant under quasi-isomorphism. Recall that the Hochschild homology
	in degree $0$ of an associate algebra $A$ is the quotient $A \slash
	\text{Span}\{ab - ba \mid a,b \in A\}$. Since $CE^{*}( \Lambda^{n,
	\circ}_{\text{cyc}};\R^{4} )$ is the path algebra of a cyclic quiver
	this implies that $HH_{0}(CE^{*}( \Lambda^{n,
	\circ}_{\text{cyc}};\R^{4} ) )$ is infinite dimensional (a
	linearly independent infinite subset 
	is given by e.g. $\{[(a_{21}a_{12})^{i}] \mid i
	\in \N \}$). On the other
	hand, for any permutation Legendrian $\Lambda_{\sigma}$, the algebra
	$CE^{*}( \Lambda_\sigma^{\circ};\R^{4} )$ is a semi-free finite
	dimensional dg-algebra and therefore homologically smooth and compact,
	see \cite[Section 8]{KS09}. By \cite[Proposition 8.10]{KS09}, it
	then follows that the Hochschild homology of $CE^{*}(
	\Lambda_\sigma^{\circ};\R^{4} )$ is finite dimensional. Thus, such a
	Weinstein isotopy cannot exist.
\end{proof}
\end{example}

\begin{example}
	Let $n > 1$ and let $\Lambda_{\theta'}^{n}$ be the Legendrian
	illustrated in \autoref{fig:stabilized-theta}. 
\begin{figure}[H]
    \centering
    
    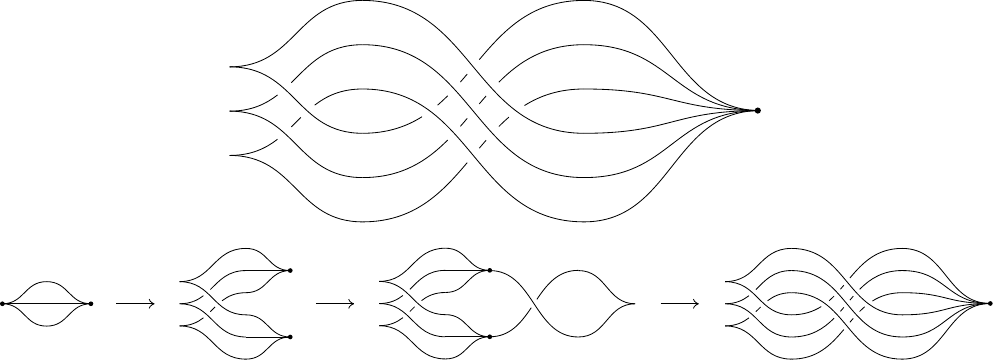

    \caption{The front projection the singular Legendrian knot
    $\Lambda_{\theta'}^{n}$ for $n=3$, and its construction from
    $\Lambda_{\theta}^{n}$.}
    \label{fig:stabilized-theta}
\end{figure}
	 It is obtained from the Legendrian $\Lambda_{\theta}^{n}$ from 
	 \autoref{ex:theta} by performing a
	 Legendrian isotopy using \autoref{res:push-twist} so that we get two
	 right singularities, and then attaching and contracting a handle with
	 core $\Pi$, as shown in the figure.
\begin{proposition}
	 There is a quasi-isomorphism
	 \[
		 CE^{*}( \Lambda_{\theta'}^{n};V_{\theta',0}^{n};\R^{4} ) \cong 
		 CE^{*}( \Lambda_{\theta}^{n};V_{\theta,0}^{n};\R^{4} ).
	 \] 
\end{proposition}
\begin{proof}
	The Reeb chord corresponding to the right cusp of $\Pi$ in
	\autoref{fig:stabilized-theta} has
	differential $\pm e_{\Pi}$. The result then follows by
	\autoref{res:exact-removal} and \autoref{res:weinstein-isotopy-cor}
\end{proof}
\begin{remark}
	This provides an alternative proof for a result by Etg\"{u}--Lekili.
	Let $CE^{*}( \partial \Omega;D_{2} )$ be the Chekanov--Eliashberg
	dg-algebra of $n$ one-point stops in the boundary of the disk $D^{2}$.  
	In \cite[Theorem 12]{EL19}, it was shown that the inclusion of the
	dg-subalgebra of all chords $c_{ij}^{p}$ with $p \leq 1$ into $CE^{*}(
	\partial \Omega;D_{2} )$ is a quasi-isomorphism. This subalgebra is
	isomorphic to $CE^{*}( \Lambda_{\theta'}^{n, \circ};\R^{4} )$. The
	chords with $p=0$ correspond to the chords in the triangle shape to
	the left in \autoref{fig:stabilized-theta}, and the chords with $p=1$
	to the chords in the diamond shape in the center. Using
	\autoref{res:opening} we thus recover \cite[Theorem 12]{EL19}.
\end{remark}
\end{example}

%% file: figures/unknot.pdf_tex
\begingroup%
  \makeatletter%
  \providecommand\color[2][]{%
    \errmessage{(Inkscape) Color is used for the text in Inkscape, but the package 'color.sty' is not loaded}%
    \renewcommand\color[2][]{}%
  }%
  \providecommand\transparent[1]{%
    \errmessage{(Inkscape) Transparency is used (non-zero) for the text in Inkscape, but the package 'transparent.sty' is not loaded}%
    \renewcommand\transparent[1]{}%
  }%
  \providecommand\rotatebox[2]{#2}%
  \newcommand*\fsize{\dimexpr\f@size pt\relax}%
  \newcommand*\lineheight[1]{\fontsize{\fsize}{#1\fsize}\selectfont}%
  \ifx\svgwidth\undefined%
    \setlength{\unitlength}{210.67131944bp}%
    \ifx\svgscale\undefined%
      \relax%
    \else%
      \setlength{\unitlength}{\unitlength * \real{\svgscale}}%
    \fi%
  \else%
    \setlength{\unitlength}{\svgwidth}%
  \fi%
  \global\let\svgwidth\undefined%
  \global\let\svgscale\undefined%
  \makeatother%
  \begin{picture}(1,0.35320161)%
    \lineheight{1}%
    \setlength\tabcolsep{0pt}%
    \put(0,0){\includegraphics[width=\unitlength,page=1]{unknot.pdf}}%
  \end{picture}%
\endgroup%

%% file: figures/rainbow-sum.pdf_tex
\begingroup%
  \makeatletter%
  \providecommand\color[2][]{%
    \errmessage{(Inkscape) Color is used for the text in Inkscape, but the package 'color.sty' is not loaded}%
    \renewcommand\color[2][]{}%
  }%
  \providecommand\transparent[1]{%
    \errmessage{(Inkscape) Transparency is used (non-zero) for the text in Inkscape, but the package 'transparent.sty' is not loaded}%
    \renewcommand\transparent[1]{}%
  }%
  \providecommand\rotatebox[2]{#2}%
  \newcommand*\fsize{\dimexpr\f@size pt\relax}%
  \newcommand*\lineheight[1]{\fontsize{\fsize}{#1\fsize}\selectfont}%
  \ifx\svgwidth\undefined%
    \setlength{\unitlength}{186.16551437bp}%
    \ifx\svgscale\undefined%
      \relax%
    \else%
      \setlength{\unitlength}{\unitlength * \real{\svgscale}}%
    \fi%
  \else%
    \setlength{\unitlength}{\svgwidth}%
  \fi%
  \global\let\svgwidth\undefined%
  \global\let\svgscale\undefined%
  \makeatother%
  \begin{picture}(1,0.68732886)%
    \lineheight{1}%
    \setlength\tabcolsep{0pt}%
    \put(0,0){\includegraphics[width=\unitlength,page=1]{rainbow-sum.pdf}}%
    \put(0.69101119,0.52200372){\color[rgb]{0,0,0}\makebox(0,0)[lt]{\lineheight{1.25}\smash{\begin{tabular}[t]{l}$\Lambda^{1}$\end{tabular}}}}%
    \put(0.68885521,0.12230906){\color[rgb]{0,0,0}\makebox(0,0)[lt]{\lineheight{1.25}\smash{\begin{tabular}[t]{l}$\Lambda^{2}$\end{tabular}}}}%
    \put(0,0){\includegraphics[width=\unitlength,page=2]{rainbow-sum.pdf}}%
  \end{picture}%
\endgroup%

%% file: figures/handcuff-graph.pdf_tex
\begingroup%
  \makeatletter%
  \providecommand\color[2][]{%
    \errmessage{(Inkscape) Color is used for the text in Inkscape, but the package 'color.sty' is not loaded}%
    \renewcommand\color[2][]{}%
  }%
  \providecommand\transparent[1]{%
    \errmessage{(Inkscape) Transparency is used (non-zero) for the text in Inkscape, but the package 'transparent.sty' is not loaded}%
    \renewcommand\transparent[1]{}%
  }%
  \providecommand\rotatebox[2]{#2}%
  \newcommand*\fsize{\dimexpr\f@size pt\relax}%
  \newcommand*\lineheight[1]{\fontsize{\fsize}{#1\fsize}\selectfont}%
  \ifx\svgwidth\undefined%
    \setlength{\unitlength}{263.09035654bp}%
    \ifx\svgscale\undefined%
      \relax%
    \else%
      \setlength{\unitlength}{\unitlength * \real{\svgscale}}%
    \fi%
  \else%
    \setlength{\unitlength}{\svgwidth}%
  \fi%
  \global\let\svgwidth\undefined%
  \global\let\svgscale\undefined%
  \makeatother%
  \begin{picture}(1,0.20151294)%
    \lineheight{1}%
    \setlength\tabcolsep{0pt}%
    \put(0,0){\includegraphics[width=\unitlength,page=1]{handcuff-graph.pdf}}%
  \end{picture}%
\endgroup%

%% file: figures/an-pos-neg-sing.pdf_tex
\begingroup%
  \makeatletter%
  \providecommand\color[2][]{%
    \errmessage{(Inkscape) Color is used for the text in Inkscape, but the package 'color.sty' is not loaded}%
    \renewcommand\color[2][]{}%
  }%
  \providecommand\transparent[1]{%
    \errmessage{(Inkscape) Transparency is used (non-zero) for the text in Inkscape, but the package 'transparent.sty' is not loaded}%
    \renewcommand\transparent[1]{}%
  }%
  \providecommand\rotatebox[2]{#2}%
  \newcommand*\fsize{\dimexpr\f@size pt\relax}%
  \newcommand*\lineheight[1]{\fontsize{\fsize}{#1\fsize}\selectfont}%
  \ifx\svgwidth\undefined%
    \setlength{\unitlength}{234.05726624bp}%
    \ifx\svgscale\undefined%
      \relax%
    \else%
      \setlength{\unitlength}{\unitlength * \real{\svgscale}}%
    \fi%
  \else%
    \setlength{\unitlength}{\svgwidth}%
  \fi%
  \global\let\svgwidth\undefined%
  \global\let\svgscale\undefined%
  \makeatother%
  \begin{picture}(1,0.54499078)%
    \lineheight{1}%
    \setlength\tabcolsep{0pt}%
    \put(0,0){\includegraphics[width=\unitlength,page=1]{an-pos-neg-sing.pdf}}%
  \end{picture}%
\endgroup%

%% file: figures/weinstein-isotopy-proof.pdf_tex
\begingroup%
  \makeatletter%
  \providecommand\color[2][]{%
    \errmessage{(Inkscape) Color is used for the text in Inkscape, but the package 'color.sty' is not loaded}%
    \renewcommand\color[2][]{}%
  }%
  \providecommand\transparent[1]{%
    \errmessage{(Inkscape) Transparency is used (non-zero) for the text in Inkscape, but the package 'transparent.sty' is not loaded}%
    \renewcommand\transparent[1]{}%
  }%
  \providecommand\rotatebox[2]{#2}%
  \newcommand*\fsize{\dimexpr\f@size pt\relax}%
  \newcommand*\lineheight[1]{\fontsize{\fsize}{#1\fsize}\selectfont}%
  \ifx\svgwidth\undefined%
    \setlength{\unitlength}{320.74386597bp}%
    \ifx\svgscale\undefined%
      \relax%
    \else%
      \setlength{\unitlength}{\unitlength * \real{\svgscale}}%
    \fi%
  \else%
    \setlength{\unitlength}{\svgwidth}%
  \fi%
  \global\let\svgwidth\undefined%
  \global\let\svgscale\undefined%
  \makeatother%
  \begin{picture}(1,0.47213385)%
    \lineheight{1}%
    \setlength\tabcolsep{0pt}%
    \put(0,0){\includegraphics[width=\unitlength,page=1]{weinstein-isotopy-proof.pdf}}%
  \end{picture}%
\endgroup%

%% file: figures/source-mutation-example.pdf_tex
\begingroup%
  \makeatletter%
  \providecommand\color[2][]{%
    \errmessage{(Inkscape) Color is used for the text in Inkscape, but the package 'color.sty' is not loaded}%
    \renewcommand\color[2][]{}%
  }%
  \providecommand\transparent[1]{%
    \errmessage{(Inkscape) Transparency is used (non-zero) for the text in Inkscape, but the package 'transparent.sty' is not loaded}%
    \renewcommand\transparent[1]{}%
  }%
  \providecommand\rotatebox[2]{#2}%
  \newcommand*\fsize{\dimexpr\f@size pt\relax}%
  \newcommand*\lineheight[1]{\fontsize{\fsize}{#1\fsize}\selectfont}%
  \ifx\svgwidth\undefined%
    \setlength{\unitlength}{337.91008416bp}%
    \ifx\svgscale\undefined%
      \relax%
    \else%
      \setlength{\unitlength}{\unitlength * \real{\svgscale}}%
    \fi%
  \else%
    \setlength{\unitlength}{\svgwidth}%
  \fi%
  \global\let\svgwidth\undefined%
  \global\let\svgscale\undefined%
  \makeatother%
  \begin{picture}(1,0.41257531)%
    \lineheight{1}%
    \setlength\tabcolsep{0pt}%
    \put(0,0){\includegraphics[width=\unitlength,page=1]{source-mutation-example.pdf}}%
    \put(0.10029907,0.00543782){\color[rgb]{0,0,0}\makebox(0,0)[lt]{\lineheight{1.25}\smash{\begin{tabular}[t]{l}$\sigma$\end{tabular}}}}%
    \put(0.47779328,0.00543782){\color[rgb]{0,0,0}\makebox(0,0)[lt]{\lineheight{1.25}\smash{\begin{tabular}[t]{l}$\sigma[1]$\end{tabular}}}}%
    \put(0.85528703,0.00543782){\color[rgb]{0,0,0}\makebox(0,0)[lt]{\lineheight{1.25}\smash{\begin{tabular}[t]{l}$\sigma[2]$\end{tabular}}}}%
    \put(0,0){\includegraphics[width=\unitlength,page=2]{source-mutation-example.pdf}}%
  \end{picture}%
\endgroup%

%% file: figures/an-no-relations.pdf_tex
\begingroup%
  \makeatletter%
  \providecommand\color[2][]{%
    \errmessage{(Inkscape) Color is used for the text in Inkscape, but the package 'color.sty' is not loaded}%
    \renewcommand\color[2][]{}%
  }%
  \providecommand\transparent[1]{%
    \errmessage{(Inkscape) Transparency is used (non-zero) for the text in Inkscape, but the package 'transparent.sty' is not loaded}%
    \renewcommand\transparent[1]{}%
  }%
  \providecommand\rotatebox[2]{#2}%
  \newcommand*\fsize{\dimexpr\f@size pt\relax}%
  \newcommand*\lineheight[1]{\fontsize{\fsize}{#1\fsize}\selectfont}%
  \ifx\svgwidth\undefined%
    \setlength{\unitlength}{234.05853271bp}%
    \ifx\svgscale\undefined%
      \relax%
    \else%
      \setlength{\unitlength}{\unitlength * \real{\svgscale}}%
    \fi%
  \else%
    \setlength{\unitlength}{\svgwidth}%
  \fi%
  \global\let\svgwidth\undefined%
  \global\let\svgscale\undefined%
  \makeatother%
  \begin{picture}(1,0.72665383)%
    \lineheight{1}%
    \setlength\tabcolsep{0pt}%
    \put(0,0){\includegraphics[width=\unitlength,page=1]{an-no-relations.pdf}}%
  \end{picture}%
\endgroup%

%% file: figures/an-isotopy.pdf_tex
\begingroup%
  \makeatletter%
  \providecommand\color[2][]{%
    \errmessage{(Inkscape) Color is used for the text in Inkscape, but the package 'color.sty' is not loaded}%
    \renewcommand\color[2][]{}%
  }%
  \providecommand\transparent[1]{%
    \errmessage{(Inkscape) Transparency is used (non-zero) for the text in Inkscape, but the package 'transparent.sty' is not loaded}%
    \renewcommand\transparent[1]{}%
  }%
  \providecommand\rotatebox[2]{#2}%
  \newcommand*\fsize{\dimexpr\f@size pt\relax}%
  \newcommand*\lineheight[1]{\fontsize{\fsize}{#1\fsize}\selectfont}%
  \ifx\svgwidth\undefined%
    \setlength{\unitlength}{401.09811401bp}%
    \ifx\svgscale\undefined%
      \relax%
    \else%
      \setlength{\unitlength}{\unitlength * \real{\svgscale}}%
    \fi%
  \else%
    \setlength{\unitlength}{\svgwidth}%
  \fi%
  \global\let\svgwidth\undefined%
  \global\let\svgscale\undefined%
  \makeatother%
  \begin{picture}(1,0.72933627)%
    \lineheight{1}%
    \setlength\tabcolsep{0pt}%
    \put(0,0){\includegraphics[width=\unitlength,page=1]{an-isotopy.pdf}}%
  \end{picture}%
\endgroup%

%% file: figures/an-plus-one-construction.pdf_tex
\begingroup%
  \makeatletter%
  \providecommand\color[2][]{%
    \errmessage{(Inkscape) Color is used for the text in Inkscape, but the package 'color.sty' is not loaded}%
    \renewcommand\color[2][]{}%
  }%
  \providecommand\transparent[1]{%
    \errmessage{(Inkscape) Transparency is used (non-zero) for the text in Inkscape, but the package 'transparent.sty' is not loaded}%
    \renewcommand\transparent[1]{}%
  }%
  \providecommand\rotatebox[2]{#2}%
  \newcommand*\fsize{\dimexpr\f@size pt\relax}%
  \newcommand*\lineheight[1]{\fontsize{\fsize}{#1\fsize}\selectfont}%
  \ifx\svgwidth\undefined%
    \setlength{\unitlength}{130.51119827bp}%
    \ifx\svgscale\undefined%
      \relax%
    \else%
      \setlength{\unitlength}{\unitlength * \real{\svgscale}}%
    \fi%
  \else%
    \setlength{\unitlength}{\svgwidth}%
  \fi%
  \global\let\svgwidth\undefined%
  \global\let\svgscale\undefined%
  \makeatother%
  \begin{picture}(1,0.49173963)%
    \lineheight{1}%
    \setlength\tabcolsep{0pt}%
    \put(0,0){\includegraphics[width=\unitlength,page=1]{an-plus-one-construction.pdf}}%
  \end{picture}%
\endgroup%

%% file: figures/vanishing-example.pdf_tex
\begingroup%
  \makeatletter%
  \providecommand\color[2][]{%
    \errmessage{(Inkscape) Color is used for the text in Inkscape, but the package 'color.sty' is not loaded}%
    \renewcommand\color[2][]{}%
  }%
  \providecommand\transparent[1]{%
    \errmessage{(Inkscape) Transparency is used (non-zero) for the text in Inkscape, but the package 'transparent.sty' is not loaded}%
    \renewcommand\transparent[1]{}%
  }%
  \providecommand\rotatebox[2]{#2}%
  \newcommand*\fsize{\dimexpr\f@size pt\relax}%
  \newcommand*\lineheight[1]{\fontsize{\fsize}{#1\fsize}\selectfont}%
  \ifx\svgwidth\undefined%
    \setlength{\unitlength}{235.33430685bp}%
    \ifx\svgscale\undefined%
      \relax%
    \else%
      \setlength{\unitlength}{\unitlength * \real{\svgscale}}%
    \fi%
  \else%
    \setlength{\unitlength}{\svgwidth}%
  \fi%
  \global\let\svgwidth\undefined%
  \global\let\svgscale\undefined%
  \makeatother%
  \begin{picture}(1,0.54372459)%
    \lineheight{1}%
    \setlength\tabcolsep{0pt}%
    \put(0,0){\includegraphics[width=\unitlength,page=1]{vanishing-example.pdf}}%
  \end{picture}%
\endgroup%

%% file: figures/stabilized-theta.pdf_tex
\begingroup%
  \makeatletter%
  \providecommand\color[2][]{%
    \errmessage{(Inkscape) Color is used for the text in Inkscape, but the package 'color.sty' is not loaded}%
    \renewcommand\color[2][]{}%
  }%
  \providecommand\transparent[1]{%
    \errmessage{(Inkscape) Transparency is used (non-zero) for the text in Inkscape, but the package 'transparent.sty' is not loaded}%
    \renewcommand\transparent[1]{}%
  }%
  \providecommand\rotatebox[2]{#2}%
  \newcommand*\fsize{\dimexpr\f@size pt\relax}%
  \newcommand*\lineheight[1]{\fontsize{\fsize}{#1\fsize}\selectfont}%
  \ifx\svgwidth\undefined%
    \setlength{\unitlength}{476.47847814bp}%
    \ifx\svgscale\undefined%
      \relax%
    \else%
      \setlength{\unitlength}{\unitlength * \real{\svgscale}}%
    \fi%
  \else%
    \setlength{\unitlength}{\svgwidth}%
  \fi%
  \global\let\svgwidth\undefined%
  \global\let\svgscale\undefined%
  \makeatother%
  \begin{picture}(1,0.36221077)%
    \lineheight{1}%
    \setlength\tabcolsep{0pt}%
    \put(0,0){\includegraphics[width=\unitlength,page=1]{stabilized-theta.pdf}}%
  \end{picture}%
\endgroup%